\documentclass[11pt]{amsart}
\usepackage{latexsym,amssymb,mathrsfs,array,manfnt}
\usepackage[pdftex]{graphicx}
\usepackage[all]{xy}
\usepackage{paralist}
\usepackage{framed}
\usepackage{cancel}
\usepackage{enumitem}
\usepackage{units}
\usepackage{gensymb}
\usepackage{setspace}
\usepackage{amscd}
\usepackage{microtype}
\usepackage{tikz}
\usepackage{tikz-cd}
\usepackage{hyperref}
\usepackage{chngcntr}
\usepackage{multicol}

\newcounter{parentnumber}



\newcommand{\margincolor}{red}      
\definecolor{darkgreen}{rgb}{0,0.7,0}

\newcounter{margincounter}
\newcommand{\marginnum}{
\ifnum\value{margincounter}<10
\textcolor{\margincolor}{\begin{picture}(0,0)\put(2.2,2.4){\circle{9}}\end{picture}\footnotesize\arabic{margincounter}}
\else\ifnum\value{margincounter}<100
\textcolor{\margincolor}{\begin{picture}(0,0)\put(4.256,2.5){\circle{11}}\end{picture}\footnotesize\arabic{margincounter}}
\else
\textcolor{\margincolor}{\begin{picture}(0,0)\put(6.8,2.5){\circle{14}}\end{picture}\footnotesize\arabic{margincounter}}
\fi\fi
}

\newcommand{\into}{\hookrightarrow}

\newcommand{\deriv}[3][]{\frac{d^{#1} #2}{d{#3}^{#1}}}
\newcommand{\pderiv}[3][]{\frac{\partial^{#1} #2}{\partial {#3}^{#1}}}

\newcommand{\abs}[1]{\left\lvert#1\right\rvert}

\newcommand{\Z}{\ensuremath{\mathbb{Z}}}

\newcommand{\C}{\ensuremath{\mathbb{C}}}
\newcommand{\Q}{\ensuremath{\mathbb{Q}}}

\newcommand{\A}{\ensuremath{\mathbb{A}}}

\newcommand{\Bl}{\ensuremath{\mathcal{B}\ell}}

\renewcommand{\P}{\ensuremath{\mathbb{P}}}
\renewcommand{\H}{\ensuremath{\mathcal{H}}}
\renewcommand{\O}{\ensuremath{\mathcal{O}}}
\renewcommand{\L}{\ensuremath{\mathcal{L}}}

\renewcommand{\bar}[1]{\overline{#1}}

\renewcommand{\Im}{\ensuremath{\operatorname{Im}}}

\renewcommand{\wr}{\ensuremath{\operatorname{wr}}}

\renewcommand{\vec}[1]{\overrightarrow{#1}}
\renewcommand{\mod}{\mathrm{mod}}
\let\oldparallel\parallel
\renewcommand{\parallel}{\!\!\oldparallel\!\!}

\DeclareMathOperator{\Hilb}{Hilb}

\DeclareMathOperator{\Aut}{Aut}

\DeclareMathOperator{\Defm}{Def}
\DeclareMathOperator{\Obs}{Obs}

\DeclareMathOperator{\Spec}{Spec}

\renewcommand{\Im}{\mathop{\mathrm{Im}}}

\usepackage{color}

\newcommand{\cM}{\mathcal{M}}

\theoremstyle{plain}
\newtheorem{theorem}{Theorem}
\numberwithin{theorem}{section}

\newtheorem{thm}[theorem]{Theorem}

\newtheorem{prop}[theorem]{Proposition}

\newtheorem{cor}[theorem]{Corollary}

\newtheorem{lem}[theorem]{Lemma}

\newtheorem{identity}[theorem]{Identity}

\theoremstyle{definition}
\newtheorem{Definition/Theorem}[theorem]{Definition/Theorem}
\newtheorem{Definition/Proposition}[theorem]{Definition/Proposition}

\newtheorem{Def}[theorem]{Definition}

\newtheorem{Corollary/Definition}[theorem]{Corollary/Definition}

\theoremstyle{remark}
\newtheorem{cl}[theorem]{Claim}
\newtheorem{rem}[theorem]{Remark}

\newtheorem{notation}[theorem]{Notation}

\usepackage{fullpage}
\usepackage{longtable}

\newcommand{\Mbar}{\bar{\cM}}

\renewcommand{\setminus}{\smallsetminus}

\DeclareMathOperator{\rig}{rig}

\DeclareMathOperator{\tw}{tw}
\DeclareMathOperator{\untw}{untw}

\DeclareMathOperator{\loc}{loc}
\DeclareMathOperator{\vir}{vir}
\DeclareMathOperator{\fix}{fix}

\DeclareMathOperator{\mov}{mov}

\DeclareMathOperator{\ch}{ch}

\DeclareMathOperator{\Sym}{Sym}

\DeclareMathOperator{\lcm}{lcm}

\DeclareMathOperator{\pr}{pr}
\DeclareMathOperator{\val}{val}
\DeclareMathOperator{\Mark}{Mark}
\DeclareMathOperator{\Mon}{Mon}
\DeclareMathOperator{\Mov}{Mov}
\DeclareMathOperator{\VEval}{VEval}
\DeclareMathOperator{\ZPart}{ZPart}
\DeclareMathOperator{\Part}{Part}
\DeclareMathOperator{\MultiPart}{MultiPart}
\DeclareMathOperator{\Comb}{Comb}
\DeclareMathOperator{\Laur}{Coef}
\DeclareMathOperator{\nov}{nov}
\DeclareMathOperator{\RC}{\mathbf{RC}}
\DeclareMathOperator{\ev}{ev}
\DeclareMathOperator{\Graphs}{Graphs}
\DeclareMathOperator{\Contr}{Contr}
\DeclareMathOperator{\Mult}{Mult}
\DeclareMathOperator{\Stat}{Stat}
\DeclareMathOperator{\noneq}{ne}

\renewcommand{\sf}{\mathrm{sf}}

\usetikzlibrary{calc}

\author{Rob Silversmith} \address{Warwick Mathematics Institute, University of Warwick, Coventry, UK, CV4 7AL}
\email{Rob.Silversmith@warwick.ac.uk} \date{} \title{Gromov-Witten invariants of
  $\Sym^d\P^r$}


\begin{document}
\begin{abstract}
  We give a graph-sum algorithm that expresses any genus-$g$ Gromov-Witten invariant of the symmetric product orbifold $\Sym^d\P^r:=[(\P^r)^d/S_d]$ in terms of ``Hurwitz-Hodge integrals'' -- integrals over (compactified) Hurwitz spaces. We apply the algorithm to prove a mirror-type theorem for $\Sym^d\P^r$ in genus zero. The theorem states that a generating function of Gromov-Witten invariants of $\Sym^d\P^r$ is equal to an explicit power series $I_{\Sym^d\P^r}$, conditional upon a conjectural combinatorial identity. This is a first step in the direction of proving Ruan's Crepant Resolution Conjecture for the resolution $\Hilb^{(d)}(\P^2)$ of the coarse moduli space of $\Sym^d\P^2.$
\end{abstract}
\maketitle
\section{Introduction}
Over the last 20 years, following predictions from string theory
\cite{CandelasDeLaOssaGreenParkes1991}, mathematicians have proven a
series of results known as \emph{mirror theorems}; an incomplete list
is
\cite{Givental1998quintic,LianLiuYau1997iii,Givental1998toric,BatyrevCiocanFontanineKimVanStraten2000,Zinger2009,Li2011,JarvisKimura2002,CoatesCortiIritaniTseng2015,CheongCiocanFontanineKim2015,FangLiuZong2013,FangLiuZong2020,CoatesCortiIritaniTseng2014}. These
theorems reveal elegant patterns and structures embedded in the
collection of (usually genus-zero) Gromov-Witten invariants of a fixed
target manifold or orbifold $X$. They also allow for easy computation
of these invariants in certain cases where direct computation involves
difficult combinatorial computations. However, the scope of these
results, and much of Gromov-Witten theory in general, is limited to
the world of toric geometry; in all cases above, $X$ is a complete
intersection in a toric variety or stack (or a deformation
thereof). The essential reason for this is that computing a
Gromov-Witten invariant of a toric variety can be reduced, via the
Atiyah-Bott localization theorem, to evaluating a certain sum over
labeled graphs.

In this paper, we study the Gromov-Witten invariants of $\Sym^d\P^r$,
which has a torus action, but without a dense orbit. Some aspects of
the theory remain similar to the toric case, many new obstacles must
be dealt with, and some interesting new behaviors appear. In the first
half of the paper we use localization to give an algorithm expressing any Gromov-Witten
invariant of $\Sym^d\P^r$ explicitly in terms of \textit{Hurwitz-Hodge
  integrals} (Theorem \ref{thm:Algorithm}). Hurwitz-Hodge integrals
are numerical invariants of a representation of a finite group $G$;
they are defined as integrals over compactified Hurwitz
spaces. Computing them in general is a main stumbling block in
orbifold Gromov-Witten theory.

In order to apply localization to the case of $\Sym^d\P^r$, we must
carefully describe the torus-invariant curves on $\Sym^d\P^r$ and
their deformation theory. We do this in Sections \ref{sec:TAction} and
\ref{sec:VirtualNormalBundle}. (These sections contain the main
geometric content of the paper.)


In the second half of the paper, we apply the above
algorithm in a recursive form (Theorem \ref{thm:ConeCharacterization}) to prove a 
genus-zero mirror-type theorem for $\Sym^d\P^r$ (Theorem
\ref{thm:MirrorTheorem}), which was not possible using existing
techniques. The theorem, which is conditional upon two explicit combinatorial identities we were unable to prove, gives a formula for a generating
function of Gromov-Witten invariants of $\Sym^d\P^r$. The proof of Theorem \ref{thm:MirrorTheorem} is notably combinatorial, and the specific combinatorics are of independent interest, see Remark \ref{rem:Combinatorics}. Theorem
\ref{thm:MirrorTheorem} is also the
only known mirror theorem for a nonabelian orbifold, besides single
points $[\bullet/G]$. 


 \medskip
\noindent\textbf{Corollary \ref{cor:IandJ}}
Assuming Identities \ref{lem:WhatIsTauIdentity} and \ref{lem:WhatIsTauSigma}, for any $d,r\ge1$ there is an
equality 
$$I_{\Sym^d\P^r}=J_{\Sym^d\P^r}\quad\mod(\mathbf{x})^2,$$ where
 $J_{\Sym^d\P^r}$ is a generating function of genus-zero Gromov-Witten
  invariants of $\Sym^d\P^r$ (see Section \ref{sec:OrbifoldGWTheory}),
  and 
$I_{\Sym^d\P^r}$ is the explicit power series \eqref{eq:IDef}.

\medskip

\begin{rem}
In Theorem \ref{thm:MirrorTheorem} and Corollary \ref{cor:IandJ}, $I_{\Sym^d\P^r}$ is only defined up to first
  order in $\mathbf{x}$ --- it would be very desirable to generalize this mirror theorem so that it
  involves a power series $I'$ with arbitrary powers of $\mathbf{x}$.
The primary obstacle is that one must first
  \textit{produce} such a power series --- and then check
  that it satisfies the conditions of Theorem \ref{thm:ConeCharacterization}. The power
  series \eqref{eq:IDef} was produced after much computer experimentation, and
  we were unable to generalize it to arbitrary order in $\mathbf{x}$. Furthermore, the combinatorics required to prove that  $I_{\Sym^d\P^r}$ satisfies the conditions of Theorem \ref{thm:ConeCharacterization} are extremely complicated, and we were only able to establish them conditional upon the conjectural combinatorial identities in Section \ref{sec:Appendix}. While there are some systematic methods for producing such ``$I$-functions'' (e.g.
  \cite{CiocanFontanineKim2013Big}), applying these methods to
  $\Sym^d\P^r$ (or any nonabelian orbifold) results in the
  \textit{zeroth} order truncation of $I_{\Sym^d\P^r}$ in
  $\mathbf{x}$, losing \emph{all} combinatorial structure.
\end{rem}

We have three motivations for working with $\Sym^d\P^r$.
\begin{itemize}
\item $\Sym^d\P^r$ is very concrete, and is therefore a good starting
  point for studying both non-toric and non-abelian behavior. While
  the natural $(\C^*)^{r+1}$-action has infinitely many orbits, it
  also has finitely many fixed points; in this sense $\Sym^d\P^r$ is
  not too much more complicated than a toric variety. On the other
  hand, it is complicated enough that studying its Gromov-Witten
  invariants requires various new methods, which we expect to be
  useful for studying the Gromov-Witten theory of other non-toric and
  non-abelian
\item \textit{The crepant resolution conjecture.} Following physical
  predictions, Ruan \cite{Ruan2006}, Bryan-Graber
  \cite{BryanGraber2009}, and Coates-Iritani-Tseng
  \cite{CoatesIritaniTseng2009} made a conjecture relating the
  Gromov-Witten invariants of an orbifold $X$ to those of a crepant
  resolution of its coarse moduli space. This conjecture has been
  proven in the context of toric geometry
  \cite{CoatesIritaniJiang2014}. However, the crepant resolution
  $\Hilb^{(d)}(\P^2)$ of the coarse moduli space of $\Sym^d\P^r$ was
  one of Ruan's motivating examples; this case has now been open for
  over a decade. Theorem \ref{thm:MirrorTheorem} is a first step
  towards this case.
\item \textit{Higher genus invariants of projective space.} Costello's
  thesis expressed the genus $g$ Gromov-Witten invariants of a smooth
  projective variety $X$ in terms of the genus-zero Gromov-Witten
  invariants of $\Sym^{g+1}X.$ Theorem \ref{thm:MirrorTheorem}
  provides an efficient way of encoding the latter for $X=\P^r.$ It
  may be possible to combine Costello's result with ours to find
  explicit formulas for genus-$g$ Gromov-Witten invariants of
  $\P^r$.
\end{itemize}

%
%

We briefly describe the difficulties caused by the fact that
$\Sym^d\P^r$ is not toric. To do so, we first broadly outline the
proof of Coates-Corti-Iritani-Tseng of the mirror theorem for a toric
stack $X$ \cite{CoatesCortiIritaniTseng2015}. The two main ingredients
are
\begin{enumerate}
\item An algorithm for expressing Gromov-Witten invariants of $X$ in
  terms of Hurwitz-Hodge integrals; this is supplied by localization
  calculations of Johnson \cite{Johnson2014} and Liu
  \cite{Liu2013}. The localization technique roughly involves
  integrating over the moduli space of \textit{torus-invariant} curves
  $C\subseteq X$, which is easy: this moduli space is a finite
  collection of points, in bijection with codimension-1 cones in the
  fan of $X$. The hardest part of the calculation is to find an
  explicit expression for the integrand, which is defined in terms of
  the deformation theory of the curves $C$.
\item A technique of Brown \cite{Brown2014}, which reinterprets the
  above algorithm as follows. To each torus-fixed point $\sigma\in X$
  is associated a power series $\mathbf{f}_\sigma$ in a variable $z$;
  these power series together encode all genus-zero Gromov-Witten
  invariants of $X$. Each power series $\mathbf{f}_\sigma$ has a
  collection of simple poles, and using the algorithm, one shows
  that the power series satisfy a recursion in the following sense:
  the residue of $\mathbf{f}_\sigma$ at a pole $w$ is expressed as a
  linear combination of values $\mathbf{f}_{\sigma'}(w)$ for other
  fixed points $\sigma'\ne\sigma$ (such that $\mathbf{f}_{\sigma'}$
  has no pole at $w$). The recursion uniquely defines
  $\mathbf{f}_\sigma$ for all $\sigma,$ up to some change of
  variables.
\end{enumerate}
The outline of the proof of Theorem \ref{thm:MirrorTheorem} is
similar, but with the following differences:
\begin{enumerate}[label=(\arabic*$'$)]
\item As mentioned, Theorem \ref{thm:Algorithm} expresses any
  Gromov-Witten invariant of $\Sym^d\P^r$ in terms of Hurwitz-Hodge
  integrals. However, both parts of the calculation are substantially
  more difficult than in the toric case. The moduli space of
  torus-invariant curves is not finite --- rather, it is positive
  dimensional, disconnected, and quite complicated. Luckily, we are
  able to give a complete characterization of the moduli space
  (Theorem \ref{thm:FixedLocusLosevManin}). Our characterization is
  concrete enough to allow us to compute the requisite integrals. The
  deformation theory of torus-fixed curves is difficult for
  essentially the same reason, but again the computation can be
  carried out fully (Section \ref{sec:VirtualNormalBundle}).
\item Theorem \ref{thm:ConeCharacterization} is analogous to Brown's
  description above --- we again have a power series
  $\mathbf{f}_\sigma$ attached to each torus-fixed point
  $\sigma\in\Sym^d\P^r.$ However, these power series no longer have
  simple poles, but may have poles of arbitrarily high order. The
  algorithm again gives a recursion relation, this time expressing
  \textit{any negative-power} Laurent coefficient of
  $\mathbf{f}_\sigma$ in terms of \textit{nonnegative-power} Laurent
  coefficients of $\mathbf{f}_{\sigma'}$ for other fixed points
  $\sigma'.$ We wish to highlight this feature, both because it is
  new, and because it is expected to appear in the Gromov-Witten
  theory of any nontoric variety with a nontrivial torus action. (The
  fact that there are only simple poles in the toric case should be
  viewed as exceptional.) We hope that this first example might
  provide clues for proving other nontoric mirror theorems.
  \end{enumerate}
  
  \begin{rem}\label{rem:Combinatorics}We also wish to draw attention to the fact that the combinatorial structure encoded in Theorems \ref{thm:ConeCharacterization} and (especially) \ref{thm:MirrorTheorem} is much more intricate than in the toric case --- so much so that we were not able to give an unconditional version of Theorem \ref{thm:MirrorTheorem} despite the apparent fact that combinatorial complexity is the only hurdle --- for example, the Chu-Vandermonde identity played a crucial role in the proof of Theorem \ref{thm:MirrorTheorem}. We hope that the combinatorics in this paper, though not quite complete, will be a useful case study in proving mirror theorems where high-order poles appear. The generating functions in this paper exhibit rich combinatorial structure, and are surely important for further understanding mirror symmetry for symmetric products, so we believe a more systematic study is worthwhile in the future. This is especially true of the generating functions appearing on pages \pageref{eq:GDef}--\pageref{eq:MoreGeneralF}, which are not specific to $\Sym^d\P^r$ but instead deal with twisted Gromov-Witten invariants of an orbifold \emph{point}. (We note that some of the relevant framework may already exist, e.g. in the integrable systems literature --- though we were unable to find anything that would imply Identities \ref{lem:WhatIsTauIdentity} and \ref{lem:WhatIsTauSigma}. The specific form of these identities, and the other combinatorial tools used in the proof of Theorem \ref{thm:MirrorTheorem}, are quite unlike anything appearing in the Gromov-Witten theory to our knowledge.)
\end{rem}

\subsection{Acknowledgements} This work is based on my Ph.D. thesis. I
would like to thank my Ph.D. advisor, Yongbin Ruan, for introducing me
to the area, and for many useful conversations. I am grateful to David
Speyer, Chiu-Chu Melissa Liu, Hsian-Hua Tseng, and Karen Smith for
reading versions of this paper, and helping me to improve it. I am particularly grateful to Hsian-Hua Tseng for pointing out gaps in the proof given in the original preprint.

This research was supported in part by NSF grants EMSW21-RTG 1045119
and EMSW21-RTG 0943832, and by the NSF GRFP.

\section{Notation, conventions, and background}\label{sec:Background}
This section sets up combinatorial conventions, and reviews
Atiyah-Bott torus localization, orbifold Gromov-Witten theory, and
moduli spaces of curves called \emph{Losev-Manin spaces}, which are
used in Section \ref{sec:FixedLocusLosevManin} to describe the torus
invariant curves in $\Sym^d\P^r$.

We always work over $\C$. We write $H^*(X):=H^*(X,\Q).$ For a point
$x$ of an orbifold $X$, we write $G_x$ for the isotropy group of $x.$

\subsection{Multipartitions and graphs}\label{sec:Combinatorics}
It is convenient to use the language of multisets, denoted with parentheses,
e.g. $(a,a,b)$. 
We write $\Mult(\Pi,a)$ for the number of times that $a$ appears in
$\Pi.$ We will refer to multiset unions and intersections, and sums
indexed by multisets, without comment. 

For an integer $d\ge0,$ $\Part(d)$ is the set of partitions of $d$,
i.e. the set of multisets of positive integers that sum to
$d$. 
A \emph{weak composition of $d$} is an ordered tuple of nonnegative
integers whose sum is $d$. The (finite) set of weak compositions of
$d$ of length $r$ is denoted $\ZPart(d,r).$ If $D\in\ZPart(d,r)$, a
\emph{multipartition of $D$} is a multiset $(\Pi_d)_{d\in D}$, with
$\Pi_d$ a partition of $d$. The (finite) set of multipartitions of $D$
is denoted $\MultiPart(D)$. For each partition $D\in\ZPart(d,r)$, there is a ``trivial multipartition'' of $D$, which we usually denote (abusing notation) by $(1,\ldots,1)$, where every part of every $\Pi_d$ is equal to 1. 
There is an ``underlying partition'' map
$\MultiPart(D)\to\Part(\sum_{d\in
  D}d)$. 

If $\Pi$ is a partition, we write $S_\Pi$ for the group of
automorphisms of $\Pi$ as a multiset (defined up to isomorphism);
e.g. for $\Pi=(1,1,1,2,2)$ of $7,$ we have $S_\Pi=S_3\times S_2.$ For
$\sigma=(\Pi_d)_{d\in D}$ a multipartition of $D\in\ZPart(d,r)$, we
define $S_\sigma:=\prod_{d\in D}S_{\Pi_d}.$


\medskip

Let $\Gamma=(V(\Gamma),E(\Gamma))$ be a finite graph. We denote by
$E(\Gamma,v)$ the set of edges incident to $v$. The \emph{valence}
$\val(v)$ of $v\in V(\Gamma)$ is $\abs{E(\Gamma,v)}$. (This is
different from some Gromov-Witten theory literature, where $\val(v)$
includes contributions from certain decorations on $\Gamma$.) A
\emph{flag} of $\Gamma$ is a pair $(v,e)\in V(\Gamma)\times E(\Gamma)$
with $e\in E(\Gamma,v)$. The set of flags of $\Gamma$ is denoted
$F(\Gamma)$.

\subsection{Equivariant cohomology} We will consider actions of the
torus $T:=(\C^*)^{r+1}$ on various spaces, e.g. $\P^r$, $\Sym^d\P^r,$
and $\Mbar_{g,n}(\Sym^d\P^r,\beta)$. If $T$ acts on a Deligne-Mumford
stack $X$, the equivariant cohomology $H^*_T(X)$ is a module over
$H^*_T(\Spec\C)\cong\Q[\alpha_0,\ldots,\alpha_r],$ where $-\alpha_i$
is the weight of the character $T\to\C^*$ defined by
$(\lambda_0,\ldots,\lambda_r)\mapsto\lambda_i$. We write
$H^*_{T,\loc}(\Spec\C)$ for the localization
$\Q(\alpha_0,\ldots,\alpha_r)$, and more generally
$H^*_{T,\loc}(X):=H^*_T(X)\otimes_{H^*_{T}(\Spec\C)}H^*_{T,\loc}(\Spec\C)$. We
will use the Atiyah-Bott \emph{localization theorem}, as well as
Graber-Pandharipande's generalization, the \emph{virtual localization
  theorem}.
\begin{thm}[\cite{AtiyahBott1984}, see \cite{EdidinGraham1998} for
  statement in the Chow ring]\label{thm:Localization}
  Let $T$ be a torus acting on a smooth compact manifold $X$, with
  fixed point set $F$. Then the map
  $(\iota_F)_*:H^*_{T,\loc}(F)\to H^*_{T,\loc}(X)$ is an isomorphism,
  where $(\iota_F)_*$ is the Gysin map associated to the inclusion
  $F\into X$. The inverse map is $\iota_F^*/e_T(N_{F|X}),$ where
  $e_T(N_{F})$ is the equivariant Euler class of the normal bundle to
  $F$. In particular, for $\alpha\in H^*_{T,\loc}(X,\Spec\C),$ we have
  \begin{align*}
    \int_X\alpha=\int_X(\iota_F)_*\left(\frac{\iota_F^*\alpha}{e_T(N_{F})}\right)=\int_F\frac{\iota_F^*\alpha}{e_T(N_{F})}.
  \end{align*}
\end{thm}
\begin{thm}[\cite{GraberPandharipande1999}]\label{thm:VirtualLocalization}
  Let $X$ be a Deligne-Mumford stack with a $T$-action and a
  $T$-equivariant perfect obstruction theory $E^{\bullet}$. Again, let
  $\iota_F:F\into X$ denote the inclusion of the fixed locus. Let
  $[X]^{\vir}$ denote the virtual fundamental class associated to
  $E^\bullet$. The $T$-fixed part of $E^{\bullet}$ defines a perfect
  obstruction theory on $F$, with virtual fundamental class
  $[F]^{\vir}$. The \emph{virtual normal bundle} $N_{F}^{\vir}$ to $F$
  is the $T$-moving part of $E^{\bullet}$. Then
  \begin{align}\label{eq:VirtualIntegralFormula}
    \int_{[X]^{\vir}}\alpha=\int_{[F]^{\vir}}\frac{\iota_F^*\alpha}{e_T(N_{F}^{\vir})}.
  \end{align}
\end{thm}
\begin{rem}
  The proof in \cite{GraberPandharipande1999} requires that $X$ have a
  global equivariant embedding into a smooth Deligne-Mumford stack,
  but this condition was removed in \cite{ChangKiemLi2015}.
\end{rem}

\subsection{Symmetric product stacks}\label{sec:SymmetricPowers} Let $X$ be a scheme over
$\C$. There are two common (equivalent) definitions of $\Sym^dX.$ The
first is the stack quotient $[X^d/S_d]$, where $S_d$ acts in the usual
way on $X^d$. That is, objects and morphisms are described by
\begin{center}
  Objects:
  \begin{tikzcd}[row sep=scriptsize, column sep=scriptsize]
    \tilde{S}\arrow[r,"\tilde{f}"]\arrow[d,"\pr"]& X^d\\
    S
  \end{tikzcd}
  \quad\quad\quad\quad
  Arrows:
  \begin{tikzcd}[row sep=scriptsize, column sep=scriptsize]
    \tilde{S}\arrow[r]\arrow[d]\arrow[bend left]{rr}{\tilde{f}}&\tilde{T}\arrow[swap]{r}{\tilde{g}}\arrow{d}{}& X^d\\
    S\arrow{r}{}&T
  \end{tikzcd}
\end{center}
where vertical maps are $S_d$-principal bundles, $\tilde{f}$ and
$\tilde{g}$ are $S_d$-equivariant, and the square on the right is
Cartesian. The second definition is given by
\begin{center}
  Objects:
  \begin{tikzcd}[row sep=scriptsize, column sep=scriptsize]
    S'\arrow[r,"f'"]\arrow[d,"\rho"]& X\\
    S
  \end{tikzcd}
  \quad\quad\quad\quad
  Arrows:
  \begin{tikzcd}[row sep=scriptsize, column sep=scriptsize]
    S'\arrow{r}{}\arrow{d}{}\arrow[bend left]{rr}{f'}&T'\arrow[swap]{r}{g'}\arrow{d}{}& X\\
    S\arrow{r}{}&T
  \end{tikzcd}
\end{center}
where vertical maps are degree $d$ \'etale, and the square on the
right is Cartesian. It is a straightforward exercise to show that the
two stacks defined are naturally isomorphic. We will usually use the
second, and we will consistently use the notations $S'\to S$ and
$f':S'\to X$ when referring to $S$-points of $\Sym^dX.$ The two
descriptions are related by the diagram:
\begin{align}\label{eq:SymDiagram}
  \begin{tikzcd}[row sep=scriptsize, column sep=scriptsize,ampersand replacement=\&]
    \tilde{S}\times\{1,\ldots,d\} \arrow[dr]
    \arrow[rr] \arrow[dd] \& \& X^d\times\{1,\ldots,d\}
    \arrow[dr,"\tilde{\rho}"]
    \arrow[dd,"\pr'" near
    start]\arrow[rrdd, bend left] \&\&\\
    \&\tilde{S} \arrow[rr, crossing over, "\tilde{f}" near start]  \& \& X^d \\
    S'=\tilde{S}\times_{S_d}\{1,\ldots,d\} \arrow[dr, "\rho"]
    \arrow[rr] \& \&
    X^d\times_{S_d}\{1,\ldots,d\} \arrow[dr, "\rho"] \arrow[rr,"P"
    near start] \&\& X\\
    \&S \arrow[rr, "f"]\arrow[from=uu, crossing over, "\pr" near
    start] \& \& \Sym^dX
    \arrow[from=uu,crossing over,"\pr" near start]\\
  \end{tikzcd}
\end{align}
Here the cube is Cartesian, and the left and right faces consist of
\'etale maps. The composition
$S'\to X^d\times_{S_d}\{1,\ldots,d\}\xrightarrow{P}{}X$ is $f'$.

Now assume $X$ is smooth. We can understand the tangent bundle to
$\Sym^d\P^r$ as follows:
\begin{lem}
  There is a natural isomorphism $T\Sym^dX\cong\rho_*(P^*TX),$ where
  $\rho$ and $P$ are as in the diagram above.
\end{lem}
\begin{proof}
  Since the square is cartesian and consists of \'etale maps, we have
  \begin{align*}
    \pr^*(\rho_*(P^*TX))\cong\tilde\rho_*((\pr')^*(P^*TX))=\tilde\rho_*((\pr'\circ
    P^*TX)).
  \end{align*}
  Recall that $\pr'\circ P$ is simply the ``universal coordinate
  map,'' so since $\tilde\rho$ is a trivial \'etale cover, there is a
  canonical isomorphism
  $$\tilde\rho_*((\pr'\circ
  P^*TX))\cong\bigoplus_{\ell=1}^dP_\ell^*TX\cong T(X^d).$$ Since
  $\tilde\rho$ is $S_d$-equivariant, there is an induced $S_d$-action
  on $T(X^d)$ which agrees with the usual one. Thus the isomorphism
  descends to give $\rho_*(P^*TX)\cong T\Sym^dX.$
\end{proof}

Finally, we describe the cyclotomic inertia stack
$I\Sym^dX\to\Sym^dX$, see Section 3 of
\cite{AbramovichGraberVistoli2008}. Assume $X$ is connected. For each
partition $\sigma\in\Part(d)$, there is a component $(\Sym^dX)_\sigma$
of $I\Sym^dX$, isomorphic to (a trivial gerbe over)
$\prod_{\eta\ge1}\Sym^{\Mult(\sigma,\eta)}X$, and the map
$(\Sym^dX)_\sigma\to\Sym^dX$ is (a rigidification followed by) the
obvious one. The generic point of $(\Sym^dX)_\sigma$ maps to a point
in $\Sym^dX$ with isotropy group isomorphic to
$\prod_{\eta\ge1}S_\eta$.

\begin{rem}
  The map $(\Sym^dX)_\sigma\to\Sym^dX$ (after rigidification) may not
  be an embedding. For example, consider $\Sym^4X,$ and let
  $\sigma=(2,1,1).$ By the above, $(\Sym^4X)_\sigma$ is a trivial
  gerbe over $X\times\Sym^2X.$ The induced map
  $X\times\Sym^2X\to\Sym^4X$ sends points $(a,(b,c))\mapsto(a,a,b,c),$
  but this identifies the two distinct points $(a,(b,b))$ and
  $(b,(a,a))$ for all $a,b\in X.$
\end{rem}



The (equivariant, nonorbifold) cohomology with rational coefficients
may be computed explicitly by the K\"unneth decomposition, as the
$S_d$-invariant part of $H^*_T(X^d,\Q)=\bigotimes_{j=1}^dH^*_T(X,\Q).$
In particular, for $X=\P^r,$ we will use the identification
$H^2_T(\Sym^d\P^r,\Q)\cong H^2_T((\P^r)^d,\Q)^{S_d}\cong
H^2_T(\P^r,\Q)$. We will abuse notation and write
$[H_i]\in H^2_T(\Sym^d\P^r,\Q)$ for the element that pulls back to
$\sum_{j=1}^d\pr_j^*[H_i]\in H^2_T((\P^r)^d)$, where $\pr_j$ is the
$j$th coordinate map and $[H_i]$ is the equivariant fundamental class
of the $i$th coordinate hyperplane.

Fix a component $(I\Sym^d\P^r)_\sigma$ of $I\Sym^d\P^r$. For
$\eta\in\sigma,$ we denote by $[H_{\sigma,\eta,i}]$ the pullback of
$[H_i]$ from the factor of
$(I\Sym^d\P^r)_\sigma\cong\prod_{\eta\ge1}\Sym^{\Mult(\sigma,\eta)}\P^r$
corresponding to $\eta$. We write $[H_{\sigma,i}]$ for
$\sum_{\eta}[H_{\sigma,\eta,i}].$

\subsection{(Orbifold) Gromov-Witten theory}\label{sec:OrbifoldGWTheory}
Our objects of study are the moduli spaces $\Mbar_{g,n}(X,\beta)$ of
$n$-marked genus-$g$ stable maps to a smooth proper Deligne-Mumford
stack $X$ of degree $\beta,$ introduced in \cite{ChenRuan2002} and
\cite{AbramovichVistoli2002}. See \cite{Liu2013}, Section 7 for an
introduction to the subject (in all genera). Following \cite{Liu2013},
we use the technical convention that all gerbes come with the data of
a section.

In this paper we will have either $X=\Sym^d\P^r$ or $X=BG$ for some
finite group $G$. We write $(f:C\to X)$ for a $\C$-point of
$\Mbar_{g,n}(X,\beta)$, and
\begin{center}
  \begin{tikzcd}
    \mathcal{C}\arrow{r}{f}\arrow{d}{\pi}&X\\
    \Mbar_{g,n}(X,\beta)
  \end{tikzcd}
\end{center}
for the universal curve and universal map.

A \emph{Gromov-Witten invariant} is an integral of the form
\begin{align}\label{eq:GWInvariant}
  \langle\bar{\psi}_1^{a_1}\gamma_1,\ldots,\bar{\psi}_n^{a_n}\gamma_n\rangle^X_{g,n,\beta}:=\int_{[\Mbar_{g,n}(X,\beta)]^{\vir}}\prod_{j=1}^n\bar{\psi}_j^{a_j}\ev_j^*\gamma_j\in\Q,
\end{align}
where
\begin{itemize}
\item $[\Mbar_{g,n}(X,\beta)]^{\vir}$ is the virtual fundamental
  class,
\item$\bar{\psi}_j$ is the $j$th cotangent class on
  $\Mbar_{g,n}(X,\beta)$, coming from the cotangent space to the
  \emph{coarse moduli space of $C$,}\footnote{Note that locally
    $\bar{\psi}_j=r_j\psi_j,$ where $r_j$ is the size of the isotropy
    group at the mark $b_j$, and $\psi_j$ is the ``stacky'' cotangent
    class.}
\item the ``insertions'' $\gamma_j$ are in the Chen-Ruan cohomology
  (see \cite{ChenRuan2004}) $H^*_{CR}(X),$ and
\item $\ev_j:\Mbar_{g,n}(X,\beta)\to IX$ is the $j$th evaluation map.
\end{itemize}
If $X$ has an action of a torus $T$, it induces a natural $T$-action
on $IX$ and $\Mbar_{g,n}(X,\beta)$, and
$[\Mbar_{g,n}(X,\beta)]^{\vir},$ $\bar{\psi}_j$, and $\ev_j^*\gamma_j$ are
naturally equivariant classes (where $\gamma_j\in
H^*_{CR,T}(X)$). In this case \eqref{eq:GWInvariant} defines an
\emph{equivariant Gromov-Witten invariant} (an element of
$H^*_T(\Spec\C)$, denoted by $\langle\cdots\rangle_{g,n,\beta}^{X,T}$)
via $T$-equivariant integration.

We introduce some formalism for the case $g=0,$ which will be used to
state and prove Theorems \ref{thm:ConeCharacterization} and
\ref{thm:MirrorTheorem}. Following \cite{CoatesCortiIritaniTseng2015},
the \emph{$T$-equivariant Novikov ring of $\Sym^d\P^r$} is
\begin{align*}
  \Lambda_T^{\nov}:=H^*_{T,\loc}(\Spec\C)[[Q]],
\end{align*}
and \emph{Givental's symplectic vector space} is
\begin{align*}
  \H:=H^*_{CR,T,\loc}(\Sym^d\P^r)[[Q]]((z^{-1}))=\H^+\oplus\H^-,
\end{align*}
where $\H^+=H^*_{CR,T,\loc}(\Sym^d\P^r)[[Q]][z]$ and
$\H^-=z^{-1}H^*_{CR,T,\loc}(\Sym^d\P^r)[[Q]][[z^{-1}]].$ Inside $\H,$
there is a special subscheme $\mathcal{L}_{\Sym^d\P^r}$ --- precisely, a formal germ of a subscheme over $\Spec\Lambda_{\nov}^T$, defined at
$-1\cdot z,$ where $1\in H^*_{CR,T,\loc}(\Sym^d\P^r)$ is the
fundamental class of the untwisted sector --- called the
\emph{Givental cone} of $\Sym^d\P^r,$ which encodes the genus-zero
Gromov-Witten invariants of $\Sym^d\P^r$.

Fix a basis $\gamma_\phi$
of $H^*_{CR,T,\loc}(\Sym^d\P^r)$, with Poincar\'e dual basis
$\gamma^\phi$. A $\Lambda_{\nov}^T[[x]]$-valued point
of $\mathcal{L}_{\Sym^d\P^r}$ is defined to be a power series
\begin{align*}
  -1z+\mathbf{t}(z)+\sum_{n=0}^\infty\sum_{\beta=0}^\infty\sum_\phi\frac{Q^\beta}{n!}\left\langle\mathbf{t}(\bar{\psi}),\ldots,\mathbf{t}(\bar{\psi}),\frac{\gamma_\phi}{-z-\bar{\psi}}\right\rangle_{0,n+1,\beta}^{{\Sym^d\P^r},T}\gamma^\phi\in\H[[x]],
\end{align*}
where $\mathbf{t}(z)\in\langle
Q,x\rangle\subseteq\H^+[[x]]$. 
\begin{rem}
This definition as stated is both confusing and slightly imprecise. The point is this: as a formal scheme over $\Spec\Lambda_{\nov}^T$, $\mathcal{L}_{\Sym^d\P^r}$ is characterized (indeed, defined) not just by its $\C$-valued points or $\Lambda_{\nov}^T$-valued points but by its points over arbitrary (topological) $\Lambda_{\nov}^T$-algebras. The definition given is the most basic nontrivial example, and generalizes in an obvious way. See Appendix B of \cite{CoatesCortiIritaniTseng2009} for a complete discussion.
\end{rem}
\begin{rem}\label{rem:TPowerSeries}
Another subtlety is that we may wish to take $\mathbf{t}(z)$ to be a power series in $z$, in which case it is not immediately obvious that the expression $\mathbf{t}(\bar{\psi})$ makes sense. In practice this is not a major concern; the key is that $\mathbf{t}(z)$ must be ``topologically nilpotent,'' which will always be the case in practice. Again, see Appendix B of \cite{CoatesCortiIritaniTseng2009}.
\end{rem}

 An important special case is $$\mathbf{t}(z)=\theta=\sum_{\phi}x_\phi\gamma_{\phi}\in H^*_{CR,T,\loc}(\Sym^d\P^r)[[\{x\}_{\phi}]],$$ where $\{\gamma_\phi\}$ is the basis for $H^*_{CR,T,\loc}(\Sym^d\P^r)$ chosen above.  The corresponding $\Lambda_{\nov}^T[[\{x\}_{\phi}]]$-valued point is called the $J$-function of $\Sym^d\P^r$ and is denoted $J_{\Sym^d\P^r}(Q,\theta,-z).$ Here $\mathbf{t}(z)$ has no nonzero powers of $z$, so the invariants appearing in $J_{\Sym^d\P^r}(Q,\theta,-z)$ have a single $\psi$-class.

$\mathcal{L}_{\Sym^d\P^r}$ has several
important geometric properties that follow from relations between
Gromov-Witten invariants: see Appendix B of
\cite{CoatesCortiIritaniTseng2009}, which also defines
$\mathcal{L}_{\Sym^d\P^r}$ rigorously as a non-Noetherian formal
scheme. For example, it is a cone in a certain sense, hence the name
(Proposition B.2 of \cite{CoatesCortiIritaniTseng2009}).

Given a vector bundle $E$ on $X$, there is also a notion of an
\emph{$E$-twisted Gromov-Witten invariant} of $X$.  We need this
notion only when $X=BG$, with the trivial action of a torus $T$. Let
$E$ be a $T\times G$ representation. Then
$R\pi_*f^*E\in K^0_T(\Mbar_{g,n}(BG,0)).$ An $E$-twisted Gromov-Witten
invariant of $BG$ is known as a \textit{Hurwitz-Hodge integral}, and
is defined by
\begin{align}\label{eq:TwistedInvariant}
  \langle\bar{\psi}_1^{a_1}\gamma_1,\ldots,\bar{\psi}_n^{a_n}\gamma_n\rangle_{g,n,0}^{BG,T,E}:=\int_{[\Mbar_{g,n}(BG,0)]^{\vir}}\prod_{j=1}^n\bar{\psi}_j^{a_j}\ev_j^*\gamma_j\cup
  e_T^{-1}(R\pi_*f^*E).
\end{align}

As above, in genus zero we can define the \textit{twisted Lagrangian
  cone} $\mathcal{L}_{BG}^E:$ a $\Lambda_{\nov}^T[[x]]$-valued point
of $\mathcal{L}_{BG}^{E}$ is defined to be
\begin{align}\label{eq:TwistedConeDef}
  -1z+\mathbf{t}(z)+\sum_{n=0}^\infty\sum_\phi\frac{1}{n!}\left\langle\mathbf{t}(\bar{\psi}),\ldots,\mathbf{t}(\bar{\psi}),\frac{\gamma_\phi}{-z-\bar{\psi}}\right\rangle_{0,n+1,0}^{BG,T,E}\gamma^\phi,
\end{align}
for some $\mathbf{t}(z)\in\langle Q,x\rangle\subseteq\H^+[[x]]$. Here $\gamma_\phi$ and
  $\gamma^\phi$ are dual bases of $H^*_T(X)$ under the \emph{twisted}
  Poincar\'e pairing, see \cite{CoatesCortiIritaniTseng2015}. 

\begin{notation}
  In the important case where $\mu\cong BG$ is a $T$-fixed point of an
  ambient orbifold $Y$, and $E=T_\mu Y,$ we write
  $\mathcal{L}_\mu^{\tw}:=\mathcal{L}_\mu^{T_\mu Y}.$
\end{notation}

\subsection{Losev-Manin spaces}\label{sec:LosevManin}
We recall certain moduli spaces of marked curves, studied originally
by Losev and Manin \cite{LosevManin2000}.
\begin{Def}
  Let $k\ge1$, and fix a 2-element set $\{0,\infty\}$. An
  \emph{$(0|k|\infty)$-marked Losev-Manin curve} is a connected
  genus zero $(k+2)$-marked nodal curve
  $(C,b_{0},b_1,\ldots,b_k,b_{\infty})$, satisfying:
  \begin{itemize}
  \item The irreducible components of $C$ form a chain, with two
    leaves $C_0$ and $C_\infty$,
  \item The points $b_{0},b_1,\ldots,b_k,b_{\infty}$ are smooth
    points of $C$, with $b_{0}\in C_0$ and $b_{\infty}\in C_\infty$,
  \item $b_i\ne b_0$ and $b_i\ne b_\infty$ for $i=1,\ldots,k$ (though it is
    possible that $b_i=b_j$ for $i\ne j$), and
  \item Each irreducible component of $C$ contains at least one point
    of $b_1,\ldots,b_k$.
  \end{itemize}
\end{Def}
\begin{thm}[\cite{LosevManin2000}, Theorems 2.2 and 2.6.3]\label{thm:LosevManin}
  The moduli space of $(0|k|\infty)$-marked Losev-Manin curves
  $\Mbar_{0|k|\infty}$ is a smooth projective (toric) variety, and there
  is a natural birational morphism
  $\varphi:\Mbar_{0,k+2}\to\Mbar_{0|k|\infty}$.
\end{thm}
\begin{rem}
  The spaces $\Mbar_{0|k|\infty}$ is an example of a moduli space
  $\Mbar_{0,\mathcal{A}}$ of weighted stable curves, developed later
  by Hassett \cite{Hassett2003}, and Theorem \ref{thm:LosevManin} is a
  special case of Theorems 2.1 and 4.1 of
  \cite{Hassett2003}. Specifically, there is a natural isomorphism
  $\Mbar_{0|k|\infty}\to\Mbar_{0,\mathcal{A}},$ where $\mathcal{A}$ is
  the weight datum $(1,\epsilon,\epsilon,\ldots,\epsilon,1)$ of length
  $k+2$, for $\epsilon\le1/k$.
\end{rem}
\begin{Def}
  Let $s\ge1$ be an integer. An \emph{order-$s$ orbifold
    $(0|k|\infty)$-marked Losev-Manin curve} is a $(k+2)$-marked twisted
  curve $(C,b_{0},b_1,\ldots,b_k,b_{\infty})$ (in the sense of
  \cite{Olsson2007}) whose coarse moduli space is a $k$-marked
  Losev-Manin curve, such that $C$ has orbifold structure only at
  $b_0$, $b_\infty$, and the nodes of $C$, all of which have order
  $s$.
\end{Def}
The moduli space $\Mbar^s_{0|k|\infty}$ of order-$s$ orbifold
$k$-marked Losev-Manin curves has a natural map
$\Mbar^s_{0|k|\infty}\to\Mbar_{0|k|\infty}$ that comes from taking
coarse moduli spaces of curves. Our calculations in Section
\ref{sec:CharacterizationOfLagrangianCone} will use the following
fact, a special case from Lemma 2.3 of \cite{Moon2011}.
\begin{lem}\label{lem:PsiPullback}
  Let $\psi_{0,LM}$ and $\psi_{\infty,LM}$ denote the tautological
  cotangent classes at $b_0$ and $b_\infty$ on $\Mbar_{0|k|\infty}.$
  The pullbacks $\varphi^*\psi_{0,LM}$ and $\varphi^*\psi_{\infty,LM}$
  along the reduction morphism $\Mbar_{0,k+2}\to\Mbar_{0|k|\infty}$
  are the cotangent classes $\psi_0$ and $\psi_\infty,$ respectively.
\end{lem}
\begin{rem}
  Lemma \ref{lem:PsiPullback} holds for order-$s$ orbifold Losev-Manin
  spaces, either using the cotangent classes $\bar\psi$ (as we do in
  this paper), or replacing $\Mbar_{0,k+2}$ with a stacky replacement
  $\Mbar^s_{0,k+2}$. ($\Mbar^s_{0,k+2}$ parametrizes curves where
  $b_0$ and $b_\infty$ have order-$s$ orbifold structure, as do any
  nodes that separate $b_0$ from $b_\infty.$)
\end{rem}

\section{The action of \texorpdfstring{$(\C^*)^{r+1}$}{(C\^{}*)\^{}\{r+1\}} on \texorpdfstring{$\Sym^d\P^r$}{Sym\^{}d P\^{}r}}\label{sec:TAction}
There is a natural action of $T:=(\C^*)^{r+1}$ on $\P^r.$ This induces
a diagonal action of $(\C^*)^{r+1}$ on $(\P^r)^d,$ which commutes with
the action of $S_d,$ hence acts on $\Sym^d\P^r.$ (The action on a
diagram $S\xleftarrow{\rho}{}S'\xrightarrow{f'}{}\P^r$ as in Section
\ref{sec:SymmetricPowers} is by postcomposition of
$f'$.) 
This
$T$-action on $\Sym^d\P^r$ induces an action on
$\Mbar_{g,n}(\Sym^d\P^r,\beta)$ for all $n$ and $\beta$.

The goal of this section is Theorem \ref{thm:FixedLocusLosevManin},
which explicitly characterizes the $T$-fixed locus in
$\Mbar_{g,n}(\Sym^d\P^r,\beta)$. The building blocks of the
construction are spaces $\Mbar_{g,n}(BG,0)$ of admissible covers from
\cite{AbramovichCortiVistoli2003}\footnote{These stacks compactify
  Hurwitz spaces, and are now usually referred to as moduli spaces of
  admissible covers, though \cite{AbramovichCortiVistoli2003} reserves
  that term for the related compactifications defined earlier by
  Harris-Mumford \cite{HarrisMumford1982}.}, the Losev-Manin spaces
from Section \ref{sec:LosevManin}, and combinatorial objects called
\textit{decorated graphs}.

\subsection{\texorpdfstring{$T$}{T}-fixed points and 1-dimensional orbits of
  \texorpdfstring{$\Sym^d\P^r$}{Sym\^{}d P\^{}r}}\label{sec:OrbitsOfT}
We begin by fixing notation for points and lines in $\P^r.$ We will
denote the coordinate points of $\P^r$ by $P_0,P_1,\ldots,P_r,$ where
$P_i$ is the point where the only nonzero coordinate is the $i$th
one. We denote by $L_{(i_1,i_2)}=L_{(i_2,i_1)}$ the line through
$P_{i_1}$ and $P_{i_2}.$ We write $P_{(i_1,i_2)}$ for the ``midpoint''
of this line, where the $i_1$-th and $i_2$-th coordinates are equal.

Recall from Section \ref{sec:SymmetricPowers} that a map
$f:S\to\Sym^d\P^r$ is the same as a degree-$d$ \'etale cover
$\rho:S'\to S,$ and a map $f':S'\to\P^r.$ We use the notation $\bullet$ for
$\Spec\C$, and $d(\bullet)$ for the union of $d$ copies of $\Spec\C.$
Note $d(\bullet)$ is the only degree-$d$ \'etale cover of $\bullet$,
so ($\C$-valued) points of $\Sym^d\P^r$ are in natural bijective
correspondence with maps $f':d(\bullet)\to\P^r$.
\begin{prop}\label{prop:UnionOfOrbits}
  Points of $\Sym^d\P^r$ with 0- and 1-dimensional $T$-orbits are
  classified as follows:
  \begin{enumerate}
  \item A point $(d(\bullet)\xrightarrow{f'}{}\P^r)\in\Sym^d\P^r$ is
    $T$-fixed if and only if $\Im(f')\subseteq\{P_0,\ldots,P_r\}$.\label{item:FixedPoints}
  \item $(d(\bullet)\xrightarrow{f'}{}\P^r)$ has a 1-dimensional
    $T$-orbit if and only if it is not $T$-fixed and
    $\Im(f')\subseteq\{P_0,\ldots,P_r\}\cup L_{(i_1,i_2)}$ for some
    $0\le i_1,i_2\le r$.\label{item:1DimOrbits}
  \end{enumerate}
\end{prop}
\begin{proof}
  \eqref{item:FixedPoints} follows from the definition of the
  $T$-action by post-composition, and that fact that
  $\{P_0,\ldots,P_r\}$ is the $T$-fixed locus of $\P^r.$

  The $r$-dimensional subtorus defined by $t_{i_1}=t_{i_2}$ acts
  trivially on $\{P_0,\ldots,P_r\}\cup L_{(i_1,i_2)}$, proving the
  backwards direction of \eqref{item:1DimOrbits}. If
  $\Im(f')\not\subseteq\{P_0,\ldots,P_r\}\cup L_{(i_1,i_2)},$ then
  $\Im(f')$ contains either two points on different coordinate lines,
  or a point not on a coordinate line. In either case, is it is easy
  to check explicitly that the $T$-orbit is at least 2-dimensional.
\end{proof}
\begin{rem}\label{rem:TFixZPart}
  The $T$-fixed points of $\Sym^d\P^r$ are in natural bijection with
  the set $\ZPart(d,r+1)$ of length-$(r+1)$ weak compositions, where the
  $i$th part is the number of points of $d(\bullet)$ mapping to
  $P_i$. We will use this identification from now on.
\end{rem}

\subsection{\texorpdfstring{$T$}{T}-fixed stable maps to
  \texorpdfstring{$\Sym^d\P^r$}{Sym\^{}d P\^{}r} with irreducible source curve}\label{sec:TFixedStableMaps}
It is well-known (see \cite{Liu2013}) that if $X$ is a Deligne-Mumford
stack with an action of a torus $T$, then a stable map $f:C\to X$ is
$T$-fixed if and only if each component $C_\nu$ of $C$ maps into the
fixed locus $X^T$, or maps to the closure $\bar{U}$ of a 1-dimensional
$T$-orbit $U,$ with special points (nodes and marks) and ramification
points mapping to $\bar{U}\setminus U.$ (In the latter case it follows
that $C_\nu$ is rational; we may regard $f|_{C_\nu}$ as a point of
$\Mbar_{0,2}(X,\beta)$ for some $\beta$.) If $T$ acts with isolated
fixed points, we refer to the two types of components of $C$ as
\emph{contracted} and \emph{noncontracted}, since those of the first
type map to a single point of $X$. On contracted components $C_\nu$,
$f$ factors through $BG$ for some $G$; thus $f|_{C_\nu}$ is an
admissible $G$-cover in the sense of
\cite{AbramovichCortiVistoli2003}. The following lemma classifies
noncontracted components of $T$-fixed stable maps to $\Sym^d\P^r$.
\begin{lem}\label{lem:IrreducibleFixedMap}
  Let $(f:C\to\Sym^d\P^r)\in\Mbar_{0,2}(\Sym^d\P^r,\beta)$ be a stable
  map of degree $\beta>0$ with irreducible source curve. Denote by
  $b_1$ and $b_2$ the two marked points of $C$. Denote by
  $\rho:C'\to C$ and $f':C'\to\P^r$ the associated degree $d$ \'etale
  cover and map to projective space, respectively. (See Section
  \ref{sec:SymmetricPowers}.) Then $(f:C\to\Sym^d\P^r)$ is $T$-fixed if
  and only if all of the following hold:
  \begin{itemize}
  \item $C'$ is a disjoint union of rational connected components
    $C'_{\eta}$. (Since $C$ has two orbifold points, this means that
    on coarse moduli spaces, $\rho$ is a cover, fully ramified over
    $b_1$ and $b_2$.)
  \item There exist distinct indices $0\le i_1,i_2\le r$ such that $f'$
    maps each component $C'_\eta$ either
    \begin{enumerate}[label=(\roman*)]
    \item to the line $L_{(i_1,i_2)}$, or
    \item to a $T$-fixed point of $\P^r$.
    \end{enumerate}
  \item On the level of coarse moduli spaces, the restriction
    $f'|_{C'_\eta}$ to any component of type (i) is a cover of
    $L_{(i_1,i_2)},$ fully ramified at the two points $\rho^{-1}(b_1)$
    and $\rho^{-1}(b_2)$.
  \item For each component $C'_\eta$, write $c_\eta$ for the degree of
    $\rho|_{C_\eta'}:C_\eta'\to C$. For components $C'_\eta$ of type
    $(i)$, write $\beta_\eta$ for the degree of
    $f'|_{C_\eta'}:C_\eta'\to L_{(i_1,i_2)}$, and
    $q_\eta:=\beta_\eta/c_\eta$. Then $q:=q_\eta$ is independent of
    the type (i) component $C'_\eta$.
  \end{itemize}
\end{lem}
\begin{proof}
  The first three statements follow from the fact that $C$ is genus
  zero with exactly two orbifold points, and from Proposition
  \ref{prop:UnionOfOrbits}. It is a straightforward computation in
  coordinates to check that the last statement is equivalent to the
  fact that the $T$-action is compatible with the map $\rho$,
  i.e. that the action of $\lambda\in T$ is equivalent to a coordinate
  change on $C$.
\end{proof}
\begin{rem}\label{rem:TrivialMonodromy}
  The same statement and proof apply to
  $\Mbar_{0,1}(\Sym^d\P^r,\beta)$ and $\Mbar_{0,0}(\Sym^d\P^r,\beta)$
  and in these cases we have a slightly stronger statement: since $C$
  has at most one orbifold point, it has no nontrivial \'etale
  cover. Thus $C'\cong C\times\{1,\ldots,d\}$ and $c_\eta=1$ for all
  $\eta$.
\end{rem}
From an irreducible $T$-fixed stable map as in Lemma
\ref{lem:IrreducibleFixedMap}, we may extract discrete data (see
\ref{sec:Combinatorics} for notation) as follows:
\begin{itemize}
\item The rational number $q$ associated to type $(i)$ components of $C'$.
\item The two compositions $f(b_1),f(b_2)\in\ZPart(d,r+1).$ (See Remark
  \ref{rem:TFixZPart}.)
\item A refinement of the above: for each $i\in\{0,\ldots,r\},$ the
  points of $C'$ mapping to $P_i$ are each counted with a multiplicity
  $c_\eta.$ Whereas $f(b_1)$ remembers only the sum for each $i$, we
  could instead record the list of multiplicities $c_\eta.$ The result
  is a multipartition $\Mon(b_1)\in\MultiPart(f(b_1)).$ This
  multipartition describes the monodromy of $f$ at $b_1$ as a
  conjugacy class in $G_{f(b_1)}$. Similarly
  $\Mon(b_2)\in\MultiPart(f(b_2)).$
\end{itemize}
\subsection{Decorated graphs}
Having classified irreducible components of $T$-fixed stable maps to
$\Sym^d\P^r,$ we will now describe how these components fit together. 
Following \cite{Liu2013}, we introduce combinatorial objects called
\emph{decorated graphs}, which capture the combinatorial data of
elements of $(\Mbar_{g,n}(\Sym^d\P^r,\beta))^T.$
\begin{Def}\label{Def:DecoratedGraph}
  An \emph{$n$-marked genus-$g$ $\Sym^d\P^r$-decorated graph}
  $(\Gamma,\Mark,\{g_v\},\VEval,q,\vec\Mon)$ is
  \begin{itemize}
  \item A graph $\Gamma$,
  \item A marking map $\Mark:\{1,\ldots,n\}\to V(\Gamma)$,
  \item A ``vertex genus'' map $V(\Gamma)\to\Z_{\ge0}$ denoted
    $v\mapsto g_v$,
  \item A ``vertex evaluation'' map
    $\VEval=(\VEval_0,\ldots,\VEval_r):V(\Gamma)\to\ZPart(d,r+1),$
  \item An ``edge degree ratio'' map $q:E(\Gamma)\to\Q_{>0}$,
  \item A ``monodromy map'' $\Mon=(\Mon_0,\ldots,\Mon_r)$ that
    assigns to each $j\in\{1,\ldots,n\}$ an element of
    $\MultiPart(\VEval(\Mark(j)))$ (see Section
    \ref{sec:Combinatorics}), and assigns to each flag
    $(v,e)\in F(\Gamma)$ an element of $\MultiPart(\VEval(v)),$
  \end{itemize}
  subject to the conditions:
  \begin{enumerate}
  \item $h_1(\Gamma)+\sum_{v\in V(\Gamma)}g_v=g.$
  \item Let $e$ be an edge of $\Gamma$ connecting vertices $v$ and
    $v'$. Then there exist two distinct indices
    $0\le i^{\mov}(v,e),i^{\mov}(v',e)\le r$ such
    that:\label{item:PartitionsAgree}
    \begin{itemize}
    \item $\VEval_{i^{\mov}(v,e)}(v)-\VEval_{i^{\mov}(v,e)}(v')>0.$
    \item If $i\not\in\{i^{\mov}(v,e),i^{\mov}(v',e)\},$ then
      $\VEval_i(v)=\VEval_i(v')$ and $\Mon_i(v,e)=\Mon_i(v',e)$ (as
      partitions of $\VEval_i(v)$).
    \item There are containments
      $\Mon_{i^{\mov}(v,e)}(v',e)\subseteq\Mon_{i^{\mov}(v,e)}(v,e)$
      and
      $\Mon_{i^{\mov}(v',e)}(v,e)\subseteq\Mon_{i^{\mov}(v',e)}(v',e)$,
      and the relation between complements holds:
      \begin{align*}
      \Mon_{i^{\mov}(v,e)}(v,e)\setminus\Mon_{i^{\mov}(v,e)}(v',e)=\Mon_{i^{\mov}(v',e)}(v',e)\setminus\Mon_{i^{\mov}(v',e)}(v,e).
      \end{align*}
      \item For
    $\eta\in\Mon_{i^{\mov}(v,e)}(v,e)\setminus\Mon_{i^{\mov}(v,e)}(v',e),$
    we have $\eta\in\frac{1}{q(e)}\Z$.
    \end{itemize}


  \item If $v\in V(\Gamma)$ with $g_v=0,$ $E(\Gamma,v)=\{e_v\}$, and
    $\Mark^{-1}(v)=\emptyset,$ then $\Mon(v,e_v)$ is the ``trivial''
    multipartition of $\MultiPart(\VEval(v))$ whose elements are all
    1.\label{item:V1Mon}
  \item If $v\in V(\Gamma)$ with $g_v=0,$ $E(\Gamma,v)=\{e_v\}$, and
    $\Mark^{-1}(v)=\{j\},$ then
    $\Mon(v,e_v)=\Mon(j).$\label{item:V11Mon}
  \item If $v\in V(\Gamma)$ with $g_v=0,$
    $E(\Gamma,v)=\{e_v^1,e_v^2\}$, and $\Mark^{-1}(v)=\emptyset,$ then
    $\Mon(v,e_v^1)=\Mon(v,e_v^2)$.\label{item:V2Mon}
  \end{enumerate}
\end{Def}
For brevity, we will write $\Gamma$ instead of $(\Gamma,\Mark,\{g_v\},\VEval,q,\vec\Mon)$. For a fixed $\Gamma$,
we introduce notation:
\begin{itemize}
\item Each part $\eta$ of the multipartitions $\Mon(v,e)$ and
  $\Mon(j)$ is an element of one of the multisets
  $(\Mon_0,\ldots,\Mon_r),$ and we write $i(\eta)$ for the element of
  $\{0,\ldots,r\}$ such that $\eta\in\Mon_{i(\eta)}$.
\item Let $\Mov(e)$ be the difference multiset
  $\Mon_{i^{\mov}(v,e)}(v,e)\setminus\Mon_{i^{\mov}(v,e)}(v',e)$, and
  let $\Stat(e):=\Mon(v,e)\setminus\Mov(e)$ be its complement.  By
  condition \ref{item:PartitionsAgree}, $\Mov(e)$ and $\Stat(e)$
  depend on $e$ rather than $(v,e)$. $\Mov(e)$ is the submultiset of
  ``moving parts'' of $\Mon(v,e)$ (or $\Mon(v',e)$), and $\Stat(e)$ is
  the submultiset of ``stationary parts''. Note that $\Stat(e)$ is a
  $\{0,\ldots,r\}$-labeled multiset. We write
  $\mov(e):=\abs{\Mov(e)}$.
\item Let $\Mon(e)$ be the partition $\bigcup_k\Mon_k(v,e)$ of $d$,
  which again by condition \ref{item:PartitionsAgree} depends only on
  $e$. Note that unlike $\Mon(v,e)$ and $\Mon(j)$, $\Mon(e)$ is only
  a partition of $d$, rather than a multipartition.
\item For $v$ satisfying any one of conditions \ref{item:V1Mon},
  \ref{item:V11Mon}, or \ref{item:V2Mon}, we write $\Mon(v)$ for
  $\Mon(v,e_v)$ or $\Mon(v,e_v^1)=\Mon(v,e_v^2).$
\item For an edge $e\in E(\Gamma)$, let
  $\beta(e)=\sum_{\eta\in\Mov(e)}\beta_\eta(e):=\sum_{\eta\in\Mov(e)}q(e)\eta.$
  Let $\beta(\Gamma)=\sum_{e\in E(\Gamma)}\beta(e).$
\item Denote by $\Graphs_{g,n}(\Sym^d\P^r,\beta)$ the finite set of
  $n$-marked genus-$g$ $\Sym^d\P^r$-decorated graphs $\Gamma$
  with $\beta(\Gamma)=\beta$. We refer to these as simply
  ``decorated graphs'' when no confusion is possible.
\end{itemize}
\begin{lem}\label{lem:FixedMapDeterminesGraph}
  There is a natural map
  $$\Psi:(\Mbar_{g,n}(\Sym^d\P^r,\beta))^T\to\Graphs_{g,n}(\Sym^d\P^r,\beta).$$
\end{lem}
\begin{proof}
  Let
  $(f:(C,b_1,\ldots,b_n)\to\Sym^d\P^r)\in(\Mbar_{g,n}(\Sym^d\P^r,\beta))^T$. Define
  sets $V(\Gamma)$ equal to the set of connected components of
  $f^{-1}((\Sym^d\P^r)^T),$ and $E(\Gamma)$ the set of noncontracted
  irreducible components of $C$. By Lemma
  \ref{lem:IrreducibleFixedMap}, associated to each noncontracted
  irreducible component of $C$ are two $T$-fixed points $P_{i_1}$ and
  $P_{i_2}$, so these define a graph $\Gamma.$

  We now define the various decorations of $\Gamma.$ Let $\Mark(j)$ be
  the connected component of $f^{-1}((\Sym^d\P^r)^T)$ containing
  $b_j$. Let $\VEval(v)$ be the $(r+1)$-tuple representing the
  $T$-fixed point $f(v)$, from Section \ref{sec:OrbitsOfT}. Let
  $q(e)=q$ be the rational number determined by Lemma
  \ref{lem:IrreducibleFixedMap}. Let $\Mon(j)$ be the monodromy of $f$
  at $b_j$. This is a conjugacy class in the isotropy group
  $G_{f(b_j)},$ and these are in natural bijection with
  $\MultiPart(\VEval(\Mark(j)))$. Finally, let $\Mon(v,e)$ be the
  monodromy of $f$ at the point $\xi(v,e)$ where the connected
  component $v$ meets the irreducible component $e$; this monodromy is
  naturally an element of $\MultiPart(\VEval(v))$.

  Condition \eqref{item:PartitionsAgree} for decorated graphs follows
  from the description in Lemma
  \ref{lem:IrreducibleFixedMap}. Condition \eqref{item:V1Mon} follows
  from Remark \ref{rem:TrivialMonodromy}. Condition
  \eqref{item:V11Mon} holds because for such $v$, $\xi(v,e_v)$ and
  $b_j$ are the same point of $C$. Condition \eqref{item:V2Mon} is
  true for the same reason, together with the fact that the inverse of
  a conjugacy class in $S_d$ is itself.
\end{proof}
\subsection{Classifying the connected components of
  \texorpdfstring{$(\Mbar_{g,n}(\Sym^d\P^r,\beta))^T$}{M\_{}\{g,n\}(Sym\^{}d P\^{}r,beta)\^{}T}}
The map in Lemma \ref{lem:FixedMapDeterminesGraph} gives a
stratification of $(\Mbar_{g,n}(\Sym^d\P^r,\beta))^T$ into (as we will
see) locally closed substacks. In this section we describe how the
strata fit together. To be precise, what we show does not quite
classify connected components, but rather certain open and closed
substacks --- see Remark \ref{rem:ConnectedComponentsOfFixedLocus}.
\begin{notation}\label{not:StableVertex}
  Let $(f:C\to\Sym^d\P^r)\in\Psi^{-1}(\Gamma)$. If
  $v\in V(\Gamma)$, then from Lemma \ref{lem:FixedMapDeterminesGraph},
  $v$ corresponds to a subcurve of $C$. We denote this by
  $C_v$. Similarly, for $e\in E(\Gamma)$, we write $C_e$ for the
  corresponding irreducible component of $C$. For
  $(v,e)\in F(\Gamma)$, we write $\xi(v,e)$ for the point
  $v\cap e\in C$, again using the notation of the proof of Lemma
  \ref{lem:FixedMapDeterminesGraph}. We say $(v,e)$ is a \emph{special
    flag} if $\xi(v,e)$ is a special point, equivalently if $g_v>0$ or
  $\val(v)>1$ or $\Mark^{-1}(v)\ne\emptyset$. Note that the isotropy
  group at $\xi(v,e)$ (resp. $b_j$) has order $\lcm(\Mon(v,e))$
  (resp. $\lcm(\Mon(j))$). For brevity we denote this by $r(v,e)$
  (resp. $r_j$).

  We adopt the following notation from \cite{Liu2013}, corresponding
  to conditions \ref{item:V1Mon}, \ref{item:V11Mon}, and
  \ref{item:V2Mon} in Definition \ref{Def:DecoratedGraph}:
  \begin{align*}
    V^1({\Gamma})&=\{v\in V(\Gamma)|g_v=0,\val(v)=1,\abs{\Mark^{-1}(v)}=0\}\\
    V^{1,1}({\Gamma})&=\{v\in
                             V(\Gamma)|g_v=0,\val(v)=1,\abs{\Mark^{-1}(v)}=1\}\\
    V^2(\Gamma)&=\{v\in
                       V(\Gamma)|g_v=0,\val(v)=2,\abs{\Mark^{-1}(v)}=0\}\\
    V^S({\Gamma})&=V(\Gamma)\setminus(V^1({\Gamma})\cup
                         V^{1,1}({\Gamma})\cup V^2(\Gamma)).
  \end{align*}
  We call vertices in $V^S({\Gamma})$ \emph{stable}. A vertex
  $v$ is stable if and only if $C_v$ is 1-dimensional (rather than a
  single point).

  For $v\in V^1({\Gamma})\cup V^{1,1}({\Gamma})$, we always write
  $E(\Gamma,v)=\{e_v=(v,v')\}$. For $v\in V^2(\Gamma)$, we always write
  $E(\Gamma,v)=\{e_v^1=(v,v_1),e_v^2=(v,v_2)\}.$
\end{notation}
\begin{Def}\label{Def:CombinableEdges}
  Let $\Gamma\in\Graphs_{g,n}(\Sym^d\P^r,\beta)$, and let
  $e_1,e_2\in E(\Gamma).$ We say $e_1$ and $e_2$ are
  \emph{combinable}, and write $e_1\parallel e_2$, if there exists
  $v\in V^2(\Gamma)$ with $\{e_1,e_2\}=\{e_v^1,e_v^2\}$ and the
  following hold:
  \begin{itemize}
  \item $q(e_1)=q(e_2)$,
  \item $i^{\mov}(v_1,e_1)=i^{\mov}(v,e_2)$ and $i^{\mov}(v,e_1)=i^{\mov}(v_2,e_2).$
  \end{itemize}
  Denote by $\mathcal{P}\subseteq\binom{E(\Gamma)}{2}$ the set of
  pairs $\{\{e_1,e_2\}:e_1\parallel e_2\}$.
\end{Def}
\begin{Def}
  Let $(v,e)\in F(\Gamma)$. We say $(v,e)$ is a \emph{steady flag} if
  either of the following holds:
  \begin{enumerate}
  \item $v\not\in V^{2}(\Gamma)$, or
  \item $v\in V^{2}(\Gamma)$ and $\{e_v^1,e_v^2\}\not\in\mathcal{P}$.
  \end{enumerate}
\end{Def}
\begin{Def}\label{Def:CombineEdges}
  Let $\Gamma\in\Graphs_{g,n}(\Sym^d\P^r,\beta)$ and let
  $e_1\parallel e_2$ be a pair of combinable edges. We may define a
  new decorated graph $\Comb(\Gamma,e_1\parallel
  e_2)\in\Graphs_{g,n}(\Sym^d\P^r,\beta)$ by \textbf{combining $e_1$
    and $e_2$}. In other words, we delete the vertex $v$ and the edges
  $e_1$ and $e_2,$ and add an edge $e_{12}=(v_1,v_2)$ with
  $q(e_{12})=q(e_1)=q(e_2),$ $\Mon(v_1,e_{12})=\Mon(v_1,e_1),$ and
  $\Mon(v_2,e_{12})=\Mon(v_2,e_2).$ (See Figure
  \ref{fig:CombineEdges}.) It is easy to check that
  $\Comb(\Gamma,e_1\parallel
  e_2)$ satisfies the two conditions of a decorated   
  graph, and that $\Mov(e_{12})=\Mov(e_1)\cup\Mov(e_2),$ and
  $\Mon(e_{12})=\Mon(e_1)=\Mon(e_2).$ There is a natural map
  $\phi_{e_1,e_2}:E(\Gamma)\to E(\Comb(\Gamma,e_1\parallel
  e_2))$ with $\phi_{e_1,e_2}(e_1)=\phi_{e_1,e_2}(e_2)=e_{12},$ and
  $\phi_{e_1,e_2}(e)=e$ for $e\in E(\Gamma)\setminus\{e_1,e_2\}$.
\end{Def}
\begin{figure}
  \centering
  \begin{tikzpicture}
    \draw[dashed] (1,0) circle (1);
    \draw (.5,0) node {$\Gamma_1$};
    \draw (2,0) node {$\bullet$} node[left] {$v_1$};
    \draw (2,0)--(4,0);
    \draw (2.5,0) node[below] {$e_1$};
    \draw (3.5,0) node[below] {$e_2$};
    \draw (3,0) node {$\bullet$} node[above] {$v$};
    \draw (4,0) node {$\bullet$} node[right] {$v_2$};
    \draw[dashed] (5,0) circle (1);
    \draw (5.5,0) node {$\Gamma_2$};
    \draw[->] (7,0)--(8,0);
    \draw[dashed] (10,0) circle (1);
    \draw (9.5,0) node {$\Gamma_1$};
    \draw (11,0) node {$\bullet$} node[left] {$v_1$};
    \draw (11,0)--(12.5,0);
    \draw (11.75,0) node[below] {$e_{12}$};
    \draw (12.5,0) node {$\bullet$} node[right] {$v_2$};
    \draw[dashed] (13.5,0) circle (1);
    \draw (14,0) node {$\Gamma_2$};
  \end{tikzpicture}
  \caption{Combining edges}
  \label{fig:CombineEdges}
\end{figure}
\begin{prop}\label{prop:CombiningCommutes}
  Let $\Gamma\in\Graphs_{g,n}(\Sym^d\P^r,\beta)$, and let
  $e_1\parallel e_2$ and $e_1'\parallel e_2'$ be two distinct pairs of
  combinable edges of $\Gamma$. Then
  $\phi_{e_1,e_2}(e_1')\parallel\phi_{e_1,e_2}(e_2')$ as edges of
  $\Comb(\Gamma,e_1\parallel e_2)$ and
  $\phi_{e_1',e_2'}(e_1)\parallel\phi_{e_1',e_2'}(e_2)$ as edges of
  $\Comb(\Gamma,e_1'\parallel e_2')$. Also, combining pairs
  commutes, i.e.
  \begin{align*}
    \Comb(\Comb(\Gamma,e_1\parallel e_2),e_1'\parallel
    e_2')\cong\Comb(\Comb(\Gamma,e_1'\parallel
    e_2'),e_1\parallel e_2),
  \end{align*}
  and this isomorphism identifies the maps
  $\phi_{e_1,e_2}\circ\phi_{e_1',e_2'}$ and
  $\phi_{e_1',e_2'}\circ\phi_{e_1,e_2}.$
\end{prop}
\begin{proof}
  There are two cases, pictured in the left side of Figure
  \ref{fig:CombineCommutes}; either the pairs $e_1\parallel e_2$ and
  $e_1'\parallel e_2'$ share an edge, or they do not.
  \begin{figure}
    \begin{center}
      \begin{tikzpicture}[scale=1.2]
        \draw[dashed] (1,1.8) circle (.5);
        \draw (.75,1.8) node {$\Gamma_1$};
        \draw (1.5,1.8) node {$\bullet$};
        \draw (1.5,1.8) node[left] {$v_1$};
        \draw (1.5,1.8)--(2.5,1.8);
        \draw (2,1.8) node[below] {$e_1$};
        \draw (2.5,1.8) node {$\bullet$};
        \draw (2.5,1.8) node[above] {$v$};
        \draw (2.5,1.8)--(3.5,1.8);
        \draw (3,1.8) node[below] {$e_2=e_1'$};
        \draw (3.5,1.8) node {$\bullet$};
        \draw (3.5,1.8) node[above] {$v'$};
        \draw (3.5,1.8)--(4.5,1.8);
        \draw (4,1.8) node[below] {$e_2'$};
        \draw (4.5,1.8) node {$\bullet$};
        \draw (4.5,1.8) node[right] {$v_2'$};
        \draw[dashed] (5,1.8) circle (.5);
        \draw (5.25,1.8) node {$\Gamma_2$};
        \draw[->] (5.8,1.8)--(6.7,1.8);
        \draw[dashed] (7.5,1.8) circle (.5);
        \draw (7.25,1.8) node {$\Gamma_1$};
        \draw (8,1.8) node {$\bullet$};
        \draw (8,1.8) node[left] {$v_1$};
        \draw (8,1.8)--(9,1.8);
        \draw (8.5,1.8) node[below] {$e$};
        \draw (9,1.8) node {$\bullet$};
        \draw (9,1.8) node[right] {$v_2'$};
        \draw[dashed] (9.5,1.8) circle (.5);
        \draw (9.75,1.8) node {$\Gamma_2$};
        \draw[dashed] (-.5,0) circle (.5);
        \draw (-.75,0) node {\small$\Gamma_1$};
        \draw (0,0) node {\small$\bullet$};
        \draw (0,0) node[left] {\small$v_1$};
        \draw (0,0)--(1,0);
        \draw (0.5,0) node[below] {\small$e_1$};
        \draw (1,0) node {\small$\bullet$};
        \draw (1,0) node[above] {\small$v$};
        \draw (1,0)--(2,0);
        \draw (1.5,0) node[below] {\small$e_2$};
        \draw (2,0) node {\small$\bullet$};
        \draw (2,0) node[right] {\small$v_2$};
        \draw[dashed] (2.5,0) circle (.5);
        \draw (2.5,0.1) node[above] {\small$\Gamma_2$};
        \draw (3,0) node {\small$\bullet$};
        \draw (3,0) node[left] {\small$v_1'$};
        \draw (3,0)--(4,0);
        \draw (3.5,0) node[below] {\small$e_1'$};
        \draw (4,0) node {\small$\bullet$};
        \draw (4,0) node[above] {\small$v'$};
        \draw (4,0)--(5,0);
        \draw (4.5,0) node[below] {\small$e_2'$};
        \draw (5,0) node {\small$\bullet$};
        \draw (5,0) node[right] {\small$v_2'$};
        \draw[dashed] (5.5,0) circle (.5);
        \draw (5.75,0) node {\small$\Gamma_3$};
        \draw[->] (6.2,0)--(6.75,0);
        \draw[dashed] (7.5,0) circle (.5);
        \draw (7.25,0) node {\small$\Gamma_1$};
        \draw (8,0) node {\small$\bullet$};
        \draw (8,0) node[left] {\small$v_1$};
        \draw (8,0)--(9,0);        
        \draw (8.5,0) node[below] {\small$e_{12}$};
        \draw (9,0) node {\small$\bullet$};
        \draw (9,0) node[right] {\small$v_2$};
        \draw[dashed] (9.5,0) circle (.5);
        \draw (9.5,0.1) node[above] {\small$\Gamma_2$};
        \draw (10,0) node {\small$\bullet$};
        \draw (10,0) node[left] {\small$v_1'$};
        \draw (10,0)--(11,0);        
        \draw (10.5,0) node[below] {\small$e_{12}'$};
        \draw (11,0) node {\small$\bullet$};
        \draw (11,0) node[right] {\small$v_2'$};
        \draw[dashed] (11.5,0) circle (.5);
        \draw (11.75,0) node {\small$\Gamma_3$};
      \end{tikzpicture}
    \end{center}
    \caption{Combining two pairs of edges}
    \label{fig:CombineCommutes}
  \end{figure}
  Suppose we are in the first case, i.e. the top line of Figure
  \ref{fig:CombineCommutes}. By definition of $\phi_{e_1,e_2},$ the
  edges $\phi_{e_1,e_2}(e_1')$ and $\phi_{e_1,e_2}(e_2')$ meet at $v'$
  (precisely, at the corresponding vertex in
  $\Comb(\Gamma,e_1\parallel e_2)$), and satisfy the three
  conditions of Definition \ref{Def:CombinableEdges}. Thus
  $\phi_{e_1,e_2}(e_1')\parallel\phi_{e_1,e_2}(e_2')$. Similarly
  $\phi_{e_1',e_2'}(e_1)\parallel\phi_{e_1',e_2'}(e_2)$. To see that
  $\Comb(\Comb(\Gamma,e_1\parallel e_2),e_1'\parallel
  e_2')\cong\Comb(\Comb(\Gamma,e_1'\parallel e_2'),e_1\parallel
  e_2)$, we note that both are obtained from the graph in Figure
  \ref{fig:CombineCommutes} by replacing the three edges shown with a
  single edge $e$ connecting $v_1$ to $v_2'.$ The decorations on this
  edge are:
  \begin{itemize}
  \item $q(e):=q(e_1)=q(e_2)=q(e_2')$,
  \item $\Mon(e):=\Mon(e_1)=\Mon(e_2)=\Mon(e_2')$,
  \item $i^{\mov}(v_1,e):=i^{\mov}(v_1,e_1)=i^{\mov}(v,e_2)=i^{\mov}(v',e_2')$, and
  \item $i^{\mov}(v_2',e):=i^{\mov}(v_2,e_2')=i^{\mov}(v',e_2)=i^{\mov}(v,e_1)$,
  \end{itemize}
  where the equalities follow from $e_1\parallel e_2$ and
  $e_2\parallel e_2'$.  The maps $\phi_{e_1,e_2}\circ\phi_{e_1',e_2'}$
  and $\phi_{e_1',e_2'}\circ\phi_{e_1,e_2}$ both send all of $e_1,$
  $e_2=e_1'$, and $e_2'$ to $e.$

  The second case (the bottom line of \ref{fig:CombineCommutes}) is a
  special case of this argument, so we omit it.
\end{proof}
\begin{cor}\label{cor:CombineSetOfEdges}
  Let $\Gamma\in\Graphs_{g,n}(\Sym^d\P^r,\beta)$, and let
  $\mathcal{E}$ be any subset of the set $\mathcal{P}(\Gamma)$
  of pairs of combinable edges in $\Gamma$. Then there is a
  well-defined graph
  $\Comb(\Gamma,\mathcal{E})\in\Graphs_{g,n}(\Sym^d\P^r,\beta)$
  obtained by combining all edge pairs in $\mathcal{E}$, in any order,
  and a well-defined associated map $\phi_{\mathcal{E}}:E(\Gamma)\to
  E(\Comb(\Gamma,\mathcal{E}))$. Furthermore, $\mathcal{E}$ is
  determined by the graphs $\Gamma$ and
  $\Comb(\Gamma,\mathcal{E})$, and the map $\phi_{\mathcal{E}}$.
\end{cor}
\begin{proof}
  The existence statement comes from repeatedly applying Proposition
  \ref{prop:CombiningCommutes}. The uniqueness statement amounts to
  the fact that if $e_1\parallel e_2$ is a combinable pair of edges in
  $\Gamma,$ then
  $\phi_{\mathcal{E}}(e_1)=\phi_{\mathcal{E}}(e_2)$ if and only if
  $(e_1,e_2)\in\mathcal{E}.$ This follows from factoring
  $\phi_{\mathcal{E}}$ as a sequence of edge combination maps as in
  Definition \ref{Def:CombineEdges}.
\end{proof}
Corollary \ref{cor:CombineSetOfEdges} may be restated as
follows. Definition \ref{Def:CombineEdges} determines a partial order
$\le$ on $\Graphs_{g,n}(\Sym^d\P^r,\beta)$, where
$\Gamma'\le\Gamma$ if $\Gamma'$ can be obtained from
$\Gamma$ by combining edges. Corollary
\ref{cor:CombineSetOfEdges} then states that for
$\Gamma\in\Graphs_{g,n}(\Sym^d\P^r,\beta)$, there is a natural
order-reversing bijection between
$\{\Gamma':\Gamma'\le\Gamma\}$ and
$\{\text{subsets of $\mathcal{P}(\Gamma)$}\},$ where the latter
is partially ordered by inclusion. In particular, associated to
$\Gamma$ is a unique \emph{minimal} decorated graph
$\Comb(\Gamma,\mathcal{P}(\Gamma)).$ Denote by
$\Graphs_{g,n}^{\min}(\Sym^d\P^r,\beta)$ the set of $\le$-minimal
elements of $\Graphs_{g,n}(\Sym^d\P^r,\beta)$.
\begin{thm}\label{thm:GraphsAndStrata}
  Let $\Gamma_0\in\Graphs_{g,n}(\Sym^d\P^r,\beta)$. The closure
  of $\Psi^{-1}(\Gamma_0)$ is
  \begin{align*}
    \bigcup_{\substack{\Gamma\in\Graphs_{g,n}(\Sym^d\P^r,\beta)\\\Gamma_0\le\Gamma}}\Psi^{-1}(\Gamma),
  \end{align*}
  where $\Psi$ is the map from Lemma \ref{lem:FixedMapDeterminesGraph}.
\end{thm}
\begin{lem}\label{lem:UnionOfChainsMapping}
  Let \raisebox{-5pt}{
    \begin{tikzpicture}
      \draw (1,0.1) node {$\Gamma_0=$};
      \draw (2,0) node {$\bullet$} node[left] {$v_1$};
      \draw (2,0)--(3,0);
      \draw (2.5,0) node[above] {$e$};
      \draw (3,0) node {$\bullet$} node[right] {$v_2$,};
    \end{tikzpicture}
  } where each of $v_1$ and $v_2$ contains a single marked point,
  $b_1$ and $b_2$, and $g_{v_1}=g_{v_2}=0$. Let $f:C\to\Sym^d\P^r$ be
  in the closure of $\Psi^{-1}(\Gamma_0)$, and let
  $\rho:C'\to C$ and $f':C'\to \P^r$ be the associated maps. Write
  $C'_\eta$ for a noncontracted irreducible component of $C'$,
  corresponding to $\eta\in\Mov(e)\subseteq\Mon(e)$, as described in
  Lemma \ref{lem:IrreducibleFixedMap}. Denote by
  $L_{e}:=L_{(i^{\mov}(v_1,e),i^{\mov}(v_2,e))}$ the line in $\P^r$
  connecting $P_{i^{\mov}(v_1,e)}$ and $P_{i^{\mov}(v_2,e)}$. Then:
  \begin{enumerate}
  \item $C$ and $C_\eta'$ are nodal chains of rational
    curves,\label{item:NodalChain}
  \item $f'|_{C_\eta'}$ maps one irreducible component of $C_\eta'$
    to $L_{e}$ with degree
    $\beta_\eta(e)=q(e)\cdot\eta$ (on coarse moduli spaces), and
    is fully ramified at the two special points of this component,
    and\label{item:OneNoncontracted}
  \item $f'|_{C_\eta'}$ contracts all other irreducible components of
    $C_\eta'$ to one of the endpoints of
    $L_{e}$.\label{item:RestContracted}
  \end{enumerate}
  That is, the restriction to $C_\eta'$ of a point in
  $\bar{\Psi^{-1}(\Gamma_0)}$ may be represented as in Figure
  \ref{fig:FixedMap} (where despite appearances we mean for the map to
  $L_{e}$ to have a single preimage point over each of
  $P_{i^{\mov}(v_1,e)}$ and $P_{i^{\mov}(v_2,e)}$).
\end{lem}
\begin{figure}
  \centering
  \begin{center}
    \begin{tikzpicture}
      \draw (0,0)--(1.2,.2); \draw (.8,.2)--(1.5,.1); \draw (1.7,.1)
      node {\tiny$\cdots$}; \draw (1.9,.1)--(2.6,.2); \draw
      (2.2,.2)--(3.6,0);
      \draw (3.4,-.1) .. controls (3.9,2.6) ..(4.1,1.1) .. controls
      (4.3,-.4) .. (4.8,2.3);
      \draw (4.4,2.2)--(5.8,2); \draw
      (5.4,2)--(6.1,2.1); \draw (6.3,2.1) node {\tiny$\cdots$};
      \draw (6.5,2.1)--(7.2,2); \draw (6.8,2)--(8,2.2); \draw[->]
      (4.1,-.4)--(4.1,-1.4); \draw (2,1) node {$C_\eta'$};
      \draw (0,-2)--(1.2,-1.8); \draw (.8,-1.8)--(1.5,-1.9); \draw
      (1.7,-1.9) node {\tiny$\cdots$}; \draw (1.9,-1.9)--(2.6,-1.8);
      \draw (2.2,-1.8)--(3.6,-2); \draw (3.4,-2)--(4.8,-1.8); \draw
      (4.4,-1.8)--(5.8,-2); \draw (5.4,-2)--(6.1,-1.9); \draw
      (6.3,-1.9) node {\tiny$\cdots$}; \draw (6.5,-1.9)--(7.2,-2);
      \draw (6.8,-2)--(8,-1.8); \draw (-.5,-2) node {$C$}; \draw
      (0,-2) node {$\bullet$}; \draw (8,-1.8) node {$\bullet$};
      \draw[->] (9,1)--(10,1); \draw (11,-.2)--(11,2); \draw
      (11.1,.9) node[right] {$L_{e}$};
      \draw (11,2) node {$\bullet$};
      \draw (11.1,2) node[right] {\small$P_{i^{\mov}(v_1,e)}$};
      \draw (11,-.2) node {$\bullet$};
      \draw (11.1,-.2) node[right] {\small$P_{i^{\mov}(v_2,e)}$};
    \end{tikzpicture}
    \caption{A portion of a map in $\bar{\Psi^{-1}(\Gamma_0)}$,
      with $\eta=1$ and $q(e)=3$}
    \label{fig:FixedMap}
  \end{center}
\end{figure}
\begin{proof}[Proof of Lemma]
  Let $f:\mathcal{C}\to\P^r$ be a family over $S$ of stable maps whose
  generic fiber is in $\Psi^{-1}(\Gamma_0),$ and let $s\in S$
  such that the fiber over $s$ is the stable map $f:C\to\Sym^d\P^r.$
  After an \'etale base change $\tilde S\to S$, $\mathcal{C}'$ is a
  union of connected components $\mathcal{C}_\eta'$ indexed by
  $\Mon(e),$ and the maps $\mathcal{C}_\eta'\to\mathcal{C}$ have
  degrees determined by $\Mon(e)$. Fix $\eta\in\Mov(e)$.
  
  Consider the Stein factorization of $f'$ relative to $S$:
  \begin{center}
    \begin{tikzcd}
      \mathcal{C}_{\eta}'\arrow[r,"\sf"]\arrow[rrr,"f'",bend right]&\bar{\mathcal{C}_{\eta}'}\arrow[r]\arrow[rr,"\bar{f'}",bend
      left]&\P^r\times
      S\arrow[r]&\P^r
    \end{tikzcd}
  \end{center}
  (The map $\sf$ contracts connected components of fibers of
  $\mathcal{C}'_\eta$ over $\P^r\times S$.)
  On a generic fiber of $\bar{\mathcal{C}_\eta'}$ over $S$, the
  divisors $\bar{f'}^*(P_{i^{\mov}(v_1,e)}$ and
  $\bar{f'}^*(P_{i^{\mov}(v_2,e)}$ are each supported on a single
  point. By the definition of $\sf$, on the special
  fiber $\bar{C_\eta'},$ these divisors are each supported on a
  connected locus, hence a single point --- specifically, the
  points $\sf(\rho^{-1}(b_1))$ and $\sf(\rho^{-1}(b_2))$,
  respectively.
  As any component of $\bar{C_\eta'}$ maps surjectively to $L_{e},$
  this implies that $\bar{C_\eta'}$ is irreducible. This proves claims
  \eqref{item:OneNoncontracted} and \eqref{item:RestContracted}.

  Since $f'$ is $T$-fixed, the above implies that a component of
  $C_\eta'$ not contracted by $f'$ has exactly two points that are
  nodes or are in $\rho^{-1}(b_1)$ or
  $\rho^{-1}(b_2)$. 

  If $C$ is not a chain, then since it is genus zero, some component
  $D$ has only one special point. By stability, there is a component
  of $\rho^{-1}(D)$ that is not contracted by $f'$. This contradicts
  the previous paragraph. Thus $C$ is a chain, and it follows that
  each $C_\eta'$ is a chain. This proves claim
  \eqref{item:NodalChain}.
\end{proof}

\begin{proof}[Proof of Theorem \ref{thm:GraphsAndStrata}]
  It is sufficient to consider the situation of Lemma
  \ref{lem:UnionOfChainsMapping}. To see this, note that any
  $\Gamma_0\in\Graphs_{g,n}(\Sym^d\P^r,\beta)$ may be decomposed
  into subgraphs of the form in the Lemma, together with single-vertex
  graphs, glued at marked points. There is a corresponding
  decomposition of $\Psi^{-1}(\Gamma_0)$ as a product (up to a
  finite morphism), and this decomposition extends to the closure (see
  \cite{AbramovichGraberVistoli2008}, Section 5.2, or \cite{Liu2013},
  Section 9.2). Thus we may treat each factor of the product
  separately.
  
  First, we
  show $$\bar{\Psi^{-1}(\Gamma_0)}\subseteq\bigcup_{\Gamma\ge\Gamma_0}\Psi^{-1}(\Gamma).$$
  Let
  $(f:C\to\P^r)\in\bar{\Psi^{-1}(\Gamma_0)}.$ By 
  Lemma \ref{lem:UnionOfChainsMapping}, we conclude:
  \begin{itemize}
  \item $\Psi(f:C\to\Sym^d\P^r)$ is a chain.
  \item The degree ratios $q(e)$ are equal for all edges $e$.
  \item The partitions $\Mon(e)$ are equal for all edges $e$.
  \item For any edge $e=(v,v'),$ where $v$ and $\Mark(1)$ are on the
    same connected component of $\Gamma\setminus\{e\}$, we have
    $i^{\mov}(v,e)=i^{\mov}(v_1,e_{12})$ and
    $i^{\mov}(v',e)=i^{\mov}(v_2,e_{12}).$ (This follows from the
    proof of Lemma \ref{lem:UnionOfChainsMapping}.
  \end{itemize}
  Thus any pair of adjacent edges in
  $\Psi(f:C\to\Sym^d\P^r)$ is combinable. Combining them all yields
  $\Gamma_0,$ i.e. $\Gamma_0\le\Psi(f:C\to\Sym^d\P^r)$.

  \medskip

  For the reverse inclusion, first suppose
  ${\Gamma}\ge{\Gamma}_0$ has a single pair of combinable
  edges, i.e. 
    \begin{center}
      \begin{tikzpicture}
        \draw (1,0.1) node {$\Gamma=$}; \draw (2,0) node
        {$\bullet$} node[left] {$v_1$}; \draw (2,0)--(4,0); \draw
        (2.5,0) node[above] {$e_{1}$}; \draw (3,0) node {$\bullet$}
        node[above] {$v$}; \draw (3.5,0) node[above] {$e_2$}; \draw
        (4,0) node {$\bullet$} node[right] {$v_2$.};
      \end{tikzpicture}
    \end{center}
 Fix $(f:C\to\Sym^d\P^r)\in\Psi^{-1}(\Gamma).$ We will
  construct a family $f:\mathcal{C}\to\Sym^d\P^r$ over $\C$ whose
  restriction to $0\in\C$ is the map $f:C\to\Sym^d\P^r$.
  
  By Lemma \ref{lem:IrreducibleFixedMap} and by representability of
  $f:C\to\Sym^d\P^r$, the orbifold points and nodes of $C$ have order
  $\lcm(\Mon(e_1))=\lcm(\Mon(e_2)).$ Thus $C$ is isomorphic to
  $V(xy)\subseteq[\P^2/\mu_{\lcm(\Mon(e_1))}],$ where $\P^2$ has
  coordinates $x,y,z,$ and $\lcm(\Mon(e_1))$ acts by multiplication by
  inverse roots of unity on the first two coordinates. Define
  $\mathcal{C}$ so that $\mathcal{C}_t=V(xy-tz^2)$ for $t\in\C.$
  Precisely, $\mathcal{C}$ is an open subset of
  $\left[\Bl_{[1:0:0],[0:1:0]}\P^2/\mu_{\lcm(\Mon(e_1))}\right]$.

  For $\eta\in\Mon(e_1)$ a part, there is an \'etale quotient map
  $\tilde\rho:[\P^2/\mu_\eta]\to[\P^2/\mu_{\lcm(\Mon(e_1))}]$. As above,
  define $(\mathcal{C}_\eta')_t=V(xy-tz^2)\subseteq[\P^2/\mu_\eta].$

  We must now define a map $\tilde f':\mathcal{C}_\eta'\to\P^r$ for
  each $\eta\in\Mon(e_1).$ As $\P^r$ is a variety, it is enough to
  define this on coarse moduli spaces. We choose isomorphisms of the
  fibers $(\mathcal{C}_\eta')_0$ and $\mathcal{C}_0$ with $C_\eta'$
  and $C$ respectively, such that the maps $\tilde\rho$ and $\rho$ are
  identified. Then $f'$ defines a map $\tilde
  f'_0:(\mathcal{C}_\eta')_0\to L_{e_1}=L_{e_2}$. (The case where
  $C_\eta'$ is contracted is trivial, so we assume it is not
  contracted.) By Lemma \ref{lem:UnionOfChainsMapping}, after
  equivariantly identifying $L_{e_1}\cong\P^1,$ $\tilde f'_0$
  is given (without loss of generality, on coarse moduli spaces) by
  \begin{align*}
    [x:0:z]&\mapsto[0:1]\\
    [0:y:z]&\mapsto[y^{\beta_\eta(e_1)}:z^{\beta_\eta(e_1)}].
  \end{align*}
  It remains to extend this to a map $\tilde f':\mathcal{C}_\eta'\to
  L_{e_1}$ that is fixed with respect to the $T$-action,
  i.e. fully ramified over the endpoints of $L_{e_1}$. We
  observe that the rational map
  \begin{align*}
    [x:y:z]\mapsto[y^{\beta_\eta(e_1)}:z^{\beta_\eta(e_1)}]
  \end{align*}
  is regular after blowing up the point $[1:0:0]$. This defines a map
  $\tilde f'$ as desired. Doing this for all $\eta$ shows that
  $f:C\to\Sym^d\P^r$ is in $\bar{\Psi^{-1}(\Gamma_0)}$.

  If ${\Gamma}$ has more than one pair of combinable edges, we
  apply this argument repeatedly.
\end{proof}
\begin{cor}\label{cor:GraphOpenClosedSubstack}
  $(\Mbar_{g,n}(\Sym^d\P^r,\beta))^T$ is a disjoint union of open and
  closed substacks $\bar{\Psi^{-1}(\Gamma)}$, for
  $\Gamma\in\Graphs_{g,n}^{\min}(\Sym^d\P^r,\beta)$. We define
  $\Mbar_{{\Gamma}}:=\bar{\Psi^{-1}(\Gamma)}$.
\end{cor}
\subsection{Explicit description of \texorpdfstring{$\Mbar_{{\Gamma}}$}{M\_{}Gamma}}\label{sec:FixedLocusLosevManin}
The rest of this section proves the following:
\begin{thm}\label{thm:FixedLocusLosevManin}
  For a stable vertex $v$ or edge $e=(v_1,v_2)$ of a minimal decorated
  graph
  $\Gamma=(\Gamma,\Mark,\{g_v\},\VEval,q,\Mon)\in\Graphs_{g,n}^{\min}(\Sym^d\P^r,\beta)$,
  we define
  \begin{align*}
    \Mbar_v:&=\Mbar_{g_v,\vec{\Mon}(v)}(BS_{\VEval(v)},0)\\
    &\\
    \Mbar_e:&=\left[\Mbar^{\lcm(\Mon(e))}_{v_1|{\mov(e)}|v_2}\Big/\left(\prod_{\eta\in\Mov(e)}\mu_{\beta_\eta(e)}\wr S_e\right)\right],
  \end{align*}
  where:
  \begin{itemize}
  \item $\vec{\Mon}(v)$ is the list of multipartitions
    $\{\Mon(i)\}_{i\in\Mark^{-1}(v)}\cup\{\Mon(v,e)\}_{e\in
      E(\Gamma,v)}$,
  \item $\Mbar^{\lcm(\Mon(e))}_{v_1|{\mov(e)}|v_2}$ is the order
    $\lcm(\Mon(e))$ orbifold Losev-Manin space with ${\mov(e)}$
    marked points $b_1,\ldots,b_{{\mov(e)}}$ and labeling set
    $\{v_1,v_2\}$, from Section \ref{sec:LosevManin},
  \item $S_e$ is the group $C_{\Stat}(e)\times S_{\Mov(e)},$ where
    $C_{\Stat}(e)$ is the centralizer of any element of the conjugacy
    class $\Stat(e)$ in $\prod_{i=0}^rS_{\abs{\Stat(e)_i}}$, and acts
    trivially on the Losev-Manin space,
  \item A generator of $\mu_{\beta_\eta(e)}$ acts by translating the
    marked point $b_\eta$ by $e^{2\pi i/q(e)},$ and
  \item $\wr$ denotes the wreath product.
  \end{itemize}
  Then the substack $\Mbar_{\Gamma}$ associated to
  $\Gamma$ is isomorphic to a $\left(\prod\limits_{\text{$(v,e)$ steady}}\bar{C}_{\VEval(v)}(\Mon(v,e))\right)$-gerbe over
  \begin{align}\label{eq:ProductFixedLocus}
    \left[\left(\prod_{v\in V^S(\Gamma)}\Mbar_v\times\prod_{e\in
          E(\Gamma)}\Mbar_e\right)\Big/\Aut(\Gamma)\right],
  \end{align}
  where $\bar{C}_{\VEval(v)}(\Mon(v,e))$ is the centralizer in
  $G_{\VEval(v)}$ of any element of the conjugacy class $\Mon(v,e)$,
  modulo the subgroup generated by that element.
\end{thm}
\begin{proof}[Proof of \ref{thm:FixedLocusLosevManin}]
  Using Theorem \ref{thm:GraphsAndStrata}, Lemma
  \ref{lem:UnionOfChainsMapping}, and the gluing morphisms for
  $\Mbar_{g,n}(X,\beta)$ (see \cite{AbramovichGraberVistoli2008},
  Section 5.2), $\Mbar_{{\Gamma}}$ is a
  $\left(\prod\limits_{\text{$(v,e)$
        steady}}\bar{C}_{\VEval(v)}(\Mon(v,e))\right)$-gerbe over
  \begin{align*}
    \left[\left(\prod_{v\in
    V(\Gamma)}\Mbar_{g_v,\vec{\Mon}(v)}(BS_{\VEval(v)},0)\times\prod_{e\in
    E(\Gamma)}\Mbar_{{\Gamma}_e}
    \right)\Big/\Aut(\Gamma)\right],
  \end{align*}
  where \raisebox{-5pt}{
    \begin{tikzpicture}
      \draw (1,0.1) node {$\Gamma_e=$};
      \draw (2,0) node {$\bullet$} node[left] {$v_1$};
      \draw (2,0)--(3,0);
      \draw (2.5,0) node[above] {$e$};
      \draw (3,0) node {$\bullet$} node[right] {$v_2$,};
    \end{tikzpicture}
  } and the decorations are inherited from ${\Gamma},$ with
  $g_{v_1}=g_{v_2}=0.$ (Note that the two vertices of
  ${\Gamma}_e$ are labeled, i.e. $\Aut(\Gamma_e)=1$.)

  (The gerbe structure appears because gluing morphisms are fibered
  over the rigidified inertia stack $\bar{I}\Sym^d\P^r$, see
  \cite{AbramovichGraberVistoli2008} or \cite{Liu2013}. The group
  $\bar{C}_{\VEval(v)}(\Mon(v,e))$ is the isotropy group of
  $\bar{I}\Sym^d\P^r$ at the point of $\bar{I}\Sym^d\P^r$
  corresponding to $\Mon(v,e)$.)

  We need to show that, for all $e=(v_1,v_2)\in E(\Gamma),$ we have
  \begin{align*}
    \Mbar_{{\Gamma}_e}
    \cong\left[\Mbar^{\lcm(\Mon(e))}_{v_1|{\mov(e)}|v_2}\Big/\left(\prod_{\eta\in\Mov(e)}\mu_{\beta_\eta(e)}\wr
        S_e\right)\right].
  \end{align*}  
   Write
  $P_{e}:=P_{(i^{\mov}(v_1,e),i^{\mov}(v_2,e))}$ for the midpoint of
  $L_e$. For $(f:C\to\P^r)\in\Mbar_{\Gamma_e}$, consider the
  preimage of $P_e$ under the associated map $f':C'\to\P^r$. By Lemma
  \ref{lem:UnionOfChainsMapping}, $C'$ is a union of connected
  components $C_\eta'$ for $\eta\in\Mon(e)$, and if $\eta\in\Mov(e)$
  then the preimage of $P_e$ on $C_\eta'$ consists of $\beta_\eta(e)$
  points on the single noncontracted component of $C_\eta'$. These
  points are $\mu_{\beta_\eta(e)}$-translates of each other, under the
  natural action that fixes the two special points.

  After a principal
  $\left(\prod_{\eta\in\Mov(e)}\mu_{\beta_\eta(e)}\wr S_e\right)$-cover
  $\widetilde\Mbar_{\Gamma_e}\to\Mbar_{\Gamma_e}$, we may
  fix a labeling of the connected components $C_{\eta}'$, and label a
  distinguished preimage of $P_e$ on $C_\eta'$ for
  $\eta\in\Mov(e)$. (The $S_e$-cover removes all automorphisms of
  stable maps induced by automorphisms of the image curve that commute
  with the monodromy at $b_{v_1}$ and $b_{v_2}$.) Remembering the
  images of these distinguished points under $\rho$ yields a nodal
  chain of rational curves with ${\mov(e)}$ labeled marked points,
  none of which coincides with $b_{v_1}$ or $b_{v_2}$. The stability
  condition for $\Mbar_{0,\{\Mon(e),\Mon(e)\}}(L_e,\beta(e))$ implies
  that this is a Losev-Manin curve, with orbifold points of order
  $\lcm(\Mon(e))$ at marked points and nodes. This construction works
  in families, so it defines a map
  $\widetilde\Mbar_{\Gamma_e}\to\Mbar^{\lcm(\Mon(e))}_{v_1|{\mov(e)}|v_2}$,
  which is equivariant by definition with respect to the action of
  $\prod_{\eta\in\Mov(e)}\mu_{\beta_\eta(e)}\wr S_e$. This gives a map
  \begin{align*}
    \Phi:\Mbar_{\Gamma_e}\to\left[\Mbar^{\lcm(\Mon(e))}_{v_1|{\mov(e)}|v_2}\Big/\left(\prod_{\eta\in\Mov(e)}\mu_{\beta_\eta(e)}\wr
        S_e\right)\right].
  \end{align*}

  \medskip

  We now construct an inverse to this map. Let
  $(C,b_{v_1},b_1,\ldots,b_{{\mov(e)}},b_{v_2})\in\Mbar^{\lcm(\Mon(e))}_{v_1|{\mov(e)}|v_2}$
  be a Losev-Manin curve whose points are indexed by the multiset
  $\Mov(e)$. Fix a curve $C'=\bigsqcup_{\eta\in\Mon(e)}C_\eta'$ with
  \'etale maps $\rho_\eta:C_\eta'\to C$ of degree $\eta.$ This may be
  done uniquely up to isomorphism. Also, uniquely up to isomorphism
  (of $C'$ commuting with $\rho:C'\to C$), for each
  $\eta\in\Mov(e)\subseteq\Mon(e)$ we may choose a preimage point
  $b_\eta'\in C_\eta'$ of the corresponding marked point $b_\eta\in
  C$. Finally, there is a unique map $f':C'\to\P^r$ that sends:
  \begin{itemize}
  \item $C_\eta'$ to a $T$-fixed point, for $\eta\not\in\Mov(e)$,
  \item $C_\eta'$ to $L_{e}$ with degree $\beta_\eta(e)$, with
    $b_\eta'$ mapping to $P_e$, $\rho^{-1}(b_{v_1})$ mapping to $P_{i^{\mov}(v_1,e)}$ and
    $\rho^{-1}(b_{v_2})$ mapping to $P_{i^{\mov}(v_2,e)}$, for $\eta\in\Mov(e)$.
  \end{itemize}
  Again, this works in families, and defines a map
  $\tilde\Theta:\Mbar^{\lcm(\Mon(e))}_{v_1|{\mov(e)}|v_2}\to\Mbar_{\Gamma_e},$
  which we claim is invariant under the action of
  $\prod_{\eta\in\Mov(e)}\mu_{\beta_\eta(e)}\wr S_e$. Indeed, acting by
  $e^{2\pi i/q(e)}$ on $b_\eta$ translates the preimage $b_\eta'$ by
  some power of $e^{2\pi i/\beta_\eta(e)}$, and commutes with $f'$. Thus
  $\tilde\Theta$ descends to a
  map $$\Theta:\left[\Mbar^{\lcm(\Mon(e))}_{v_1|{\mov(e)}|v_2}\Big/\left(\prod_{\eta\in\Mov(e)}\mu_{\beta_\eta(e)}\wr
      S_e\right)\right]\to\Mbar_{\Gamma_e},$$ which is by
  construction an inverse to $\Phi.$
\end{proof}
\begin{cor}\label{cor:UniversalCurve}
  The $\left(\prod_{\eta\in\Mov(e)}\mu_{\beta_\eta(e)}\wr S_e\right)$-action on
  $\Mbar_{v_1|{\mov(e)}|v_2}^{\lcm(\Mon(e))}$ extends to the
  universal curve, so we have a universal curve on $\Mbar_e$, and by
  gluing, a universal curve on the left side of
  (\ref{eq:ProductFixedLocus}). The isomorphism of
  \ref{thm:FixedLocusLosevManin} naturally identifies this with the
  universal curve on $\Mbar_{\Gamma}.$
\end{cor}
\begin{proof}
  The first statement is by definition of the action, and the second
  is immediate from the proof of Theorem \ref{thm:FixedLocusLosevManin}.
\end{proof}
\begin{rem}\label{rem:ConnectedComponentsOfFixedLocus}
  Theorem \ref{thm:FixedLocusLosevManin} shows in particular that
  $\Mbar_{\Gamma_e}$ is irreducible, so connected components of
  $(\Mbar_{g,n}(\Sym^d\P^r,\beta))^T$ are indexed by minimal decorated
  graphs with the additional data of a connected component of
  $\Mbar_{g,\vec{\Mon}(v)}(BS_{\VEval(v)},0)$ for each $v$. (These
  connected components in turn can computed using elementary group
  theory.)
\end{rem}
\begin{notation}\label{notation:PsiClasses}
  For a special flag $(v,e)\in F(\Gamma),$ we denote by
  $\psi^{\Mbar_e}_v$ the $\psi$-class on $\Mbar_e$ at the point
  labeled by $v$. If $v\in V^S({\Gamma}),$ we denote by
  $\psi^{\Mbar_v}_{e}$ the $\psi$-class on $\Mbar_v$ at the marked
  point $\xi(v,e)$. We use the same notation for the
  $\bar{\psi}$-classes.
\end{notation}

\section{The virtual normal bundle and virtual fundamental class of
  \texorpdfstring{$\Mbar_{\Gamma}$}{M\_{}Gamma}}\label{sec:VirtualNormalBundle}
In this section we compute the Euler class of the virtual normal
bundle to $\Mbar_{\Gamma},$ and show that the virtual
fundamental class of $\Mbar_{\Gamma}$ is equal to its
fundamental class. Some of the arguments are ``classical,'' and we
will refer the reader to \cite{Liu2013} for these.

In this section we fix
$\Gamma\in\Graphs_{g,n}^{\min}(\Sym^d\P^r,\beta)$. Let
$\pi:\mathcal{C}\to\Mbar_{\Gamma}$ and
$\rho:\mathcal{C}'\to\mathcal{C}$ denote the universal curve and
universal \'etale cover, respectively:
\begin{center}
  \begin{tikzcd}
    \mathcal{C}'\arrow{r}{f'}\arrow{d}{\rho}&\P^r\\
    \mathcal{C}\arrow{r}{f}\arrow{d}{\pi}&\Sym^d\P^r\\
    \Mbar_{\Gamma}
  \end{tikzcd}
\end{center}

By a standard argument (see \cite{Liu2013}), we have an exact sequence
of $T$-equivariant sheaves on $\Mbar_{g,n+1}(\Sym^d\P^r,\beta)$ giving
the perfect obstruction theory\footnote{We will always use the
  notation in \eqref{eq:ObstructionSequence} for higher direct image
  sheaves, writing e.g. $R^i\pi_*(\mathcal{C},f^*T\Sym^d\P^r)$ instead
  of $R^i\pi_*f^*T\Sym^d\P^r$. This is because we will restrict $\pi$
  to various substacks of $\mathcal{C}$, and wish to avoid confusion.}
\begin{align}\label{eq:ObstructionSequence}
  0&\to\Aut(\mathcal{C})\to R^0\pi_*(\mathcal{C},f^*T\Sym^d\P^r)\to\Defm(\mathcal{C},f)\to\\
  &\to\Defm(\mathcal{C})\to
  R^1\pi_*(\mathcal{C},f^*T\Sym^d\P^r)\to\Obs(\mathcal{C},f)\to0,\nonumber
\end{align}
where $\Aut(\mathcal{C})$ (resp. $\Defm(\mathcal{C})$) is the sheaf on
$\Mbar_{g,n+1}(\Sym^d\P^r)$ of infinitesimal automorphisms
(resp. deformations) of the marked source curve $\mathcal{C}$. (See
\cite{Liu2013} for rigorous definitions.) For
$(f:C\to\Sym^d\P^r)\in\Mbar_{\Gamma},$ we also have a
normalization exact sequence computing the fibers of the middle terms:
\begin{align}\label{eq:NormalizationSequence}
  0&\to
  H^0(C,f^*T\Sym^d\P^r)\to\bigoplus_{\nu}H^0(C_\nu,f^*T\Sym^d\P^r)\to\bigoplus_\xi
  H^0(\xi,f^*T\Sym^d\P^r)\to\\
  &\to
  H^1(C,f^*T\Sym^d\P^r)\to\bigoplus_{\nu}H^1(C_\nu,f^*T\Sym^d\P^r)\to0,\nonumber
\end{align}
where $\nu$ runs over the set of irreducible components of $C$, and
$\xi$ runs over nodes of $C$. The sequences
\eqref{eq:ObstructionSequence} and \eqref{eq:NormalizationSequence}
each split as direct sums of two exact sequences: the $T$-fixed part
and the $T$-moving part. We use the notations
$\Aut(\mathcal{C})^{\fix}$ and $\Aut(\mathcal{C})^{\mov}$ (and
similar) to denote the $T$-fixed subsheaf or subspace and its
$T$-invariant complement. By definition (see
\cite{GraberPandharipande1999}), the \emph{Euler class of the virtual
  normal bundle} $e_T(N^{\vir}_{\Gamma})$ is
\begin{align}\label{eq:VirtualNormalBundle}
  \frac{e_T(\Defm(\mathcal{C},f)^{\mov})}{e_T(\Obs(\mathcal{C},f)^{\mov})}=\frac{e_T(\Defm(\mathcal{C})^{\mov})e_T(R^0\pi_*(\mathcal{C},f^*T\Sym^d\P^r)^{\mov})}{e_T(\Aut(\mathcal{C})^{\mov})e_T(R^1\pi_*(\mathcal{C},f^*T\Sym^d\P^r)^{\mov})}\in
  H^*_T(\Mbar_{\Gamma}),
\end{align}
and the \emph{virtual fundamental class}
$[\Mbar_{\Gamma}]^{\vir}$ of $\Mbar_{\Gamma}$ is
$e_T(\Obs(\mathcal{C},f)^{\fix}).$ We compute the various terms of
\eqref{eq:ObstructionSequence} and \eqref{eq:NormalizationSequence}
one by one. It is convenient to compute by pulling back to the
canonical $\Aut(\Gamma)$-cover $\Mbar^{\rig}_{\Gamma}$ of
$\Mbar_{\Gamma}$, so that the correspondence between $C$ and
$\Gamma$ is more concrete.

\medskip

\noindent \textbf{The sheaves $\Aut(\mathcal{C})$ and
  $\Defm(\mathcal{C})$.} In the toric case, from \cite{Liu2013} we have
\begin{align}\label{eq:AutCContribution}
  e_T(\Aut(\mathcal{C})^{\mov})=\prod_{v\in
    V^1({\Gamma})}e_T(T_{\xi(v,e_v)}C)=\prod_{v\in
    V^1({\Gamma})}\psi^{\Mbar_{e_v}}_v.
\end{align}
The same argument and answer apply here, using (Theorem
\ref{thm:GraphsAndStrata} and) the observation that combining edges
gives a natural identification of $V^1(\Gamma)$. Briefly, moving
automorphisms come from noncontracted components with only one special
point, and correspond to vector fields on such a component that are
nonvanishing at the nonspecial $T$-fixed point.

Similarly, in the toric case \cite{Liu2013} gives
\begin{align}\label{eq:DefCContribution}
  e_T(\Defm(\mathcal{C}))=\left(\prod_{\substack{v\in
        V^2(\Gamma)\\\text{$(v,e_v^1)$
          steady}}}(-\psi^{\Mbar_{e_v^1}}_v-\psi^{\Mbar_{e_v^2}}_v)\right)\left(\prod_{\substack{(v,e)\in
        F(\Gamma)\\v\in
        V^S({\Gamma})}}(-\psi^{\Mbar_v}_e-\psi^{\Mbar_e}_v)\right).
\end{align}
This is again correct in our case. The factors in
\eqref{eq:DefCContribution} come from smoothing nodes. (Classically,
the deformation space of a node is the tensor product of the tangent
spaces to the two branches.) Therefore the observation we need is that
the nodes that do not appear in \eqref{eq:DefCContribution} have
$T$-fixed deformation space. We will use the following notation.
\begin{Def}
  A node $\xi$ is called \emph{steady}\footnote{This is similar, but
    not identical, to the definition of a \emph{breaking node} from
    \cite{OkounkovPandharipande2010}.} if $T_\xi C_1\otimes T_\xi C_2$
  has a nontrivial torus action, where $C_1$ and $C_2$ are the
  branches of $\xi$.
\end{Def}
\begin{rem}
  Steady nodes are exactly those of the form $\xi(v,e)$ for $(v,e)$ a
  steady flag. By Theorem \ref{thm:GraphsAndStrata}, if
  $\Psi(f:C\to\Sym^d\P^r)=\Gamma$ (i.e. it is minimal), then all
  nodes of $C$ are steady nodes. Furthermore, the set of steady nodes
  is canonically identified for any two points of
  $\Mbar^{\rig}_{\Gamma}$.
\end{rem}
The factors in \eqref{eq:DefCContribution} are in correspondence with
steady nodes.

\medskip

\noindent \textbf{The bundles $R^0\pi_*(\mathcal{C},f^*T\Sym^d\P^r)$
  and $R^1\pi_*(\mathcal{C},f^*T\Sym^d\P^r)$.}  We use the sequence
\eqref{eq:NormalizationSequence}. The computation is similar to the
original one by Kontsevich \cite{Kontsevich1995} (and the orbifold
computations of Johnson \cite{Johnson2014} and Liu \cite{Liu2013}),
but requires some care due to the edge moduli spaces.

Note that normalization does not commute with base change, so
\eqref{eq:NormalizationSequence} cannot naively be applied to
commute $R^i\pi_*(\mathcal{C},f^*T\Sym^d\P^r)$. However, normalization
of steady nodes does commute with base change on
$\Mbar^{\rig}_{\Gamma}$, due to the canonical identification of
nodes above. Thus we have the sequence
\begin{align}\label{eq:NormalizationSteady}
  0&\to
  R^0\pi_*(\mathcal{C},f^*T\Sym^d\P^r)\to\bigoplus_{\underline{\nu}}R^0\pi_*(\mathcal{C}_{\underline{\nu}},f^*T\Sym^d\P^r)\to\bigoplus_\xi
  R^0\pi_*(\xi,f^*T\Sym^d\P^r)\to\\\nonumber
  &\to
  R^1\pi_*(\mathcal{C},f^*T\Sym^d\P^r)\to\bigoplus_{\underline{\nu}}R^1\pi_*(\mathcal{C}_{\underline{\nu}},f^*T\Sym^d\P^r)\to0,
\end{align}
where $\underline{\nu}$ runs over closures of maximal subcurves of $\mathcal{C}$
containing only non-steady nodes, and $\xi$ runs over steady
nodes. Observe that either
$\mathcal{C}_{\underline{\nu}}$ is contracted by $f$, or each fiber
$C_{\underline{\nu}}$ of $\mathcal{C}_{\underline{\nu}}$ contains only
noncontracted components.

By Section \ref{sec:SymmetricPowers}, we have
$$R^i\pi_*(\mathcal{C}_{\underline{\nu}},f^*T\Sym^d\P^r)=R^i\pi_*(\mathcal{C}_{\underline{\nu}},\rho_*(f')^*T\P^r)=R^i(\pi\circ\rho)_*(\mathcal{C_{\underline{\nu}}}',(f')^*T\P^r).$$
(The second equality follows from the fact that $\rho$ is \'etale,
hence $\rho_*$ is exact.) After an \'etale base change, we may
distinguish the connected components of fibers of
$\mathcal{C}_{\underline{\nu}}'\to\Mbar^{\rig}_{\Gamma}$. In
other words, we may write
\begin{align*}
  \mathcal{C}_{\underline{\nu}}'=\bigsqcup_\eta\mathcal{C}_{\underline{\nu},\eta}',
\end{align*}
where $\mathcal{C}_{\underline{\nu},\eta}'$ has connected fibers. Then
\begin{align}\label{eq:DefsSumOverEta}
  R^i\pi_*(\mathcal{C}_{\underline{\nu}}',(f')^*T\P^r)=\bigoplus_\eta
  R^i(\pi\circ\rho)_*(\mathcal{C}_{\underline{\nu},\eta}',(f')^*T\P^r).
\end{align}
If $\mathcal{C}_{\underline{\nu}}=\mathcal{C}_v$ is contracted, then
$(f')^*T\P^r$ is trivial on
$\mathcal{C}_{\underline{\nu},\eta}'$. Thus we have
\begin{align*}
  R^i(\pi\circ\rho)_*(\mathcal{C}_{\underline{\nu},\eta}',(f')^*T\P^r)\cong
  R^i(\pi\circ\rho)_*(\mathcal{C}_{\underline{\nu},\eta}',\O_{\mathcal{C}_{\underline{\nu},\eta}'})\otimes
  T_{P_{i(\eta)}}\P^r,
\end{align*}
where as $i(\eta)\in\{0,\ldots,r\}$ is the label of $\eta$,
i.e. $P_{i(\eta)}=f'(\mathcal{C}_{\underline{\nu},\eta}')$. In
particular,
\begin{align}\label{eq:ContractedFixedVanish}
  R^0\pi_*(\mathcal{C}_v,f^*T\Sym^d\P^r)^{\fix}=R^1\pi_*(\mathcal{C}_{\underline{\nu}},f^*T\Sym^d\P^r)^{\fix}=0.
\end{align}
The bundle $R^1\pi_*(\mathcal{C}_v,f^*T\Sym^d\P^r)^{\mov}$ is
nontrivial, and is isomorphic to a Hurwitz-Hodge bundle (see
\cite{Liu2013}, Section 7.5). However, note that
$e_T(R\pi_*(\mathcal{C}_v,f^*T\Sym^d\P^r))$ is the inverse of the
twisting class from \eqref{eq:TwistedConeDef}. We will use this fact
in Section \ref{sec:CharacterizationOfLagrangianCone} in our
characterization of $\mathcal{L}_{\Sym^d\P^r}$, and in Section
\ref{sec:MirrorTheorem} to apply the orbifold quantum Riemann-Roch
theorem.

Similarly for a steady node $\xi(v,e),$ we have
\begin{align}\label{eq:R0Nodes}
  R^0\pi_*(\xi(v,e),f^*T\Sym^d\P^r)^{\fix}&=0\nonumber\\
  R^0\pi_*(\xi(v,e),f^*T\Sym^d\P^r)^{\mov}&=T_{(\VEval(v),\Mon(v,e))}I\Sym^d\P^r=\bigoplus_{\eta\in\Mon(v,e)}T_{P_{i(\eta)}}\P^r.
\end{align}

Suppose $\mathcal{C}_{\underline{\nu}}$ is not contracted. The
components $\mathcal{C}_{\underline{\nu},\eta}'$ are in bijection with
$\Mon(e)$, where $e$ is the edge of $\Gamma$ corresponding to
$\mathcal{C}_{\underline{\nu}}$.) First, we argue that
$R^1(\pi\circ\rho)_*(\mathcal{C}_{\underline{\nu},\eta}',(f')^*T\P^r)$
vanishes for all $\eta$. 
The normalization exact sequence for a fiber
$C_{\underline{\nu},\eta}'$ reads:
\begin{align*}
  0&\to
  H^0(C_{\underline{\nu},\eta}',(f')^*T\P^r)\to\bigoplus_{\nu\in\underline{\nu}}H^0(C_{\nu,\eta}',(f')^*T\P^r)\to\bigoplus_\xi
  H^0(\xi,(f')^*T\P^r)\to\\
  &\to
  H^1(C_{\underline{\nu},\eta}',(f')^*T\P^r)\to\bigoplus_{\nu\in\underline{\nu}}H^1(C_{\nu,\eta}',(f')^*T\P^r)\to0,
\end{align*}
where we also denote by $\underline{\nu}$ the set indexing irreducible
components $C_\nu$ of $C_{\underline{\nu}}$ (equivalently, irreducible
components $C_{\nu,\eta}'$ of $C_{\underline{\nu},\eta}$). 
For each $\nu\in\underline{\nu},$ we have
\begin{align}\label{eq:NoncontractedFixedR1Vanish}
  H^1(C_{\nu},(f')^*T\P^r)=0
\end{align}
by convexity of $\P^r.$ We claim that the map
\begin{align*}
  \bigoplus_{\nu\in\underline{\nu}}H^0(C_{\nu,\eta}',(f')^*T\P^r)\to\bigoplus_\xi
  H^0(\xi,(f')^*T\P^r)
\end{align*}
is surjective, so that $H^1(C_{\underline{\nu},\eta}',(f')^*T\P^r)=0.$
(The map takes the difference of the sections on the two branches of a
node.) If $C_{\underline{\nu},\eta}'$ has a component
$C_{\nu_0,\eta}'$ not contracted by $f'$, there is at most one, by
Lemma \ref{lem:UnionOfChainsMapping}. On any other component
$C_{\nu,\eta}'$, we have $(f')^*T\P^r\cong\O_{C_{\nu,\eta}'}\otimes
T\P^r$, i.e. $H^0(C_{\nu,\eta}',\O_{C_{\nu,\eta}'}\otimes T\P^r)\cong
T\P^r$. Fix an arbitrary section $s\in
H^0(C_{\nu_0,\eta}',(f')^*T\P^r)$. Then ``working outward'' from
$C_{\nu_0,\theta}'$ shows that the map is surjective. The case where
$f'$ contracts $C_{\underline{\nu},\eta}'$ is similar and
simpler. 

Next, we compute
$R^0(\pi\circ\rho)_*(\mathcal{C}_{\underline{\nu},\eta}',(f')^*T\P^r).$
If $\mathcal{C}_{\underline{\nu},\eta}'$ is contracted, $(f')^*T\P^r$
is trivial and we have
\begin{align*}
  R^0(\pi\circ\rho)_*(\mathcal{C}_{\underline{\nu},\eta}',(f')^*T\P^r)\cong
  T\P^r\otimes\O_{\Mbar^{\rig}_{\Gamma}}
\end{align*}
by properness of $\pi\circ\rho.$ Suppose
$\mathcal{C}_{\underline{\nu},\eta}'$ is not contracted. Consider the
Stein factorization of $f'|_{\mathcal{C}_{\underline{\nu},\eta}'}$
relative to $\pi\circ\rho$:
\begin{center}
  \begin{tikzcd}
    \mathcal{C}_{\underline{\nu},\eta}'\arrow{r}{\sf}\arrow[bend left]{rr}{f'}\arrow[swap]{d}{\pi\circ\rho}
    &\overline{\mathcal{C}_{\underline{\nu},\eta}'}\arrow[swap]{r}{f''}
    \arrow{dl}{\bar{\pi\circ\rho}}&\P^r\\
    \Mbar^{\rig}_{\Gamma}
  \end{tikzcd}
\end{center}
If $(f:C\to\Sym^d\P^r)$ is in the dense open substack
$\Psi^{-1}(\Gamma)\subseteq\Mbar^{\rig}_{\Gamma}$,
then $C_{\underline{\nu}}$ is irreducible, hence so is
$C_{\underline{\nu},\eta}'$. This, with the fact that
$\mathcal{C}_{\underline{\nu},\eta}'$ is not contracted, implies that
$\sf$ is birational. By the projection formula for coherent sheaves,
\begin{align*}
  (\pi\circ\rho)_*(f')^*T\P^r&=(\pi\circ\rho)_*\sf^*(f'')^*T\P^r\\
  &=(\bar{\pi\circ\rho})_*\sf_*\sf^*(f'')^*T\P^r\\
  &=(\bar{\pi\circ\rho})_*((f'')^*T\P^r\otimes\sf_*\O_{C_{\underline{\nu},\eta}'})\\
  &=(\bar{\pi\circ\rho})_*(f'')^*T\P^r.
\end{align*}
After an \'etale base change on $\Mbar^{\rig}_{\Gamma}$, the map
$f''$ trivializes
$\overline{\mathcal{C}_{\underline{\nu},\eta}'}$. Thus
$R^0(\bar{\pi\circ\rho})_*(\overline{\mathcal{C}_{\underline{\nu},\eta}'},(f'')^*T\P^r)$
is a trivial vector bundle. Calculation of the $T$-weights of this
vector bundle is identical to Kontsevich's calculation in Section
3.3.4 of \cite{Kontsevich1995}, which uses the Euler sequence on
$\P^r$. The weights are
\begin{align}\label{eq:TWeights}
  \frac{A}{\beta_\eta(e)}\alpha_{i^{\mov}(v_1,e)}+\frac{B}{\beta_\eta(e)}\alpha_{i^{\mov}(v_2,e)}-\alpha_i,
\end{align}
where $0\le A,B\le\beta_\eta(e)$, $A+B=\beta_\eta(e)$, and
$i\in\{0,\ldots,r\}$. Note that this is zero exactly when $A=0$ and
$i=i^{\mov}(v_2,e)$, or $B=0$ and $i=i^{\mov}(v_1,e)$. (These factors
contribute to
$e_T(R^0(\bar{\pi\circ\rho})_*(\overline{\mathcal{C}_{\underline{\nu},\eta}'},(f'')^*T\P^r)^{\fix}).$)
Putting together \eqref{eq:R0Nodes} and \eqref{eq:TWeights}, for
$\underline\nu$ noncontracted, the Euler class
$e_T(R^0(\bar{\pi\circ\rho})_*(\overline{\mathcal{C}_{\underline{\nu}}'},(f'')^*T\P^r)^{\mov})$
is equal to
\begin{align}\label{eq:NoncontractedR0}
  \left(\prod_{\eta\in\Stat(e)}\prod_{i\ne
      i(\eta)}(\alpha_{i(\eta)}-\alpha_i)\right)\prod_{\eta\in\Mov(e)}\prod_{\substack{A+B=\beta_\eta(e)\\0\le
        i\le
        r\\(A,i)\ne(0,i^{\mov}(v_2,e))\\(B,i)\ne(0,i^{\mov}(v_1,e))}}\left(\frac{A}{\beta_\eta(e)}\alpha_{i^{\mov}(v_1,e)}+\frac{B}{\beta_\eta(e)}\alpha_{i^{\mov}(v_2,e)}-\alpha_i\right).
\end{align}

\noindent\textbf{Summary.} 
We collect the arguments of this section in the following two
statements.
\begin{prop}\label{prop:VirtualFundamentalClass}
  For any minimal decorated graph $\Gamma$,
  $\Mbar_{\Gamma}$ is smooth, and the virtual fundamental class
  is equal to the fundamental class.
\end{prop}
\begin{prop}\label{prop:VirtualNormalBundle}
  The equivariant Euler class $e_T(N^{\vir}_{\Mbar_{\Gamma}})$
  of the virtual normal bundle to $\Mbar_{\Gamma}$ is
  \begin{align*}
    &\left(\frac{\prod_{v\in
          V^2(\Gamma)}(-\psi^{\Mbar_{e_v^1}}_v-\psi^{\Mbar_{e_v^2}}_v)\prod_{\substack{(v,e)\in
            F(\Gamma)\\v\in
            V^S({\Gamma})}}(-\psi^{\Mbar_v}_e-\psi^{\Mbar_e}_v)}{\prod_{v\in
          V^1({\Gamma})}\psi^{\Mbar_{e_v}}_v}\right)\\
    &\quad\cdot\prod_{e\in
      E(\Gamma)}\left(\rule{0cm}{1.7cm}\right.\left(\rule{0cm}{.75cm}\right.\prod_{\substack{\eta\in\Stat(e)\\i\ne i(\eta)}}
    (\alpha_{i(\eta)}-\alpha_i)\left.\rule{0cm}{.75cm}\right)
    \prod_{\substack{\eta\in\Mov(e)\\A+B=\beta_\eta(e)\\0\le i\le r\\(A,i)\ne(0,i^{\mov}(v_2,e))\\(B,i)\ne(0,i^{\mov}(v_1,e))}}\left(\frac{A}{\beta_\eta(e)}\alpha_{i^{\mov}(v_1,e)}+\frac{B}{\beta_\eta(e)}\alpha_{i^{\mov}(v_2,e)}-\alpha_i\right)\left.\rule{0cm}{1.7cm}\right)\\
    &\quad\quad\cdot\left(\frac{\prod_{v\in V^1(\Gamma)\cup V^{1,1}(\Gamma)\cup
      V^2(\Gamma)}e_T(T_{(\VEval(v),\Mon(v))}I\Sym^d\P^r)}{\prod_{(v,e)\in
      F(\Gamma)}e_T(T_{(\VEval(v),\Mon(v,e))}I\Sym^d\P^r)}\right)\\
    &\quad\quad\quad\cdot\left(\prod_{v\in
      V^S(\Gamma)}e_T(R\pi_*(C_v,f^*T\Sym^d\P^r))^{\mov}\right).
  \end{align*}
\end{prop}
\begin{proof}[Proof of Proposition \ref{prop:VirtualFundamentalClass}]
  Recall from Theorem \ref{thm:VirtualLocalization} that the virtual
  fundamental class of $\Mbar_{\Gamma}$ is obtained from the
  fixed part of the perfect obstruction theory on
  $\Mbar_{g,n}(\Sym^d\P^r,\beta)$. By \eqref{eq:R0Nodes}, the
  fixed part of $\bigoplus_\xi R^0\pi_*(\xi,f^*T\Sym^d\P^r)$ is
  zero. Thus by \eqref{eq:NormalizationSteady},
  \begin{align*}
    R^1\pi_*(\mathcal{C},f^*T\Sym^d\P^r)\cong\bigoplus_{\underline{\nu}}R^1\pi_*(\mathcal{C}_{\underline{\nu}},f^*T\Sym^d\P^r).
  \end{align*}
  But we showed, in \eqref{eq:ContractedFixedVanish} and
  \eqref{eq:NoncontractedFixedR1Vanish}, that
  $\bigoplus_{\underline{\nu}}R^1\pi_*(\mathcal{C}_{\underline{\nu}},f^*T\Sym^d\P^r)$
  has no fixed part. Thus $R^1\pi_*(\mathcal{C},f^*T\Sym^d\P^r)$ has
  no fixed part. By Proposition 5.5 of \cite{BehrendFantechi1997}, the
  Proposition follows. (Smoothness already followed easily from
  Theorem \ref{thm:FixedLocusLosevManin}.)
\end{proof}
\begin{proof}[Proof of Proposition \ref{prop:VirtualNormalBundle}]
  The first line is the contribution from $\Defm(\mathcal{C})^{\mov}$
  and $\Aut(\mathcal{C})^{\mov},$ from \eqref{eq:AutCContribution} and
  \eqref{eq:DefCContribution}. The second line is the contribution of
  noncontracted components to
  $R\pi_*(\mathcal{C},f^*T\Sym^d\P^r),$ from
  \eqref{eq:NoncontractedR0} and
  \eqref{eq:NoncontractedFixedR1Vanish}. The third line is the
  contribution of steady nodes to
  $R\pi_*(\mathcal{C},f^*T\Sym^d\P^r),$ from \eqref{eq:R0Nodes}. (The
  numerator corrects for the fact that $F(\Gamma)$ overcounts the
  steady nodes.) The last line is the contribution of contracted
  components to
  $R\pi_*(\mathcal{C},f^*T\Sym^d\P^r)^{\mov},$ by definition.
\end{proof}

\begin{thm}\label{thm:Algorithm}
  The results of this section, together with Corollary
  \ref{cor:GraphOpenClosedSubstack} and Theorem
  \ref{thm:FixedLocusLosevManin}, provide an algorithm to compute any
  Gromov-Witten invariant of $\Sym^d\P^r$ (for any $d$) in terms of
  Hurwitz-Hodge integrals, i.e. twisted Gromov-Witten invariants of
  $BG$ for $G$ a product of symmetric groups.
\end{thm}
\begin{proof}
  Applying the virtual localization theorem
  \ref{thm:VirtualLocalization}, a genus-$g$ Gromov-Witten invariant
  of $\Sym^d\P^r$ is expressed as a sum
  $$\sum_{{\Gamma}\in
    G_{g,n}^{\min}(\Sym^d\P^r,\beta)}\int_{\Mbar_{{\Gamma}}}\frac{\iota^*\alpha}{e_T(N_{\Mbar_{{\Gamma}}}^{\vir})}.$$
  By Theorem \ref{thm:FixedLocusLosevManin}, $\Mbar_{{\Gamma}}$
  is a finite cover of a product of Losev-Manin spaces $\Mbar_e$
  (Section \ref{sec:LosevManin}) and spaces
  $\Mbar_v=\Mbar_{g_v,\vec\Mon(v)}(BS_{\VEval(v)},0)$ of admissible
  covers. The factors $\Mbar_e$ can be integrated over using Lemma
  \ref{lem:PsiPullback}, since the only cohomology classes in the
  integrands are $\psi$ classes at the two distinguished marked points
  (cf. \eqref{eqn:PsiPullback} in the proof of Theorem
  \ref{thm:ConeCharacterization}). The remaining integrals are over
  the factors $\Mbar_v.$ The integrand contains the factor
  $$\prod_{v\in
    V^S({\Gamma})}\frac{1}{e_T(R\pi_*(C_v,f^*T\Sym^d\P^r))^{\mov}},$$
  as well as $\psi$ classes and classes pulled back along evaluation
  maps, and is thus a twisted Gromov-Witten invariant of
  $BS_{VEval(v)}.$
\end{proof}

\section{Characterization of the Givental
  cone \texorpdfstring{$\mathcal{L}_{\Sym^d\P^r}$}{L\_\{Sym\^{}d P\^{}r\}}}\label{sec:CharacterizationOfLagrangianCone}
In this section, we apply the results of Sections
\ref{sec:TFixedStableMaps} and \ref{sec:VirtualNormalBundle} to give a
criterion (Theorem \ref{thm:ConeCharacterization}) that exactly
determines whether a given power series lies on the Givental cone
$\mathcal{L}_{\Sym^d\P^r}$. For the rest of the paper, we work only in
genus zero, so we refer to ``decorated trees'' rather than ``decorated
graphs.''
\begin{Def}\label{Def:1EdgeTrees}
  Fix $(\mu,\sigma)\in(I\Sym^d\P^r)^T.$ Let
  $\Upsilon(\mu,\sigma)\subseteq\Graphs_{0,2}(\Sym^d\P^r,\beta)$ be the
  set of 1-edge decorated trees
  \raisebox{-5pt}{
    \begin{tikzpicture}
      \draw (1,0.1) node {$\kappa=$};
      \draw (2,0) node {$\bullet$} node[left] {$v_1$};
      \draw (2,0)--(3,0);
      \draw (2.5,0) node[above] {$e$};
      \draw (3,0) node {$\bullet$} node[right] {$v_2$,};
    \end{tikzpicture}
  } with $g_{v_1}=g_{v_2}=0$, marking set $\{b_{n+1},b_{\bullet}\}$,
  with $\Mark({n+1})=v_1$ and $\Mark({\bullet})=v_2$, such that
  $\mu=\VEval(v_1)$ and $\sigma=\Mon(v_1,e).$
\end{Def}
\begin{notation}\label{not:WeightOfKappa}
  For $\kappa\in\Upsilon(\mu,\sigma)$, we write (using the notation of Definition \ref{Def:DecoratedGraph}):
  \begin{multicols}{2}
    \begin{itemize}
    \item $q(\kappa):=q(e)$,
    \item $\Mov(\kappa):=\Mov(e),$
    \item $\mov(\kappa):=\mov(e),$
    \item $\Stat(\kappa):=\Stat(e),$
    \item for $\eta\in\Mov({\kappa})$,
      $\beta_\eta({\kappa}):=\beta_\eta(e)=q(e)\cdot\eta$,
    \item $\beta(\kappa)=\sum_{\eta\in\Mov({\kappa})}\beta_\eta(\kappa)$
    \item $i_1^{\mov}(\kappa):=i^{\mov}(v_1,e)$,
    \item $i_2^{\mov}(\kappa):=i^{\mov}(v_2,e)$,
    \item $\mu'(\kappa):=\VEval(v_2)$,
    \item $\sigma'(\kappa):=\Mon(v_2,e)$, and
    \item $r(\kappa):=r(v_1,e)=r(v_2,e)=r_{n+1}$.
    \end{itemize}
  \end{multicols}
  
  We also define:
  $$w(\kappa):=\frac{\alpha_{i_1^{\mov}(\kappa)}-\alpha_{i_2^{\mov}(\kappa)}}{q(\kappa)}\in
  H^2_T(\Spec\C).$$
  \begin{rem}\label{rem:QCoarse}
    Note that $w({\kappa})$ is equal to the $T$-weight of the
    tangent space to the \textit{coarse moduli space} of the source
    curve $C$ at $b_{n+1}$; this is because $q({\kappa})$ is
    defined via coordinates on this coarse moduli space (see Lemma
    \ref{lem:IrreducibleFixedMap}).
  \end{rem}
\end{notation}
\begin{Def}\label{Def:RecursionCoefficient}
  Let $\kappa\in\Upsilon(\mu,\sigma)$ and let $a\in\Z_{>0}$ We
  define the recursion coefficient
  \begin{align*}
    \RC(\kappa,a)&=\frac{(-1)^{{\mov(\kappa)}-a}}{q(\kappa)^{{\mov(\kappa)}}}\binom{\sigma_{i_1^{\mov}(\kappa)}}{\Mov(\kappa)}\binom{{\mov(\kappa)}-1}{a-1}\\
    &\quad\cdot\frac{1}{\prod_{\eta\in\Mov(\kappa)}\prod_{\substack{1\le
          B\le\beta_\eta({\kappa})\\0\le
          i\le r\\(B,i)\ne(\beta_\eta({\kappa}),i_2^{\mov}(\kappa))}}\left(\frac{\beta_\eta({\kappa})-B}{\beta_\eta({\kappa})}\alpha_{i_1^{\mov}(\kappa)}+\frac{B}{\beta_\eta({\kappa})}\alpha_{i_2^{\mov}(\kappa)}-\alpha_i\right)},
  \end{align*}
  where
  $\binom{\sigma_{i_1^{\mov}(\kappa)}}{\Mov(\kappa)}$ is
  the number of ways of choosing $\Mov(\kappa)$ as a
  subpartition of $\sigma_{i_1^{\mov}(\kappa)}$ with specified
  parts.
\end{Def}
The following theorem and its proof are in the same spirit as Theorem 41
of \cite{CoatesCortiIritaniTseng2015}, which in turn is adapted from
Theorem 2 of \cite{Brown2014}.
\begin{thm}\label{thm:ConeCharacterization}
  Let $\mathbf{f}$ be an element of $\mathcal{H}[[x]]$ such that
  $\mathbf{f}|_{Q=x=0}=-1z,$ where $1$ denotes the fundamental class
  of $\Sym^d\P^r\subseteq I\Sym^d\P^r$. Then $\mathbf{f}$ is a
  $\Lambda_{\nov}^T[[x]]$-valued point of $\mathcal{L}_{\Sym^d\P^r}$
  if and only if for each $T$-fixed point $(\mu,\sigma)\in
  I\Sym^d\P^r$, the following three conditions hold:
  \begin{enumerate}[label=(\textbf{\Roman*})]
  \item The restriction $\mathbf{f}_{(\mu,\sigma)}$ along
    $\iota_{(\mu,\sigma)}:(\mu,\sigma)\into I\Sym^d\P^r$ is a power
    series in $Q$ and $x$, such that each coefficient of this power
    series is an element of $H^*_{T,\loc}(\bullet)(z).$ Each
    coefficient is regular in $z$ except for possible poles at $z=0,$
    $z=\infty,$ and\label{item:Poles}
    \begin{align*}
      z\in\{w(\kappa):\kappa\in\Upsilon(\mu,\sigma)\}.
    \end{align*}
  \item The Laurent coefficients of $\mathbf{f}_{(\mu,\sigma)}$ at the poles (other than $z=0$ and
    $z=\infty$) satisfy the recursion relation:
  \begin{align}\label{eq:LaurentRecursion}
    \Laur(\mathbf{f}_{\mu,\sigma},(w-z)^{-a})=\sum_{\substack{\kappa\in\Upsilon_{(\mu,\sigma)}\\w(\kappa)=w\\\mov(\kappa)\ge
        a}}Q^{\beta(\kappa)}\RC(\kappa,a)\Laur(\mathbf{f}_{(\mu'(\kappa),\sigma'(\kappa))},(w-z)^{{\mov(\kappa)}-a})
  \end{align}
  for $a>0$, and\label{item:Recursion}
\item The restriction $\mathbf{f}_\mu$ along $\iota_\mu:I\mu\into
  I\Sym^d\P^r$ is a $\Lambda_{\nov}^T[[x]]$-valued point of
  $\mathcal{L}_\mu^{\tw}.$\label{item:TwistedCone}
  \end{enumerate}
\end{thm}
\begin{rem}
  In \ref{item:TwistedCone}, $\Lambda_{\nov}^T$ is the equivariant
  Novikov ring associated to $\Sym^d\P^r,$ not $\mu$. In other words,
  $\Lambda_{\nov}^T[[x]]=H^*_{CR,T,\loc}(\mu)[[Q,x]].$
\end{rem}
\begin{rem}
  The major difference between Theorem \ref{thm:ConeCharacterization}
  and the corresponding theorems in \cite{CoatesCortiIritaniTseng2015}
  and \cite{Brown2014} is that condition \ref{item:Recursion} gives a
  recursive relation for \emph{all} negative-exponent Laurent
  coefficients at $z=w(\kappa)$, in terms of
  nonnegative-exponent ones. In \cite{CoatesCortiIritaniTseng2015} and
  \cite{Brown2014}, only stacks with isolated 1-dimensional $T$-orbits
  are considered.  Thus in that case, the poles at $z=w(\kappa)$
  are simple, and a recursive relation is given for their residues.
\end{rem}
\begin{proof}
 Let $\mathbf{f}$ be a
  $\Lambda_{\nov}^T[[x]]$-valued point of $\mathcal{L}_{\Sym^d\P^r}$. By definition, we can write
  \begin{align*}
    \mathbf{f}&=-1z+\mathbf{t}(z)+\sum_{n=0}^\infty\sum_{\beta=0}^\infty\sum_{\phi}\frac{Q^\beta}{n!}\left\langle
                \mathbf{t}(\bar{\psi}),\ldots,\mathbf{t}(\bar{\psi}),\frac{\gamma_\phi}{-z-\bar{\psi}}\right\rangle_{0,n+1,\beta}^{\Sym^d\P^r,T}\gamma^\phi\\
              &=-1z+\mathbf{t}(z)+\sum_{n=0}^\infty\sum_{\beta=0}^\infty\frac{Q^\beta}{n!}(\ev_{n+1})_*\left(\prod_{j=1}^n\ev_j^*\mathbf{t}(\bar{\psi})\cup\frac{1}{-z-\bar{\psi}}\cap[\Mbar_{0,n+1}(\Sym^d\P^r,\beta)]^{\vir}\right)
  \end{align*}
  for $\mathbf{t}(z)\in\H^+[[x]]$ with $\mathbf{t}|_{Q=x=0}=0.$ The
  restriction $\mathbf{f}_{(\mu,\sigma)}$ is then
  \begin{align*}
    -\delta_{\sigma=(1,\ldots,1)}z&+\iota_{(\mu,\sigma)}^*\mathbf{t}(z)\\
    &+\sum_{n=0}^\infty\sum_{\beta=0}^\infty\frac{Q^\beta}{n!}\iota_{(\mu,\sigma)}^*\left((\ev_{n+1})_*\left(\prod_{j=1}^n\ev_j^*\mathbf{t}(\bar{\psi})\cup\frac{1}{-z-\bar{\psi}}\cap[\Mbar_{0,n+1}(\Sym^d\P^r,\beta)]^{\vir}\right)\right).
  \end{align*}
  Using the projection formula, we write
  \begin{align*}
    \iota_{(\mu,\sigma)}^*&\left((\ev_{n+1})_*\left(\prod_{j=1}^n\ev_j^*\mathbf{t}(\bar{\psi}_j)\cup\frac{1}{-z-\bar{\psi}_{n+1}}\cap[\Mbar_{0,n+1}(\Sym^d\P^r,\beta)]^{\vir}\right)\right)\\
    &=\abs{C_{\mu}(\sigma)}\int_{\Sym^d\P^r}(\iota_{(\mu,\sigma)})_*\iota_{(\mu,\sigma)}^*\left((\ev_{n+1})_*\left(\prod_{j=1}^n\ev_j^*\mathbf{t}(\bar{\psi}_j)\cup\frac{1}{-z-\bar{\psi}_{n+1}}\cap[\Mbar_{0,n+1}(\Sym^d\P^r,\beta)]^{\vir}\right)\right)\\
                          &=\abs{C_{\mu}(\sigma)}\int_{\Sym^d\P^r}[(\mu,\sigma)]\cup\left((\ev_{n+1})_*\left(\prod_{j=1}^n\ev_j^*\mathbf{t}(\bar{\psi}_j)\cup\frac{1}{-z-\bar{\psi}_{n+1}}\cap[\Mbar_{0,n+1}(\Sym^d\P^r,\beta)]^{\vir}\right)\right)
  \end{align*}
  \begin{align*}
    &=\abs{C_{\mu}(\sigma)}\int_{\Sym^d\P^r}\left((\ev_{n+1})_*\left(\prod_{j=1}^n\ev_j^*\mathbf{t}(\bar{\psi}_j)\cup\frac{\ev_{n+1}^*([(\mu,\sigma)])}{-z-\bar{\psi}_{n+1}}\cap[\Mbar_{0,n+1}(\Sym^d\P^r,\beta)]^{\vir}\right)\right)\\
    &=\abs{C_{\mu}(\sigma)}\int_{[\Mbar_{0,n+1}(\Sym^d\P^r,\beta)]^{\vir}}\left(\prod_{j=1}^n\ev_j^*\mathbf{t}(\bar{\psi}_j)\cup\frac{\ev_{n+1}^*([(\mu,\sigma)])}{-z-\bar{\psi}_{n+1}}\right)\\
    &=\abs{C_{\mu}(\sigma)}\left\langle\mathbf{t}(\bar{\psi}),\ldots,\mathbf{t}(\bar{\psi}),\frac{[(\mu,\sigma)]}{-z-\bar{\psi}}\right\rangle_{0,n+1,\beta}^{\Sym^d\P^r,T}.
  \end{align*}
  The first equality uses the identification of
  $\int_{\Sym^d\P^r}\circ\iota_{(\mu,\sigma)}$ with the identity map
  $\Spec\C\to\Spec\C$ on coarse moduli spaces, and the factor
  $\abs{C_{\mu}(\sigma)}$ corrects for the isotropy at
  $(\mu,\sigma)\in I\Sym^d\P^r$. (Recall that $C_\mu(\sigma)$ denotes
  the centralizer of any element of $\sigma$ in
  $G_\mu$.) 
  In summary,
  \begin{align}\label{eq:RestrictionOfF}
    \mathbf{f}_{(\mu,\sigma)}&=-\delta_{\sigma=(1,\ldots,1)}z+\mathbf{t}_{(\mu,\sigma)}(z)\\
    &\quad\quad\quad+\sum_{n=0}^\infty\sum_{\beta=0}^\infty\frac{\abs{C_{\mu}(\sigma)}Q^\beta}{n!}\left\langle\mathbf{t}(\bar{\psi}),\ldots,\mathbf{t}(\bar{\psi}),\frac{[(\mu,\sigma)]}{-z-\bar{\psi}}\right\rangle_{0,n+1,\beta}^{\Sym^d\P^r,T},\nonumber
  \end{align}
  where
  $\mathbf{t}_{(\mu,\sigma)}(z):=\iota_{(\mu,\sigma)}^*\mathbf{t}(z).$ 
  Now we calculate \eqref{eq:RestrictionOfF} by virtual torus
  localization (see Theorem \ref{thm:VirtualLocalization}). Namely, we
  may write
  \begin{align}\label{eq:CoeffOfFAsFiniteSum}
    \abs{C_{\mu}(\sigma)}\left\langle\mathbf{t}(\bar{\psi}),\ldots,\mathbf{t}(\bar{\psi}),\frac{[(\mu,\sigma)]}{-z-\bar{\psi}}\right\rangle_{0,n+1,\beta}^{\Sym^d\P^r,T}=\sum_{\Gamma\in\Graphs_{0,n+1}^{\min}(\Sym^d\P^r,\beta)}\Contr_{(\mu,\sigma)}(\Gamma).
  \end{align}
  We can partition $\Graphs_{0,n+1}^{\min}(\Sym^d\P^r,\beta)$ into
  three subsets:
  \begin{enumerate}[label=(\roman*)]
  \item $\Gamma$ such that
    $(\VEval(\Mark({n+1})),\Mon({n+1}))\ne(\mu,\sigma)$,\label{item:vanish}
  \item $\Gamma$ such that
    $(\VEval(\Mark({n+1})),\Mon({n+1}))=(\mu,\sigma)$ and
    $\Mark({n+1})\in V^{1,1}(\Gamma)$,\label{item:recursion} and
  \item $\Gamma$ such that
    $(\VEval(\Mark({n+1})),\Mon({n+1}))=(\mu,\sigma)$ and
    $\Mark({n+1})\in V^S(\Gamma)$. \label{item:initial}
  \end{enumerate}
  In some literature, e.g. \cite{CiocanFontanineKim2014}, decorated
  trees of type \ref{item:recursion} are called \emph{recursion type}
  and those of type \ref{item:initial} are called \emph{initial type}.
  (We will see below, however, that in our setup both types are used
  recursively.) Let $v_1:=\Mark({n+1})$.

  For a tree $\Gamma$ of type \ref{item:vanish}, the
  restriction $\ev_{n+1}^*([(\mu,\sigma)])$ vanishes, hence
  $\Contr_{(\mu,\sigma)}(\Gamma)=0$. For this reason, we may
  simplify our notation, and write
  $\Contr(\Gamma):=\Contr_{(\mu,\sigma)}(\Gamma),$ where
  $\mu=\VEval(\Mark({n+1}))$ and $\sigma=\Mon({n+1}).$

  If $\Gamma$ is a tree of type \ref{item:initial}, then by
  Theorem \ref{thm:FixedLocusLosevManin} and Corollary
  \ref{cor:UniversalCurve}, $\bar{\psi}_{n+1}$ is pulled back from
  $\Mbar_{0,\vec\Mon(v_1)}(BG_{\mu},0),$ where $G_\mu$ is the isotropy
  group of $\mu$. Since this stack parametrizes maps that factor
  through the fixed point $\mu$, the action of $T$ is trivial, hence
  $$H^*_{T,\loc}(\Mbar_{0,\vec\Mon(v_1)}(BG_{\mu},0))\cong
  H^*(\Mbar_{0,\vec\Mon(v_1)}(BG_{\mu},0))\otimes
  H^*_{T,\loc}(\bullet).$$ In particular, $\bar{\psi}_{n+1}$ is
  nilpotent. It follows that $\Contr(\Gamma)$ is a polynomial in
  $z^{-1}$, hence has a pole only at $z=0$.

  Finally, let $\Gamma$ be a tree of type
  \ref{item:recursion}. By \eqref{eq:VirtualIntegralFormula}, we have
  \begin{align}\label{eq:LocalizationContribution}
    \Contr(\Gamma)=\abs{C_{\mu}(\sigma)}\int_{[\Mbar_{\Gamma}]'}\frac{1}{e_T(N_{\Gamma}^{\vir})}\iota_{\Gamma}^*\left(\prod_{j=1}^n\ev_j^*\mathbf{t}(\bar{\psi})\cup\frac{\ev_{n+1}^*[(\mu,\sigma)]}{-z-\bar{\psi}_{n+1}}\right),
  \end{align}
  where $\iota_{\Gamma}$ is the inclusion
  $\Mbar_{\Gamma}\into\Mbar_{0,n+1}(\Sym^d\P^r,\beta)$. 
  Note that
  $\ev_{n+1}\circ\iota_{\Gamma}$ factors through $(\mu,\sigma),$
  hence $\iota_{\Gamma}^*\ev_{n+1}^*[(\mu,\sigma)]$ is the
  weight $e_T(T_{(\mu,\sigma)}I\Sym^d\P^r)$.

  Then $\Gamma$ has a decorated subtree
  $\kappa\in\Upsilon(\mu,\sigma)$, obtained by removing all
  edges except for $e:=e_{v_1}$ (and necessary vertices), and all
  marked points except $b_{n+1}$. Let
  $\Gamma\setminus\kappa$ denote the tree obtained by
  \emph{pruning $\kappa$}. That is,
  $\Gamma\setminus
  \kappa\in\Graphs_{0,n+1}^{\min}(\Sym^d\P^r,\beta-\beta({\kappa}))$
  is defined by $V(\Gamma\setminus\kappa)=V(\Gamma)\setminus\{v_1\}$,
  $E(\Gamma\setminus\kappa)=E(\Gamma)\setminus e,$ and decorations
  $\Mark,$ $\VEval,$ $q$, and $\Mon$ are unchanged, except
  $\Mark({n+1}):=v_2$, where $v_2$ is the common vertex of
  $\kappa$ and $\Gamma\setminus\kappa$. Observe that
  an automorphism of $\Gamma$ fixes $b_{n+1},$ and therefore fixes
  $e$, so we have
  $\Aut(\Gamma)=\Aut(\Gamma\setminus\kappa)$. Thus
  by Theorem \ref{thm:FixedLocusLosevManin}, up to a
  $\bar{C}_{\VEval(v_2)}(\Mon(v_2,e))$-gerbe, we may write
  \begin{align*}
    \Mbar_{\Gamma}\cong\Mbar_e\times\Mbar_{\Gamma\setminus\kappa}.
  \end{align*}
  We factor the
  $T$-equivariant map $\Mbar_{\Gamma}\to\Spec\C$ through the
  second projection, i.e. we integrate over $\Mbar_e$:
  \begin{align*}
    \Contr(\Gamma)=\frac{\abs{C_{\mu}(\sigma)}\abs{C_{\mu'({\kappa})}(\sigma'({\kappa}))}}{r({\kappa})}\int_{[\Mbar_{\Gamma\setminus\kappa}]'}\left(\int_{\Mbar_e}\frac{e_T(T_{(\mu,\sigma)}I\Sym^d\P^r)}{e_T(N_{\Gamma}^{\vir})}\iota_{\Gamma}^*\left(\prod_{j=1}^n\ev_j^*\mathbf{t}(\bar{\psi})\cup\frac{1}{-z-\bar{\psi}_{n+1}}\right)\right).
  \end{align*}
  The factor
  $\abs{C_{\mu'(\kappa)}(\sigma'(\kappa))}/r(\kappa)$
  is the order of $\bar{C}_{\VEval(v_2)}(\Mon(v_2,e))$. From
  Proposition \ref{prop:VirtualNormalBundle}, we may write
  $$\frac{e_T(T_{(\mu,\sigma)}I\Sym^d\P^r)}{e_T(N_{\Gamma}^{\vir})}=\frac{1}{W}\cdot\frac{e_T(T_{(\mu'(\kappa),\sigma'(\kappa))}I\Sym^d\P^r)}{e(N_{\Gamma\setminus\kappa}^{\vir})(-\psi^{\Mbar_{v_2}}_e-\psi^{\Mbar_e}_{v_2})},$$ where
  \begin{align*}
    W:&=\frac{\prod_{\eta\in\Stat(\kappa)}\prod_{i\ne
        i(\eta)}(\alpha_{i(\eta)}-\alpha_i)}{e_T(T_{(\mu,\sigma)}I\Sym^d\P^r)}\prod_{\eta\in\Mov(\kappa)}\prod_{\substack{A+B=\beta_\eta(\kappa)\\0\le
        i\le
        r\\(A,i)\ne(0,i^{\mov}(v_2,e))\\(B,i)\ne(0,i^{\mov}(v_1,e))}}\left(\frac{A}{\beta_\eta(\kappa)}\alpha_{i^{\mov}(v_1,e)}+\frac{B}{\beta_\eta(\kappa)}\alpha_{i^{\mov}(v_2,e)}-\alpha_i\right)\\
    &=\prod_{\eta\in\Mov(\kappa)}\prod_{\substack{1\le
        B\le\beta_\eta(\kappa)\\0\le i\le
        r\\(B,i)\ne(\beta_\eta(\kappa),i^{\mov}(v_2,e))}}\left(\frac{\beta_\eta(\kappa)-B}{\beta_\eta(\kappa)}\alpha_{i^{\mov}(v_1,e)}+\frac{B}{\beta_\eta(\kappa)}\alpha_{i^{\mov}(v_2,e)}-\alpha_i\right)\in
    H^*_{T,\loc}(\Spec\C)
  \end{align*}
  Note that the cancellation in the last step removes the factors
  where $B=0$, and that $1/W$ is the product appearing in
  $\RC(\kappa,a)$.

  To avoid confusion, we write $\bar{\psi}_{n+1}^{\Gamma}$
  (resp.  $\bar{\psi}_{n+1}^{\Gamma\setminus\kappa}$) for
  the $\bar{\psi}$-class at the $(n+1)$st marked point on
  $\Mbar_{\Gamma}$ (resp.
  $\Mbar_{\Gamma\setminus\kappa}$), recalling that on
  $\Gamma\setminus\kappa$ we defined $\Mark({n+1})=v_2.$
  We also have
  $\iota_{\Gamma}^*\bar{\psi}_{n+1}^{\Gamma}=\bar{\psi}^{\Mbar_e}_{v_1}.$
  The $T$-weight on $\bar{\psi}_{v_1}^{\Mbar_e}$ is
  $-w(\kappa)$ (see Notation \ref{not:WeightOfKappa}), so
  we have
  \begin{align*}
    \bar{\psi}_{v_1}^{\Mbar_e}=\bar{\psi}_{v_1}^{\noneq}-w(\kappa)\in
    H^*_T(\Mbar_{\Gamma})\cong H^*(\Mbar_{\Gamma})\otimes
    H^*_T(\Spec\C),
  \end{align*}
  where $\bar{\psi}_{v_1}^{\noneq}$ denotes the nonequivariant
  $\bar{\psi}$-class. Similarly
  $\bar{\psi}_{v_2}^{\Mbar_e}=\bar{\psi}_{v_2}^{\noneq}+w(\kappa).$
  Then since $\iota_{\Gamma}^*\ev_j^*\mathbf{t}(\bar{\psi})$ is
  pulled back from $\Mbar_{\Gamma\setminus\kappa}$,
  \begin{align*}
    \Contr(\Gamma)&=\frac{\abs{C_{\mu}(\sigma)}\abs{C_{\mu'(\kappa)}(\sigma'(\kappa))}}{r(\kappa)}\frac{e_T(T_{(\mu'(\kappa),\sigma'(\kappa))}I\Sym^d\P^r)}{W}\\
    &\quad\quad\cdot\int_{[\Mbar_{\Gamma\setminus\kappa}]'}\left(\frac{\iota_{\Gamma}^*\left(\prod_{j=1}^n\ev_j^*\mathbf{t}(\bar{\psi})\right)}{e_T(N_{\Gamma\setminus\kappa}^{\vir})}\int_{\Mbar_e}\frac{1}{(-\psi_{n+1}^{\Gamma\setminus\kappa}-\psi^{\noneq}_{v_2}-w(\kappa))}\frac{1}{(-z-\bar{\psi}^{\noneq}_{v_1}+w(\kappa))}\right).
  \end{align*}
  We compute the last integral using the fact that $w(\kappa)$
  is invertible, and Lemma \ref{lem:PsiPullback}, which says we may
  integrate on $\Mbar_{0,k+2}$ instead of $\Mbar_e$. We use
  $$r(\kappa)(-\psi_{n+1}^{\Gamma\setminus\kappa}-\psi_{v_2})=-\bar{\psi}_{n+1}^{\Gamma\setminus\kappa}-\bar{\psi}_{v_2}=\bar{\psi}_{n+1}^{\Gamma\setminus\kappa}-\bar{\psi}^{\noneq}_{v_2}-w(\kappa).$$ It
  is well-known (see e.g. \cite{KockNotes}, Lemma 1.5.1) that
  \begin{align}\label{eqn:PsiPullback}
    \int_{\Mbar_{0,k}}\psi_1^{m}\psi_2^{k-3-m}=\binom{k-3}{m}.
  \end{align}
  By Lemma \ref{lem:PsiPullback}, this identity holds on
  $\Mbar_{0|k|\infty}$ also. Thus:
  \begin{align}\label{eq:PsiClassComputation}
    &\int_{\Mbar_e}\frac{1}{(-\bar{\psi}_{n+1}^{\Gamma\setminus\kappa}-\bar{\psi}^{\noneq}_{v_2}-w(\kappa))}\frac{1}{(-z-\bar{\psi}^{\noneq}_{v_1}+w(\kappa))}\nonumber\\
    &\quad\quad=\frac{1}{\abs{S_e}\prod_{\eta\in\Mov(\kappa)}\beta_\eta(\kappa)}\int_{\Mbar_{v_1|{\mov(\kappa)}|v_2}}\left(\sum_{m_1=0}^\infty\frac{(\bar{\psi}_{v_2})^{m_1}}{(-\bar{\psi}_{n+1}^{\Gamma\setminus\kappa}-w(\kappa))^{m_1+1}}\right)\left(\sum_{m_2=0}^\infty\frac{(\bar{\psi}_{v_1})^{m_2}}{(-z+w(\kappa))^{m_2+1}}\right)\nonumber\\
    &\quad\quad=\frac{1}{\abs{S_e}\prod_{\eta\in\Mov(\kappa)}\beta_\eta(\kappa)}\sum_{m_1+m_2={\mov(\kappa)}-1}\frac{\binom{{\mov(\kappa)-1}}{m_1}}{(-\bar{\psi}_{n+1}^{\Gamma\setminus\kappa}-w(\kappa))^{m_1+1}(-z+w(\kappa))^{m_2+1}}\\
    &\quad\quad=\frac{1}{\abs{S_e}\prod_{\eta\in\Mov(\kappa)}\beta_\eta(\kappa)}\frac{(-z-\bar{\psi}_{n+1}^{\Gamma\setminus\kappa})^{{\mov(\kappa)}-1}}{(-\bar{\psi}_{n+1}^{\Gamma\setminus\kappa}-w(\kappa))^{{\mov(\kappa)}}(-z+w(\kappa))^{{\mov(\kappa)}}}.\nonumber
  \end{align}
  The last equality in \eqref{eq:PsiClassComputation} comes from
  expanding
  $$((-z+w(\kappa))+(-\bar{\psi}_{n+1}^{\Gamma\setminus\kappa}-w(\kappa)))^{{\mov(\kappa)}-1}$$
  via the binomial theorem. Altogether, we have
  \begin{align}\label{eq:ContrFactored}
    \Contr(\Gamma)&=\frac{\abs{C_{\mu}(\sigma)}\abs{C_{\mu'(\kappa)}(\sigma'(\kappa))}}{\abs{S_e}\prod_{\eta\in\Mov(\kappa)}\beta_\eta(\kappa)}\frac{e_T(T_{(\mu'(\kappa),\sigma'(\kappa))}I\Sym^d\P^r)}{W\cdot(-z+w(\kappa))^{{\mov(\kappa)}}}\\
      &\quad\quad\cdot\int_{[\Mbar_{\Gamma\setminus\kappa}]'}\left(\frac{\iota_{\Gamma}^*\left(\prod_{j=1}^n\ev_j^*\mathbf{t}(\bar{\psi})\right)}{e_T(N_{\Gamma\setminus\kappa}^{\vir})}\frac{(-z-\bar{\psi}_{n+1}^{\Gamma\setminus\kappa})^{{\mov(\kappa)}-1}}{(-\bar{\psi}_{n+1}^{\Gamma\setminus\kappa}-w(\kappa))^{{\mov(\kappa)}}}\right).\nonumber
  \end{align}
  For fixed $\beta_0$, and $n_0$, from \eqref{eq:CoeffOfFAsFiniteSum},
  the coefficient of $Q^{\beta_0}x^{n_0}$ in
  $\mathbf{f}_{(\mu,\sigma)}$ only has contributions from
  $\Gamma\in\Graphs_{0,n}(\Sym^d\P^r,\beta)$ for
  $\beta+n\le\beta_0+n_0$. This is because $\mathbf{t}(z)\in\langle
  Q,x\rangle$, so if $\H[[x]]$ is graded by giving $Q$ and $x$ degree
  1, then the $(n,\beta)$ term in \eqref{eq:RestrictionOfF} has degree
  at least $n+\beta$. In particular,
  $\bigcup_{\beta+n\le\beta_0+n_0}\Graphs_{0,n}(\Sym^d\P^r,\beta)$ is a
  finite set. Thus \eqref{eq:CoeffOfFAsFiniteSum} and
  \eqref{eq:ContrFactored} realize the contribution to such a
  coefficient from trees of type \ref{item:recursion} as a finite
  sum of rational functions with poles at the weights
  $\kappa$. Together with the analysis above for types
  \ref{item:vanish} and \ref{item:initial}, this proves that
  $\mathbf{f}_{(\mu,\sigma)}$ satisfies condition \ref{item:Poles}
  of the Theorem.

  \bigskip

  We consider the Laurent coefficient
  $\Laur(\Contr(\Gamma),(w-z)^{-a}).$ By
  \eqref{eq:ContrFactored}, $\Laur(\Contr(\Gamma),(w-z)^{-a})$
  is zero if $w\ne w(\kappa),$ or if
  ${\mov(\kappa)}<a.$ Otherwise,
  \begin{align*}
    \Laur&(\Contr(\Gamma),(w-z)^{-a})\\
    &=\frac{1}{({\mov(\kappa)}-a)!}\left.\left(\deriv[{\mov(\kappa)}-a]{}{(w(\kappa)-z)}(w(\kappa)-z)^{{\mov(\kappa)}}\Contr(\Gamma)\right)\right|_{z\mapsto
      w(\kappa)}\\
    &=\frac{(-1)^{{\mov(\kappa)}-a}\abs{C_{\mu}(\sigma)}\abs{C_{\mu'(\kappa)}(\sigma'(\kappa))}\binom{{\mov(\kappa)}-1}{a-1}}{W\abs{S_e}\prod_{\eta\in\Mov(\kappa)}\beta_\eta(\kappa)}\int_{[\Mbar_{\Gamma\setminus\kappa}]'}\left(\frac{\iota_{\Gamma}^*\left(\prod_{j=1}^n\ev_j^*\mathbf{t}(\bar{\psi})\right)}{e_T(N_{\Gamma\setminus\kappa}^{\vir})}\frac{e_T(T_{(\mu'(\kappa),\sigma'(\kappa))}I\Sym^d\P^r)}{(-\bar{\psi}_{n+1}^{\Gamma\setminus\kappa}-w(\kappa))^{{\mov(\kappa)}-a+1}}\right).
  \end{align*}
  Now, summing over all $\Gamma$ of type \ref{item:recursion}
  with associated subtree $\kappa$ yields
  \begin{align}\label{eq:KappaContribution}
    \frac{(-1)^{{\mov(\kappa)}-a}\abs{C_{\mu}(\sigma)}\abs{C_{\mu'(\kappa)}(\sigma'(\kappa))}\binom{{\mov(\kappa)}-1}{a-1}}{W\abs{S_e}\prod_{\eta\in\Mov(\kappa)}\beta_\eta(\kappa)}&\left\langle\mathbf{t}(\bar{\psi}),\ldots,\mathbf{t}(\bar{\psi}),\frac{[(\mu'(\kappa),\sigma'(\kappa))]}{(-\bar{\psi}_{n+1}^{\Gamma\setminus\kappa}-w(\kappa))^{{\mov(\kappa)}-a+1}}\right\rangle_{0,n+1,\beta-\beta(\kappa)}^{\Sym^d\P^r,T}.
  \end{align}
  On the other hand, the coefficient
  $\Laur(\mathbf{f}_{(\mu'(\kappa),\sigma'(\kappa))},(w(\kappa)-z)^{{\mov(\kappa)}-a})$
  is
  \begin{align}\label{eq:OtherSideLaurentCoeff}
    \sum_{\substack{\beta\ge0\\n\ge0}}\frac{\abs{C_{\mu'(\kappa)}(\sigma'(\kappa))}Q^\beta}{n!}\left\langle\mathbf{t}(\bar{\psi}),\ldots,\mathbf{t}(\bar{\psi}),\frac{[(\mu'(\kappa),\sigma'(\kappa))]}{(-\bar{\psi}_{n+1}^{\Gamma\setminus\kappa}-w(\kappa))^{{\mov(\kappa)}-a+1}}\right\rangle_{0,n+1,\beta}^{\Sym^d\P^r,T}
  \end{align}
  We compute
  $\frac{\abs{C_\mu(\sigma)}}{\abs{S_e}\prod_{\eta\in\Mov({\kappa})}\beta_\eta(\kappa)}$
    explicitly:
    \begin{align*}
      \abs{C_\mu(\sigma)}&=\abs{S_\sigma}\prod_{\eta\in\sigma}\eta\\
      \abs{S_e}&=\abs{C_{\Stat}(\kappa)}\abs{S_{\Mov({\kappa})}}=\abs{S_{\Stat(\kappa)}}\abs{S_{\Mov({\kappa})}}\prod_{\eta\in\Stat(\kappa)}\eta\\
      \frac{\abs{C_\mu(\sigma)}}{\abs{S_e}\prod_{\eta\in\Mov({\kappa})}\beta_\eta(\kappa)}&=\frac{\abs{S_\sigma}\prod_{\eta\in\Mov({\kappa})}\eta}{\abs{S_{\Stat({\kappa})}}\abs{S_{\Mov({\kappa})}}\prod_{\eta\in\Mov({\kappa})}\beta_\eta(\kappa)}=\frac{1}{q({\kappa})^{\mov({\kappa})}}\binom{\sigma_{i_1^{\mov}}({\kappa})}{\Mov({\kappa})}
    \end{align*}
    With (\ref{eq:KappaContribution}) and
    (\ref{eq:OtherSideLaurentCoeff}), this proves
    \ref{item:Recursion}.  Note that the contribution from all graphs
    of type \ref{item:recursion} (and the term
    $\mathbf{t}_{(\mu,\sigma)}(z)$) is
  \begin{align}\label{eq:Tau}
    \boldsymbol{\tau}_{(\mu,\sigma)}(z):=\mathbf{t}_{(\mu,\sigma)}(z)+\sum_{\substack{\kappa\in\Upsilon(\mu,\sigma)\\a\le{\mov(\kappa)}}}\frac{Q^{\beta(\kappa)}\RC(\kappa,a)}{(w(\kappa)-z)^a}\Laur(\mathbf{f}_{(\mu'(\kappa),\sigma'(\kappa))},(w(\kappa)-z)^{{\mov(\kappa)}-a}).
  \end{align}

  \bigskip

  The proof of condition \ref{item:TwistedCone} is identical to that
  of condition (C3) in \cite{CoatesCortiIritaniTseng2015}, and we
  reproduce the argument here for convenience.

  Consider a decorated tree $\Gamma$ of type
  \ref{item:initial}. We write $v:=\Mark({n+1})\in V^S(\Gamma).$
  The marked points of $\Mbar_v$ correspond to (1) elements of
  $\Mark^{-1}(v),$ and (2) edges $e\in E(\Gamma,v)$. To $e$ is
  associated a maximal subtree $\Gamma_e$ containing $v$, with
  $E(\Gamma_e,v)=e$. We decorate $\Gamma_e$ so that $\Mark^{-1}(v)=b,$
  and the rest of the decorations inherited from $\Gamma.$ We
  will then write $\Contr(\Gamma)$ in terms of
  $\Contr(\Gamma_e)$ for $e\in E(\Gamma,v),$ and integrals over
  the vertex moduli space $\Mbar_v$.

  We apply \eqref{eq:LocalizationContribution} again. After an \'etale
  base change $\widetilde\Mbar_{\Gamma}\to\Mbar_{\Gamma},$
  we may label the subtrees $\Gamma_e.$ (Write $M$ for the
  degree of this base change.) We then write
  $\widetilde\Mbar_{\Gamma}\cong\Mbar_v\times\prod_{e\in
    E(\Gamma,v)}\Mbar_{\Gamma_e}.$ Now we again apply
  Proposition \ref{prop:VirtualNormalBundle}, to see that
  \begin{align*}
    \frac{1}{e_T(N_{\Gamma}^{\vir})}&=e_T^{-1}(R\pi_*(C_v,f^*T\Sym^d\P^r))\prod_{e\in
      E(\Gamma,v)}\frac{r(v,e)e_T(T_{(\mu,\Mon(v,e))}I\Sym^d\P^r)}{(-\bar{\psi}^{\Mbar_v}_e-\bar{\psi}^{\Mbar_e}_v)e_T(N_{\Gamma_e}^{\vir})}
  \end{align*}
  Observe that
  $\frac{e_T(T_{(\mu,\Mon(v,e))}I\Sym^d\P^r)}{(-\bar{\psi}^{\Mbar_v}_e-\bar{\psi}^{\Mbar_e}_v)}$
  is the insertion at $b$ in
  $\Contr(\Gamma_e)|_{z\mapsto\bar{\psi}^{\Mbar_v}_e}.$ Thus
  \begin{align*}
    \Contr(\Gamma)&=\frac{1}{M}\int_{\Mbar_v}\left(\prod_{e\in
        E(\Gamma,v)}\abs{C_{\mu}(\sigma)}Q^{\beta(\Gamma_e)}\Contr(\Gamma_e)|_{z\mapsto\bar{\psi}_e^{\Mbar_v}}\right)\cup\left(\prod_{i\in\Mark^{-1}(v)}\mathbf{t}(\bar{\psi}_i)\right)\\
    &\quad\quad\quad\quad\cup\frac{e_T(T_{(\mu,\sigma)}I\Sym^d\P^r)}{-z-\bar{\psi}_{n+1}}\cup
    e_T^{-1}(R\pi_*(C_v,f^*T\Sym^d\P^r)).
  \end{align*}
  This is almost a
  twisted Gromov-Witten invariant of $\VEval(v),$ but not quite, since
  there are restrictions on the monodromies at the marked
  points. Summing over $\Gamma_e$ for a single $e$, with
  everything else fixed, gives the insertion
  $\boldsymbol{\tau}_{(\mu,\Mon(v,e))}(\psi),$ where the initial term
  comes from replacing $\Gamma_e$ with a marked point. Thus
  summing over all $\sigma$, and over all $\Gamma$ of type
  \ref{item:initial}, gives
  \begin{align*}
    \sum_{m=2}^\infty\sum_\sigma\frac{1}{m!}\left\langle\boldsymbol{\tau}_{\mu}(\bar{\psi}),\ldots,\boldsymbol{\tau}_{\mu}(\bar{\psi}),\frac{[(\mu,\sigma)]}{-z-\bar{\psi}_{n+1}}\right\rangle_{0,m+1,0}^{\VEval(v),T,\tw}1_{(\mu,\sigma)}\in
    H^*_{T,\loc}(I\mu),
  \end{align*}
  where $1_{(\mu,\sigma)}$ is the fundamental class of
  $(\mu,\sigma)\in I\mu,$ and
  $\boldsymbol{\tau}_{\mu}(z)=\sum_{\sigma'\in\MultiPart(\mu)}\boldsymbol{\tau}_{(\mu,\sigma')}(z)1_{(\mu,\sigma)}$. Adding
  in the contributions from type \ref{item:recursion} graphs, summing
  \eqref{eq:RestrictionOfF} over $\sigma$ yields:
  \begin{align*}
    \mathbf{f}_\mu&=\sum_{\sigma}\mathbf{f}_{(\mu,\sigma)}1_{\mu,\sigma}=-1_\mu
    z+\boldsymbol{\tau}_{\mu}(z)+\sum_{m=2}^\infty\sum_\sigma\frac{1}{m!}\left\langle\boldsymbol{\tau}_{\mu}(\bar{\psi}),\ldots,\boldsymbol{\tau}_{\mu}(\bar{\psi}),\frac{[(\mu,\sigma)]}{-z-\bar{\psi}_{n+1}}\right\rangle_{0,m+1,0}^{\VEval(v),T,\tw}1_{(\mu,\sigma)},
  \end{align*}
  where $1_\mu$ is the untwisted fundamental class on $I\mu$. This
  shows that $\mathbf{f}_\mu$ is a $\Lambda_{\nov}^T[[x]]$-valued
  point of $\mathcal{L}_\mu^{\tw}.$

  \bigskip

  The converse also requires no modification from
  \cite{CoatesCortiIritaniTseng2015}. Suppose $\mathbf{f}$ satisfies
  the conditions of the theorem. By conditions \ref{item:Poles} and
  \ref{item:Recursion}, we may uniquely write
  \begin{align*}
    \mathbf{f}_\mu=-1_\mu
    z+\sum_{\sigma\in\MultiPart(\mu)}\boldsymbol{\tau}_{(\mu,\sigma)}1_{(\mu,\sigma)}+O(z^{-1}),
  \end{align*}
  where $\boldsymbol{\tau}_{(\mu,\sigma)}(z)$ is the expression in
  \eqref{eq:Tau}, for some
  $\mathbf{t}_{(\mu,\sigma)}(z)\in\iota_\mu^*(\H^+)[[x]].$ We claim
  that the set $\{\mathbf{t}_{(\mu,\sigma)}(z)\}$ for all fixed points
  $(\mu,\sigma)$ determines $\mathbf{f}.$ By the localization
  isomorphism, if suffices to show that it determines
  $\mathbf{f}_{(\mu,\sigma)}$ for all $(\mu,\sigma)$. We induct on the degree
  $\beta+k,$ where $k$ is the exponent of $x$. The base case
  $\beta=k=0$ is taken care of by the assumption
  $\mathbf{f}|_{Q=x=0}=-1z.$ Assume the coefficients of
  $\mathbf{f}_{(\mu,\sigma)}$ up to degree $\beta+k$ are determined by
  $\{\mathbf{t}_{(\mu,\sigma)}\}$. Consider the coefficients of degree
  $\beta+k+1.$ Some of these appear in $\mathbf{t}(z)$, but these are
  given. Some of them appear in $\boldsymbol{\tau}_{(\mu,\sigma)}(z)$,
  but these are determined since they are of the form:
  $Q^{\beta(\kappa)}$ multiplied by a factor determined by the
  inductive hypothesis. The sum of all of these terms is in
  $H^*_{CR,T,\loc}(\mu)[[Q,x]][[z]].$

  Finally, some of them appear in $O(z^{-1}).$ However, condition
  \ref{item:TwistedCone} and \eqref{eq:TwistedConeDef} show that these
  are determined by terms of $-1z+\boldsymbol{\tau}_{(\mu,\sigma)}(z)$
  of degree at most $\beta+k+1.$ Since all such terms are determined
  by $\mathbf{t}_{(\mu,\sigma)}$ and induction, the degree $\beta+k+1$
  coefficients of $\mathbf{f}_{(\mu,\sigma)}$ are determined. Thus in
  fact $\mathbf{f}$ is determined by
  $\{\mathbf{t}_{(\mu,\sigma)}(z)\}$.

  Again by the localization isomorphism, the set
  $\{\mathbf{t}_{(\mu,\sigma)}(z)\}$ corresponds uniquely to an
  element $\mathbf{t}(z)\in\H^+[[x]]$ that restricts to each
  $\mathbf{t}_{(\mu,\sigma)}(z).$ This in turn corresponds uniquely to
  a $\Lambda_{\nov}^T[[x]]$-valued point $\mathbf{f}_{\mathrm{GW}}$ of
  $\mathcal{L}_x.$ By the uniqueness argument above we have
  $\mathbf{f}=\mathbf{f}_{\mathrm{GW}}$.
\end{proof}
\begin{rem}
  No modifications are required to replace $\Lambda_{\nov}^T[[x]]$ in
  the statement of Theorem \ref{thm:ConeCharacterization} with a
  finitely generated graded $\Lambda_{\nov}^T$-algebra.
\end{rem}

\section{The \texorpdfstring{$I$}{I}-function and mirror theorem}\label{sec:MirrorTheorem}
In this section we introduce a function
$I_{\Sym^d\P^r}(Q,t,\mathbf{x},-z)$, and show that it satisfies the
conditions of Theorem \ref{thm:ConeCharacterization}, conditional upon two combinatorial identities that we checked extensively by computer, but were unable to prove. (See Section \ref{sec:Appendix}.) That is, we prove
that these identities imply $I_{\Sym^d\P^r}(Q,t,\mathbf{x},-z)$ is a
$\Lambda_{\nov}^T[[t,\mathbf{x}]]/(\mathbf{x})^2$-valued point of
$\mathcal{L}_{\Sym^d\P^r},$ where $\mathbf{t}=\{t_0,\ldots,t_r\}$ and
$\mathbf{x}=\{x_\pi\}_{\pi\in\Part(d)}$ are formal variables.

  The (first order) $I$-function
  $I_{\Sym^d\P^r}(Q,\mathbf{t},\mathbf{x},z)$ is
  defined
  by its restrictions to the
  $T$-fixed points $\iota_{\mu,\sigma}:(\mu,\sigma)\into I\Sym^d\P^r$ as follows:
  \begin{align}\label{eq:IDef}
    \iota_{(\mu,\sigma)}^*I_{\Sym^d\P^r}(Q,t,\mathbf{x},z)&:=\left(z\delta_{\sigma,(1,\ldots,1)}1_\sigma+x_{\pi(\sigma)}1_\sigma\right)                                   \sum_{\beta\ge0}\exp\left(\sum_{i=0}^rt_i\left(\beta+\sum_{j=0}^r\mu_j(\alpha_j-\alpha_i)/z\right)\right)Q^\beta\nonumber\\
                                                          &\quad\quad\cdot\sum_{\substack{(L_\eta)_{\eta\in\sigma}\\L_\eta\ge0\\\sum
    L_\eta=\beta}}
    \left(\prod_{j=0}^r\prod_{\eta\in\sigma_j}\frac{1}{\prod_{\gamma=1}^{L_\eta}\prod_{i=0}^r\left(\alpha_j-\alpha_i+\frac{\gamma}{\eta}z\right)}\right),
  \end{align}
  where $1_\sigma$ is the fundamental class of $(\mu,\sigma)\in I\mu$. We will use the notations $$I_{(\mu,\sigma)}(Q,\mathbf{t},\mathbf{x},z):=\iota_{(\mu,\sigma)}^*I_{\Sym^d\P^r}(Q,\mathbf{t},\mathbf{x},z)$$ and $$I_\mu(Q,\mathbf{t},\mathbf{x},z):=\bigoplus_{\sigma\in\MultiPart(\mu)}I_{(\mu,\sigma)}(Q,\mathbf{t},\mathbf{x},z).$$
  \begin{rem}
As in Remark \ref{rem:TPowerSeries}, we have $I_{\Sym^d\P^r}(Q,\mathbf{t},\mathbf{x},z)\not\in
  H^*(I\Sym^d\P^r,\Q)[[Q,\mathbf{t},\mathbf{x}]]((z^{-1}))$ due to the presence of arbitrarily high powers of $z.$ The ``topological nilpotence'' condition we alluded to is simply to say that for a fixed monomial in the variables $t_i$ and $x_\pi$, the powers of $z$ in that monomial's coefficient are bounded above.
\end{rem}
\begin{rem}
$I_{\Sym^d\P^r}$ may be decomposed into pieces that we might naturally regard as $I_{\Sym^d\C^r}.$ If we introduce variables $x_{j,\pi}$, where $0\le j\le r$ and $\pi$ is a partition of $\mu_j$, and set $\prod_{0\le j\le r}x_{j,\pi_j}=x_{\cup\pi_j},$ then we may repeatedly switch orders of summation and products to write 
\begin{align}\label{ISplitUp}
  I_\mu(Q,\mathbf{t},\mathbf{x},z)=I_\mu(Q,\mathbf{t},0,z)+\prod_{j=0}^rI_{j,\mu_j}(Q,\mathbf{t},\mathbf{x},z),
  \end{align}
  where
  \begin{align*}
    I_{j,d}(Q,\mathbf{t},\mathbf{x},z)&=\sum_{\pi\in\Part(d)}x_{j,\pi}1_{j,\pi}\prod_{\eta\in\pi}\sum_{\beta\ge0}\frac{Q^\beta
                                        \exp\left(\sum_{i=0}^rt_i(\beta+\eta(\alpha_j-\alpha_i)/z)\right)}{\prod_{i=0}^r\prod_{\gamma=1}^{\beta}(\alpha_j-\alpha_i+\frac{\gamma}{\eta}z)}.
  \end{align*}
  Here $1_{j,\pi}$ is the pullback of $1_\pi$ along the natural isomorphism $IBG_\mu\to\prod_{j=0}^rIBS_{\mu_j}.$
%
\end{rem}
  We now prove:
  \begin{thm}\label{thm:MirrorTheorem}
  Assuming Identities \ref{lem:WhatIsTauIdentity} and \ref{lem:WhatIsTauSigma} hold, $I_{\Sym^d\P^r}(Q,\mathbf{t},\mathbf{x},-z)$ is a
  $\Lambda_{\nov}^T[[\mathbf{t},\mathbf{x}]]/(\mathbf{x}^2)$-valued point of
  $\mathcal{L}_{\Sym^d\P^r}.$
\end{thm}
\begin{rem}
  This result is weaker than that in the original preprint, where
  $\Lambda_{\nov}^T[[\mathbf{t},\mathbf{x}]]$ appeared without the quotient by
  the ideal $(\mathbf{x})^2,$ and the dependence on Identities \ref{lem:WhatIsTauIdentity} and \ref{lem:WhatIsTauSigma} was omitted. We do not know if it is possible to
  find an explicit formula for a (nontrivial)
  $\Lambda_{\nov}^T[[t,\mathbf{x}]]$-valued point of
  $\mathcal{L}_{\Sym^d\P^r}$.
\end{rem}
\begin{proof}
  We must prove that the criteria in Theorem
  \ref{thm:ConeCharacterization} are satisfied. The form of \eqref{eq:IDef} implies that the coefficient of $Q^\beta
  x_\pi\mathbf{t}^{\mathbf{a}}$ is a rational function in $z$ with
  poles at $z=0,$ $z=\infty$, and $z=\frac{\alpha_{i_1}-\alpha_{i_2}}{q},$
  where $i_1=i(\eta)$ for some $\eta\in\sigma,$ and
  $q\in\frac{1}{\eta}\Z.$ This is exactly the set of values arising as
  $w(\kappa)$ for
  $\kappa\in\Upsilon(\mu,\sigma)$. This proves \ref{item:Poles}.

  To prove \ref{item:Recursion}, 
 we fix $\mu\in\ZPart(d,r+1),$ $\sigma\in\MultiPart(\mu),$
   $\beta\ge0,$ $L=(L_\eta)_{\eta\in\sigma}$ as in \eqref{eq:IDef}, $a\in\Z_{>0},$ distinct elements $i_1,i_2\in\{0,\ldots,r\}$ such
  that $\mu_{i_1}\ne0$, and $q\in\Q$ such that $q\in\frac{1}{\eta}\Z$
  for some $\eta\in\sigma_{i_1}$. Let
  $w=\frac{\alpha_{i_1}-\alpha_{i_2}}{q}$. 
  
%
  First, assume that $\sigma$ is not the trivial multipartition of $\mu$. The term of $I_{(\mu,\sigma)}(Q,\mathbf{t},\mathbf{x},-z)$ corresponding to $L$ is $T_{\sigma,L}(z)x_{\pi(\sigma)}1_\sigma Q^\beta,$ where:
  \begin{align}\label{TL}
  T_{\sigma,L}(z)&=\prod_{j=0}^r\prod_{\eta\in\sigma_j}H_{L_\eta,j,\eta}(z)
  \end{align}
  and
  \begin{align}\label{HL}
  H_{\beta,j,\eta}(z)&=\frac{\exp\left(\sum_{i=0}^r t_i\left(\beta+\eta(\alpha_j-\alpha_i)/(-z)\right)\right)}{\prod_{\gamma=1}^{\beta}\prod_{i=0}^r\left(\alpha_j-\alpha_i-\frac{\gamma}{\eta}z\right)}
  \end{align}
  
  Let
  $\sigma_L=\{\eta\in\sigma_{i_1}:L_\eta\ge q\eta\}$, and recall Notation \ref{not:WeightOfKappa}. Given a nonempty
  submultiset $M\subseteq\sigma_L,$ there is a unique
  $\kappa_M\in\Upsilon(\mu,\sigma)$ such that $w(\kappa_M)=w$ and
  $\Mov({\kappa})=M,$ and we may define 
  $L'({\kappa_M})=(L'(\kappa_M)_\eta)_{\eta\in\sigma'({\kappa_M})}$ by letting $L'(\kappa_M)_\eta=L_\eta-q\eta$ for
  $\eta\in M$. Note that such $\eta$ are parts of
  $\sigma_{i_2}'({\kappa_M})$, and that we have $\sum_{\eta\in\sigma'(\kappa_M)}L'(\kappa_M)_\eta=\beta-\beta(\kappa_M).$ Therefore, to prove \ref{item:Recursion}, it is sufficient to show that
  \begin{align}\label{eq:SingleLabel}
  \Laur(T_{\sigma,L}(z),(w-z)^{-a})&=\sum_{\substack{M\subseteq\sigma_L\\\abs{M}\ge a}}\RC(\kappa_M,a)\Laur(T_{\sigma'(\kappa_M),L'(\kappa_M)}(z),(w-z)^{\abs{M}-a}),
  \end{align}
  since adding up \eqref{eq:SingleLabel} over all $L$ and $\beta$ yields \eqref{eq:LaurentRecursion}.
   Note that $H_{L_\eta,j,\eta}(z)$ has a pole at $w$ if and only if $j=i_1$ and $\eta\in\sigma_L$, and in this case  $H_{L_\eta,j,\eta}(z)$ has a simple pole at $w$, coming from the factor $(\gamma,i)=(q\eta,i_2)$ in the denominator. 
  Thus $T_{\sigma,L}(z)$ has a pole at $w$ of order exactly
  $\abs{\sigma_L}$. Define $$\tilde{H}_{\mu,\sigma,L,j,\eta}(z)=\begin{cases}
   (w-z)H_{L_\eta,j,\eta}(z)&j=i_1,\eta\in\sigma_L\\
   H_{L_\eta,j,\eta}(z)&\text{else}.
   \end{cases}
   $$ If $a>\abs{\sigma_L},$ then both sides of
  \eqref{eq:SingleLabel} are zero, so assume $a\le\abs{\sigma_L}$. By the product rule, the
  left side of \eqref{eq:SingleLabel} is equal to
  \begin{align}\label{eq:LeftSideLaurent}
    \frac{1}{(\abs{\sigma_L}-a)!}\left(\deriv[{\abs{\sigma_L}-a}]{}{(w-z)}(w-z)^{\abs{\sigma_L}}T_{\sigma,L}(z)\right)_{z\mapsto
    w}&=\sum_{\substack{(k_{(j,\eta)})_{0\le j\le r,\eta\in\sigma_j}\\\sum_{j,\eta}k_{(j,\eta)}=\abs{\sigma_L}-a}}\prod_{j=0}^r\prod_{\eta\in\sigma_j}\frac{\tilde H_{L_\eta,j,\eta}^{(k_{(j,\eta)})}(w)}{k_{(j,\eta)}!}.
      \end{align}

Similarly, the right side of
  \eqref{eq:SingleLabel} is equal to
\begin{align}\label{eq:RightSideLaurent}
    &\sum_{\substack{M\subseteq\sigma_L\\\abs{M}\ge a}}\RC({\kappa_M},a)\sum_{\substack{(k_{(j,\eta)})_{0\le j\le r,\eta\in\sigma_j'(\kappa_M)}\\\sum_{j,\eta}k_{(j,\eta)}=\abs{\sigma_L}-a}}\prod_{j=0}^r\prod_{\eta\in\sigma'_j(\kappa_M)}\frac{\tilde H_{L'(\kappa_M)_\eta,j,\eta}^{(k_{(j,\eta)})}(w)}{k_{(j,\eta)}!}.
  \end{align}
  We may switch the order of summation in \eqref{eq:RightSideLaurent}, using the natural bijection between the parts of $\sigma$ and $\sigma_{j_i}'(\kappa_M).$ Note that this bijection identifies the parts of $M\subseteq\sigma_{i_1}$ with parts of $\sigma_{i_2}'(\kappa_M).$ For $\eta\not\in M$ we have $L'(\kappa_M)_\eta=L_\eta$, so the result is:
   \begin{align}\label{eq:SwitchedSummation}
\sum_{\substack{(k_{(j,\eta)})_{0\le j\le r,\eta\in\sigma_j}\\\sum_{j,\eta}k_{(j,\eta)}=\abs{\sigma_L}-a}}\sum_{\substack{M\subseteq\sigma_L\\\abs{M}\ge a}}\RC({\kappa_M},a)&\Biggl(\prod_{\substack{0\le j\le r\\ \eta\in\sigma_j\\\eta\not\in M}}\frac{\tilde H_{L_\eta,j,\eta}^{(k_{(j,\eta)})}(w)}{k_{(j,\eta)}!}\Biggr)\prod_{\eta\in M}\frac{\tilde H_{L_\eta-q\eta,i_2,\eta}^{(k_{(j,\eta)})}(w)}{k_{(i_2,\eta)}!}\\
=\sum_{\substack{(k_{(j,\eta)})_{0\le j\le r,\eta\in\sigma_j}\\\sum_{j,\eta}k_{(j,\eta)}=\abs{\sigma_L}-a}}\sum_{\substack{M\subseteq\sigma_L\\\abs{M}\ge a}}(-1)^{\abs{M}-a}&\binom{\sigma_{i_1}}{M}\binom{\abs{M}-1}{a-1}\Biggl(\prod_{\substack{0\le j\le r\\ \eta\in\sigma_j\\\eta\not\in M}}\frac{\tilde H_{L_\eta,j,\eta}^{(k_{(j,\eta)})}(w)}{k_{(j,\eta)}!}\Biggr)\nonumber\\
&\cdot\prod_{\eta\in M}\frac{\tilde H_{L_\eta-q\eta,i_2,\eta}^{(k_{(j,\eta)})}(w)}{k_{(i_2,\eta)}!\cdot q\cdot\prod_{\substack{0\le i\le r\\1\le\gamma\le q\eta\\(\gamma,i)\ne(q\eta,i_2)}}\left(\frac{q\eta-\gamma}{q\eta}\alpha_{i_1}+\frac{\gamma}{q\eta}\alpha_{i_2}-\alpha_i\right)}.\nonumber
  \end{align}

  Consider a single summand $S(k_{(j,\eta)})$ of the leftmost sum in \eqref{eq:SwitchedSummation}. Fix a subset $$U\subseteq\sigma_L\cap\{(j,\eta):k_{(j,\eta)}>0\},$$ and consider the contribution $S((k_{(j,\eta)});U)$ to $S(k_{(j,\eta)})$ from all $M$ such that $$M\cap\{(j,\eta):k_{(j,\eta)}>0\}=U.$$ By definition, we have $S(k_{(j,\eta)})=\sum_{U\subseteq\sigma_L\cap\{(j,\eta):k_{(j,\eta)}>0\}}S((k_{(j,\eta)});U).$ Explicitly, 
  \begin{align}\label{eq:UNonempty}
S((k_{(j,\eta)});U)&=\sum_{\substack{U\subseteq M\subseteq\sigma_L\\\abs{M}\ge a}}(-1)^{\abs{M}-a}\binom{\sigma_{i_1}}{M}\binom{\abs{M}-1}{a-1}\Biggl(\prod_{\substack{0\le j\le r\\ \eta\in\sigma_j\\\eta\not\in M}}\frac{\tilde H_{L_\eta,j,\eta}^{(k_{(j,\eta)})}(w)}{k_{(j,\eta)}!}\Biggr)\nonumber\\
&\quad\cdot\Biggl(\prod_{\eta\in U}\frac{\tilde H_{L_\eta-q\eta,i_2,\eta}^{(k_{(j,\eta)})}(w)}{k_{(j,\eta)}!\cdot q\cdot\prod_{\substack{0\le i\le r\\1\le\gamma\le q\eta\\(\gamma,i)\ne(q\eta,i_2)}}\left(\frac{q\eta-\gamma}{q\eta}\alpha_{i_1}+\frac{\gamma}{q\eta}\alpha_{i_2}-\alpha_i\right)}\Biggr)\\
&\quad\quad\cdot\Biggl(\prod_{\eta\in M\setminus U}\frac{\tilde H_{L_\eta-q\eta,i_2,\eta}(w)}{k_{(j,\eta)}!\cdot q\cdot\prod_{\substack{0\le i\le r\\1\le\gamma\le q\eta\\(\gamma,i)\ne(q\eta,i_2)}}\left(\frac{q\eta-\gamma}{q\eta}\alpha_{i_1}+\frac{\gamma}{q\eta}\alpha_{i_2}-\alpha_i\right)}\Biggr).\nonumber
  \end{align}
  From \eqref{HL}, a simple direct computation using the two identities
    \begin{align}\label{eq:ExponentialFactorsGoAway}
    L_\eta+\mu_{i_1}\frac{\alpha_{i_1}-\alpha_i}{-w}+\mu_{i_2}\frac{\alpha_{i_2}-\alpha_i}{-w}
                             &=(L_\eta-q\eta)+(\mu_{i_1}-q\eta)\frac{\alpha_{i_1}-\alpha_i}{-w}+(\mu_{i_2}+q\eta)\frac{\alpha_{i_2}-\alpha_i}{-w}.
  \end{align}
and 
  \begin{align}\label{eq:identifyfactors}
      \alpha_{i_2}-\alpha_i-\frac{\gamma}{\eta}w=\alpha_{i_1}-\alpha_i-\frac{\gamma+q\eta}{\eta}w.
  \end{align}
  shows that 
\begin{align}\label{eq:KeyEquivalence}
\prod_{\eta\in M\setminus U}\frac{\tilde H_{L_\eta-q\eta,i_2,\eta}(w)}{k_{(j,\eta)}!\cdot q\cdot\prod_{\substack{0\le i\le r\\1\le\gamma\le q\eta\\(\gamma,i)\ne(q\eta,i_2)}}\left(\frac{q\eta-\gamma}{q\eta}\alpha_{i_1}+\frac{\gamma}{q\eta}\alpha_{i_2}-\alpha_i\right)}=\prod_{\eta\in M\setminus U}\frac{\tilde H_{L_\eta,i_1,\eta}(w)}{k_{(i_1,\eta)}!}.
\end{align}
(Specifically, \eqref{eq:ExponentialFactorsGoAway} is used to show that the product of exponential factors appearing on both sides of \eqref{eq:KeyEquivalence} are identical, and \eqref{eq:identifyfactors} is used to show that for each $\eta,$ the corresponding products of factors $(\alpha_j-\alpha_i-\frac{\gamma}{\eta}z)$ on each side of \eqref{eq:KeyEquivalence} are identical. The factor $\frac{1}{q}$ matches with $\frac{(w-z)}{\alpha_{i_1}-\alpha_i-\frac{\gamma}{\eta}z}$, where $(\gamma,i)=(q\eta,i_2)$, on the right side of \eqref{eq:KeyEquivalence}.)

By \eqref{eq:KeyEquivalence}, the product expressions in \eqref{eq:UNonempty} are independent of $M$, i.e. we may rewrite \eqref{eq:UNonempty}:
\begin{align}\label{eq:IndepOfM}
S((k_{(j,\eta)});U)&=\Biggl(\prod_{\substack{0\le j\le r\\ \eta\in\sigma_j\\\eta\not\in U}}\frac{\tilde H_{L_\eta,j,\eta}^{(k_{(j,\eta)})}(w)}{k_{(j,\eta)}!}\Biggr)\Biggr(\prod_{\eta\in U}\frac{\tilde H_{L_\eta-q\eta,i_2,\eta}^{(k_{(j,\eta)})}(w)}{k_{(j,\eta)}!\cdot q\cdot\prod_{\substack{0\le i\le r\\1\le\gamma\le q\eta\\(\gamma,i)\ne(q\eta,i_2)}}\left(\frac{q\eta-\gamma}{q\eta}\alpha_{i_1}+\frac{\gamma}{q\eta}\alpha_{i_2}-\alpha_i\right)}\Biggr)\\
&\quad\quad\cdot\Biggl(\sum_{\substack{U\subseteq M\subseteq\sigma_L\\\abs{M}\ge a}}(-1)^{\abs{M}-a}\binom{\sigma_{i_1}}{M}\binom{\abs{M}-1}{a-1}\Biggr)\nonumber.
\end{align}
The last sum in \eqref{eq:IndepOfM} is equal to $$\sum_{m=0}^{\abs{\sigma_L}-a}(-1)^{m}\binom{m+a-1}{a-1}\binom{\abs{\sigma_L}-\abs{U}}{m+a-\abs{U}}=\sum_{m=0}^{\abs{\sigma_L}-a}\binom{-a}{m}\binom{\abs{\sigma_L}-\abs{U}}{\abs{\sigma_L}-a-m}=\binom{\abs{\sigma_L}-a-\abs{U}}{\abs{\sigma_L}-a},$$ where we have used the Chu-Vandermonde identity (as well as the usual conventions for binomial coefficients with negative first argument). Thus $S((k_{(j,\eta)});U)=0$ for $U\ne\emptyset$, i.e. $$S(k_{(j,\eta)})=S((k_{(j,\eta)});\emptyset)=\prod_{\substack{0\le j\le r\\ \eta\in\sigma_j}}\frac{\tilde H_{L_\eta,j,\eta}^{(k_{(j,\eta)})}(w)}{k_{(j,\eta)}!}.$$ This is precisely to say that \eqref{eq:SwitchedSummation} and \eqref{eq:LeftSideLaurent} agree, proving \ref{item:Recursion} in the case when $\sigma$ is nontrivial.

Lastly, we treat the case where $\sigma$ is the trivial multipartition of $\mu$, so $I_{(\mu,\sigma)}$ has a factor $z+x_{\sigma}.$ We must therefore also prove the analog of \eqref{eq:SingleLabel} for $zT_{\sigma,L}(z).$ Very little modification is required. In fact, our argument never mentioned the exact form of $H_{\beta,j,\eta}(z)$ for $\eta\not\in\sigma_L$ --- using this, it is easy to see that \emph{any} multiple $g(z)T_{\sigma,L}(z)$ satisfies \eqref{eq:SingleLabel}. This completes the proof of \ref{item:Recursion}.

  \medskip

  Finally, we prove \ref{item:TwistedCone} using Tseng's orbifold
  quantum Riemann-Roch (OQRR) operator. It is sufficient to prove the
  statement for the specializations $I_\mu(Q,0,\mathbf{x},-z)$, since
  (1) $Q$ may be rescaled to absorb $e^{t_i\beta},$ and (2) the string
  equation shows that $\mathcal{L}_\mu^{\tw}$ is invariant under
  multiplication by $e^{-\sum_{j=0}^rt_i\mu_j(\alpha_j-\alpha_i)/z}.$

  The OQRR operator is expressed in terms of the
  tangent bundle $F=T_\mu\Sym^d\P^r$ over
  $\mu=(\mu_0,\ldots,\mu_r)$. Note that $F$ splits into subbundles
  $F_{j,i},$ where $0\le i,j\le r$ and $i\ne j;$ here $F_{j,i}$
  consists of tangent vectors along which the $\mu_j$ points at the
  coordinate point $P_j\in\P^r$ move along the coordinate line
  $L_{(j,i)}$. 
  Note that
  $(t_0,\ldots,t_r)\in(\C^*)^{r+1}$ acts on $F_{j,i}$ by
  multiplication by $t_i/t_j.$ The isotropy group $G_\mu$ acts on
  $F_{j,i}$, and may prevent it from decomposing further.

  For each multipartition $\sigma=(\sigma_0,\ldots,\sigma_r)$ of
  $\mu$, let $F_{j,i,\sigma}$ denote the pullback of $F_{j,i}$ along
  $\sigma\into I\mu\to\mu.$\footnote{In fact $F_{j,i,\sigma}$ splits further, with a subbundle for each distinct integer appearing in $\sigma_j.$ We will not need this splitting directly, but it is related to the choice of variables in the proof of Claim \ref{cl:OnUntwisted} below.} In \cite{Tseng2010}, one must describe,
  for each $q\in\Q,$ the eigenbundle $F_{j,i,\sigma}^q$ of
  $F_{j,i,\sigma}$ on which a representative
  $\alpha=(\alpha_0,\ldots,\alpha_r)\in G_\mu\cong
  S_{\mu_0}\times\cdots\times S_{\mu_r}$ of $\sigma$ acts with
  eigenvalue $q\abs{G_\mu}$. By definition of $F_{j,i},$ $\alpha$ acts
  by a permutation matrix associated to the cycle type of $\alpha_j;$
  the eigenvalues are therefore
  $\bigsqcup_{\eta\in\sigma_j}\{1,e^{2\pi i/\eta}\ldots,e^{2\pi
    i(\eta-1)/\eta}\}.$ Equivalently,
  \begin{align}\label{eq:EigenvalueCalculation}
  \ch_k(F_{j,i,\sigma}^{(q\lcm(\sigma_j))})&=
                              \begin{cases}
                                \#\{\eta\in\sigma_j:q\in\frac{1}{\eta}\Z\}&k=0\\
                                0&k>0,
                              \end{cases}
  \end{align}

  Define a collection of formal variables $\mathbf{s}=(s_k^{(j,i)})$ for $0\le i,j\le r,$ $i\ne j$, and $k\ge0$. These define a family of multiplicative
  characteristic classes
  $$\mathbf{c}_{\mathbf{s}}(F)=\prod_{\substack{0\le i\le r\\i\ne j}}\exp\left(\sum_{k\ge0}s_k^{(j,i)}\ch_k(F_{j,i})\right),$$
  with the specialization
  \begin{align}\label{eq:sjik}
    s^{(j,i)}_k=
    \begin{cases}
      -\log(\alpha_j-\alpha_i)&k=0\\
      (-1)^{k}(k-1)!(\alpha_j-\alpha_i)^{-k}&k\ge1
    \end{cases}
  \end{align}
  giving the equivariant Euler class
  $\mathbf{c}_{\mathbf{s}}(F)=e_T(F)$ (see
  \cite{ChiodoRuan2010}, Lemma 4.1.2). Under this specialization,
  $\mathbf{s}^{(j,i)}(x):=\sum_{k\ge0}s_k^{(j,i)}\frac{x^k}{k!}$
  satisfies $\exp(\mathbf{s}^{(j,i)}(x))=(\alpha_j-\alpha_i+x)^{-1}.$

Let $B_m$ denote the $m$th Bernoulli polynomial; recall that $B_m(0)$ is the $m$-th Bernoulli number. The OQRR operator for
  $F=\bigoplus_{j\ge0}^r\bigoplus_{\substack{0\le i\le r\\i\ne
      j}}F_{j,i}$ is
  \begin{align*}
    \Delta&=\bigoplus_{\sigma\in\MultiPart(\mu)}\exp\left(\sum_{j=0}^r\sum_{\substack{0\le
    i\le r\\i\ne
    j}}\sum_{\substack{q\in\Q\cap[0,1)}}\sum_{k\ge0}\sum_{m\ge0}s^{(j,i)}_k\frac{B_m(q)}{m!}\ch_{k+1-m}(F_{j,i,\sigma}^{(q)})z^{m-1}\right)\\
    &=\bigoplus_{\sigma\in\MultiPart(\mu)}\exp\left(\sum_{\substack{0\le i,j\le r\\i\ne j}}\sum_{\eta\in\sigma_j}\sum_{\ell=0}^{\eta-1}\sum_{m\ge1}s^{(j,i)}_{m-1}\frac{B_m(\ell/\eta)}{m!}z^{m-1}\right)\\
    &=\bigoplus_{\sigma\in\MultiPart(\mu)}\exp\left(\sum_{\substack{0\le i,j\le r\\i\ne j}}\sum_{\eta\in\sigma_j}\sum_{m\ge1}s^{(j,i)}_{m-1}\frac{B_m(0)}{m!}(z/\eta)^{m-1}\right),
  \end{align*}
  where the second equality is by \eqref{eq:EigenvalueCalculation} and the third equality is from
  the following identity, easily proved via generating functions of Bernoulli polynomials:
  $$\sum_{0\le\ell\le\eta-1}B_m(\ell/\eta)=\frac{B_m(0)}{\eta^{m-1}}.$$  
  Let
  \begin{align}\label{eq:GDef}
    G^{(j,i)}(x,z)&=\sum_{n,m\ge0}s^{(j,i)}_{n+m-1}\frac{B_m(0)}{m!}\frac{x^n}{n!}z^{m-1},
  \end{align}
  so that
  \begin{align*}
    \Delta&=\bigoplus_{\sigma\in\MultiPart(\mu)}\exp\left(\sum_{\substack{0\le
    i,j\le r\\i\ne j}}\sum_{\eta\in\sigma_j}G^{(j,i)}(0,z/\eta)\right).
  \end{align*}

  By definition of the Bernoulli polynomials, the coefficient of
  $s_k^{(j,i)}$ in $G^{(j,i)}(x,z)$ is the degree $k$ part of
  $\frac{e^x}{e^z-1}.$ For $a\in\C,$ the equation
  \begin{align*}
    \frac{e^{x+az}}{e^{az}-1}=\frac{e^x}{e^{az}-1}+e^x
  \end{align*}
  implies the functional equation
  $G^{(j,i)}(x+az,az)-G^{(j,i)}(x,az)=\mathbf{s}^{(j,i)}(x)$. Applying
  this repeatedly shows that (after the specialization
  \eqref{eq:sjik}) we have
  $I_{\mu}(Q,0,\mathbf{x},-z)=\Delta\left(I^{\untw}_{\mu}(Q,\mathbf{x},\mathbf{s},-z)\right),$ where
  \begin{align*}
    I^{\untw}_{\mu}(Q,\mathbf{x},\mathbf{s},-z)
    &=\sum_{\sigma\in\MultiPart(\mu)}\left(z\delta_{\sigma,(1,\ldots,1)}1_\sigma+x_{\pi(\sigma)}1_\sigma\right)\prod_{j=0}^r\prod_{\eta\in\sigma_j}\sum_{\beta\ge0}\frac{Q^\beta\eta^\beta}{\beta!(-z)^{\beta}}
                                      \exp\left(-\sum_{\substack{0\le
                                      i\le r\\i\ne
    j}}G^{(j,i)}(-\beta z/\eta,z/\eta)\right).
  \end{align*}
  Let $\nu(j)=(0,\ldots,0,1,0,\ldots,0)\in\ZPart(1,r+1)$, where the 1 is in the $j$-th position, and let $\rho(j)$ be the unique element of $\MultiPart(\nu(j)).$ Note the relationship: 
  \begin{align}\label{eq:RecursiveExpressionForI}
  I^{\untw}_{\mu}(Q,\mathbf{x},\mathbf{s},-z)=\sum_{\sigma\in\MultiPart(\mu)}\left(z\delta_{\sigma,(1,\ldots,1)}1_\sigma+x_{\pi(\sigma)}1_\sigma\right)\prod_{j=0}^r\prod_{\eta\in\sigma_j}\frac{\eta}{z}I_{\nu(j)}^{\untw}(Q,\mathbf{s},-z/\eta).
  \end{align}
  It is now sufficient to prove:
  \begin{cl}\label{cl:OnUntwisted}
  Assuming Identities \ref{lem:WhatIsTauIdentity} and \ref{lem:WhatIsTauSigma}, $I^{\untw}_{\mu}(Q,\mathbf{x},\mathbf{s},-z)$ is a $\Lambda_{\nov}[[\mathbf x,\mathbf{s}]]/(\mathbf{x})^2$-valued point of the (untwisted) Lagrangian cone $\mathcal{L}_\mu.$
  \end{cl}
  If we prove Claim \ref{cl:OnUntwisted}, it will imply that $I_{\mu}(Q,0,\mathbf{x},-z)$ is a $\Lambda_{\nov}^T[[\mathbf{x}]]$-valued point of the $\mathbf{s}$-twisted Lagrangian cone of $BG_\mu$ for $\mathbf{s}$ as in \eqref{eq:sjik} --- which is precisely $\mathcal L_\mu^{\tw}$.
  \begin{proof}[Proof of claim \ref{cl:OnUntwisted}]
We prove a slightly stronger statement, replacing the variables $x_{\pi(\sigma)}$ in $I_\mu^{\untw}$ with new variables $x_\sigma$ in the obvious way. Define $\deg(s_k^{(j,i)})=k+1$. We will prove that $I_\mu^{\untw}$ is on $\mathcal{L}_\mu$ by induction on the degree. For the base case, \cite[Proposition 3.4]{JarvisKimura2002} shows that the $J$-function $J(Q,\mathbf{x},-z)\in\mathcal{L}_\mu$ is given by 
    \begin{align}\label{eq:JFunction}
    J(Q,\mathbf{x},-z)&=-z\exp(Q/(-z))\left(1+\sum_{\sigma\in\MultiPart(\mu)}\frac{x_\sigma}{-z}1_\sigma\right)\quad\quad\mod(\mathbf{x})^2.
    \end{align}
    That is, $I^{\untw}_{\mu}(Q,\mathbf{x},0,-z)=J(dQ,\mathbf{x},-z)\in\mathcal{L}_\mu.$
    
    For the inductive step, suppose that $I_\mu^{\untw}$ lies on $\mathcal{L}_\mu$ up to degree $M$ in the variables $s^{(j,i)}$. We will show that $I_\mu^{\untw}$ lies on $\mathcal{L}_\mu$ up to degree $M+1.$ It is sufficient to show that all derivatives $\pderiv{I_\mu^{\untw}}{s_k^{(j,i)}}$ lie in the tangent space $T_{I_\mu^{\untw}}\mathcal{L}_{\mu}$ up to degree $M$ \cite[p.393]{CoatesCortiIritaniTseng2009}. Let $-z+\mathbf{t}_\mu$ be the part of $I_\mu^{\untw}(Q,\mathbf{x},\mathbf{s},-z)$ with nonnegative $z$-exponents, and for $\sigma\in\MultiPart(\mu),$ define 
    \begin{align}\label{eq:TauDef}
    \tau^\sigma(\mathbf{t}_\mu)=\sum_{n\ge1}\frac{1}{n!}\left\langle1,1_\sigma,\mathbf{t}_\mu,\ldots,\mathbf{t}_\mu\right\rangle_{0,n+2}^{\mu}.
    \end{align}
     Then by \cite[Prop. B.4]{CoatesCortiIritaniTseng2009}, $T_{I_\mu^{\untw}}\mathcal{L}_{\mu}$ is freely generated as a $\Lambda_{\nov}[[\mathbf x,\mathbf{s},z]]/(\mathbf{x})^2$-module by the derivatives $\partial_\sigma J(\tau,-z)|_{\tau=\tau(\mathbf{t}_\mu)}$ with respect to the variables $Q$ and $x_\sigma$. From \eqref{eq:JFunction}, we have:
    \begin{align*}
    \partial_{(1,\ldots,1)}J(\tau,-z)|_{\tau=\tau(\mathbf{t}_\mu)}&=\exp(\tau^{(1,\ldots,1)}(\mathbf{t}_\mu)/(-z))\left(1+\sum_{\sigma\in\MultiPart(\mu)}\frac{\tau^{\sigma}(\mathbf{t}_\mu)}{-z}1_\sigma\right)\\
    \partial_{\sigma}J(\tau,-z)|_{\tau=\tau(\mathbf{t}_\mu)}&=\exp(\tau^{(1,\ldots,1)}(\mathbf{t}_\mu)/(-z))1_\sigma\quad\text{ for $\sigma\ne(1,\ldots,1)$.}\nonumber
    \end{align*}
We must therefore show that $\frac{1}{\exp(\tau^{(1,\ldots,1)}(\mathbf{t}_\mu)/(-z))}\pderiv{I_\mu^{\untw}}{s_{k_0}^{(i_0,j_0)}}$ is in the $\Lambda_{\nov}[[\mathbf x,\mathbf{s},z]]/(\mathbf{x})^2$-module generated by:
\begin{align}\label{eq:JFunctionDerivsDivide}
&1+\sum_{\sigma\in\MultiPart(\mu)}\frac{\tau^{\sigma}(\mathbf{t}_\mu)}{-z}1_\sigma&&\text{and}&&
    1_\sigma\quad\text{ for $\sigma\ne(1,\ldots,1)$},
    \end{align}
    for any $0\le i_0,j_0\le r$ with $i_0\ne j_0$ and any $k_0\ge0$.
    For convenience, define 
    \begin{align}\label{eq:FandG}
    f_{\mu,k_0}^{(i_0,j_0)}(-z):&=\frac{1}{\exp(\tau^{(1,\ldots,1)}(\mathbf{t}_\mu)/(-z))}\pderiv{I_\mu^{\untw}}{s_{k_0}^{(i_0,j_0)}}&&\text{and}&g_{\mu}(-z):&=\frac{I_\mu^{\untw}/z}{\exp(\tau^{(1,\ldots,1)}(\mathbf{t}_\mu)/(-z))}.
    \end{align}
    By \eqref{eq:RecursiveExpressionForI} and the product rule, we have
    \begin{align}\label{eq:ProductRule}
    \pderiv{I_\mu^{\untw}}{s_{k_0}^{(i_0,j_0)}}&=\sum_{\sigma\in\MultiPart(\mu)}\left(z\delta_{\sigma,(1,\ldots,1)}1_\sigma+x_{\sigma}1_\sigma\right)\sum_{\substack{0\le j_1\le r\\\eta_1\in\sigma_{j_1}}}\frac{\eta_1}{z}\pderiv{I_{\nu(j_1)}^{\untw}(Q,\mathbf{s},-z/\eta_1)}{s_{k_0}^{(i_0,j_0)}}\prod_{\substack{0\le j\le r\\\eta\in\sigma_j\\(j,\eta)\ne(j_1,\eta_1)}}\frac{\eta}{z}I_{\nu(j)}^{\untw}(Q,\mathbf{s},-z/\eta).
    \end{align}
    Observe that Identity \ref{lem:WhatIsTauIdentity} in Section \ref{sec:Appendix} implies
    \begin{align*}
    \tau^{(1,\ldots,1)}(\mathbf{t}_\mu)=\sum_{0\le j\le r}\mu_j\tau^{\rho(j)}(\mathbf{t}_{\nu(j)}),
    \end{align*}
     where $\nu(j)$ and $\rho(j)$ are as in \eqref{eq:RecursiveExpressionForI}.
    Thus \eqref{eq:ProductRule} implies:
    \begin{align}\label{eq:AlmostPolynomial}
    f_{\mu,k_0}^{(i_0,j_0)}(-z)&=\sum_{\sigma\in\MultiPart(\mu)}\left(z\delta_{\sigma,(1,\ldots,1)}1_\sigma+x_{\sigma}1_\sigma\right)\sum_{\substack{0\le j_1\le r\\\eta_1\in\sigma_{j_1}}}\frac{\eta_1}{z}f_{\nu(j_1),k_0}^{(i_0,j_0)}(-z/\eta_1)\prod_{\substack{0\le j\le r\\\eta\in\sigma_j\\(j,\eta)\ne(j_1,\eta_1)}}g_{\nu(j)}(-z/\eta)\\
    &=\sum_{\sigma\in\MultiPart(\mu)}\left(z\delta_{\sigma,(1,\ldots,1)}1_\sigma+x_{\sigma}1_\sigma\right)\left(\prod_{\substack{0\le j\le r\\\eta\in\sigma_j}}g_{\nu(j)}(-z/\eta)\right)\sum_{\eta\in\sigma_{j_0}}\frac{\eta}{z}\frac{f_{\nu(j_0),k_0}^{(i_0,j_0)}(-z/\eta)}{g_{\nu(j_0)}(-z/\eta)}.\nonumber
    \end{align}
    (In the second equality, we have used the fact that $f_{\nu(j_1),k_0}^{(i_0,j_0)}(-z/\eta)=0$ if $j_0\ne j_1$, which is immediate from the definition of $I_\mu^{\untw}$.)
    
    By the mirror theorem for $\Sym^1\P^r=\P^r$ (specifically, the proof on \cite[p.31]{CoatesCortiIritaniTseng2015}), we have $I_{\nu(j)}^{\untw}\in\mathcal{L}_{\nu(j)}$. Thus $\frac{1}{z}I_{\nu(j)}^{\untw}\in T_{I_{\nu(j)}^{\untw}}\mathcal{L}_{\mu'},$ by the tangent space property of $\mathcal{L}_{\nu(j)}$ \cite[Thm. 1]{Givental2004}. In particular, $\frac{1}{z}I_{\nu(j)}^{\untw}(Q,\mathbf{s},-z/\eta)$ is divisible by $\exp(-\eta\tau^{\rho(j)}(\mathbf{t}_{\nu(j)})/z),$ where by ``divisible'', we mean that the ratio contains only nonnegative powers of $z$; in other words, $g_{\nu(j)}(-z/\eta)$ contains only nonnegative powers of $z$. Similarly, because $I_{\nu(j)}^{\untw}\in\mathcal{L}_{\nu(j)},$ we have that $f_{\nu(j_1),k_0}^{(i_0,j_0)}(-z/\eta_1)$ contains only nonnegative powers of $z$. That is, the first line of \eqref{eq:AlmostPolynomial} implies $f_{\mu,k_0}^{(i_0,j_0)}(-z)$ is of the form: $$1\cdot(\text{power series in $z$})+\sum_{\sigma\in\MultiPart(\mu)}\frac{x_\sigma 1_\sigma}{-z}\cdot(\text{power series in $z$})+O(\mathbf{x})^2.$$
    
    In order to show $f_{\mu,k_0}^{(i_0,j_0)}(-z)$ is in the $\Lambda_{\nov}[[\mathbf x,\mathbf{s},z]]/(\mathbf{x})^2$-module generated by \eqref{eq:JFunctionDerivsDivide}, it remains to show that for all $\sigma\in\MultiPart(\mu)$ with $\sigma\ne(1,\ldots,1),$ the coefficient of $z^{-1}\cdot1_\sigma$ in \eqref{eq:AlmostPolynomial} is equal to $-\tau^{\sigma}(\mathbf{t}_\mu)$ times the coefficient of $z^{0}\cdot1$ in \eqref{eq:AlmostPolynomial}. In other words, we must show: $$x_\sigma\left(\prod_{\substack{0\le j\le r\\\eta\in\sigma_j}}g_{\nu(j)}(0)\right)\sum_{\eta\in\sigma_{j_0}}\eta\frac{f_{\nu(j_0),k_0}^{(i_0,j_0)}(0)}{g_{\nu(j_0)}(0)}=\tau^{\sigma}(\mathbf{t}_\mu)\left(\prod_{\substack{0\le j\le r\\1\le a\le\mu_j}}g_{\nu(j)}(0)\right)\sum_{1\le a_1\le\mu_{j_0}}\frac{f_{\nu(j_0),k_0}^{(i_0,j_0)}(0)}{g_{\nu(j_0)}(0)}.$$ After cancellation, this is precisely the statement of Identity \ref{lem:WhatIsTauSigma}.
    \end{proof}
    Claim \ref{cl:OnUntwisted} completes the proof of \ref{item:TwistedCone}, and hence of Theorem \ref{thm:MirrorTheorem}.
  \end{proof}
  
  From \eqref{eq:IDef}, we compute that (assuming $r>0$) we have $$I_{\Sym^d\P^r}(Q,t,\mathbf{x},z)=1\cdot z+\sum_{i=0}^rt_iH_i+\sum_{\sigma\in\MultiPart(\mu)}x_\sigma1_\sigma+O(z^{-1}),$$ where $[H_i]$ is as in Section \ref{sec:SymmetricPowers}. By definition of $J_{\Sym^d\P^r}$ (from Section \ref{sec:OrbifoldGWTheory}), Theorem \ref{thm:MirrorTheorem} implies:
\begin{cor}\label{cor:IandJ}
  Assuming Identities \ref{lem:WhatIsTauIdentity} and \ref{lem:WhatIsTauSigma}, we have $$I_{\Sym^d\P^r}(Q,t,\mathbf{x},z)=J_{\Sym^d\P^r}(Q,\theta,z)\quad\mod(\mathbf{x})^2,$$ where $\theta=\sum_{i=0}^rt_iH_i+\sum_{\sigma\in\MultiPart(\mu)}x_\sigma1_\sigma.$
\end{cor}

      \section{Appendix: Conjectural Combinatorial Identities}\label{sec:Appendix}
      The Mirror Theorem \ref{thm:MirrorTheorem} is conditional upon the following two conjectural combinatorial identities. Let $\mathbf{t}_\mu$ denote the part of $I^{\untw}_{\mu}$ with nonnegative powers of $z$, and let $\tau^\sigma(\mathbf{t}_\mu)$ be as in \eqref{eq:TauDef}. We conjecture the following:
\begin{identity}\label{lem:WhatIsTauIdentity}
  For all $\mu\ge1,$ we have $$\tau^{(1,\ldots,1)}(\mathbf{t}_\mu)=\sum_{0\le j\le r}\mu_j\sum_{n_0,n_1,\ldots\ge0}\frac{(1+\sum_{k\ge0}kn_k)^{-2+\sum_{k\ge0}n_k}}{\prod_{k\ge0}n_k!(k!)^{n_k}}Q^{1+\sum_{k\ge0}kn_k}\prod_{k\ge0}\left(\sum_{\substack{0\le i\le r\\i\ne j}}s_k^{(j,i)}\right)^{n_k}+O(\mathbf{x}^2).$$
  \end{identity}
   \begin{identity}\label{lem:WhatIsTauSigma}
  Let $\mu\ge 1$, and let $\sigma$ be a partition of $\mu$ that is \emph{not} equal to $(1,\ldots,1).$ Then
  \begin{align*}
  \tau^{\sigma}(\mathbf{t}_{\mu})&=x_\sigma\prod_{0\le j\le r}g_{\nu(j)}(0)^{\abs{\sigma_j}-\mu_j}+O(\mathbf{x})^2,
  \end{align*}
  where $g_\mu(-z)$ is as in \eqref{eq:FandG}. (Recall that $\nu(j)\in\ZPart(1,r+1)$ is the composition with 1 in the $j$th entry, and $\rho(j)$ is the unique multipartition of $\nu(j).$)
  \end{identity}
  Both identities are entirely combinatorial in nature, as $\tau^{(1,\ldots,1)}(\mathbf{t}_\mu)$, $\tau^{\sigma}(\mathbf{t}_{\mu})$, and $g_{\nu(j)}(0)$ are entirely explicit. Specifically, one uses the formulas \cite[Lemma 1.5.1]{KockNotes} and \cite[Prop. 3.4]{JarvisKimura2002}, both of which we have already used extensively in this paper, to evaluate the integrals appearing in $\tau^{(1,\ldots,1)}(\mathbf{t}_\mu)$ and $\tau^{\sigma}(\mathbf{t}_{\mu})$ in terms of multinomial coefficients. The resulting expressions for $\tau^{(1,\ldots,1)}(\mathbf{t}_\mu)$ and $\tau^{\sigma}(\mathbf{t}_{\mu})$ are iterated sums over partitions.
  
  We expect that both identities can be proved via cleverly switching the order of summation, and applying basic multinomial coefficient identities or generating function techniques. (In the introduction we speculated that tools from integrable systems might also yield a less-hands-on proof.) However, due to the complicatedness of the generating functions involved, we were not able to complete either proof. We instead conclude with some relevant observations and experimental verifications of both identities to small order.

\medskip

    \noindent\textbf{Notes about Identity \ref{lem:WhatIsTauIdentity}:}
    \begin{enumerate}
     \item The variables $x_\sigma$ are entirely absent from Identity \ref{lem:WhatIsTauIdentity}, as follows. The invariants appearing in $\tau^{(1,\ldots,1)}(\mathbf{t}_\mu)$ are of the form $$\left\langle1,1,\mathbf{t}_{\mu},\ldots,\mathbf{t}_{\mu}\right\rangle_{0,n+2}^{\mu}.$$ For $\sigma\ne(1,\ldots,1)$, the class $1_\sigma$ always appears with an $x_\sigma$ factor in $\mathbf{t}_{\mu}$. As we are working modulo $(\mathbf{x})^2,$ the only contributions are from invariants with at most one nontrivial class $1_\sigma.$ By \cite[Prop. 3.4]{JarvisKimura2002}, invariants with exactly one nontrivial class $1_\sigma$ vanish. Thus we may replace $\mathbf{t}_{\mu}$ in Identity \ref{lem:WhatIsTauIdentity} with $\mathbf{t}_\mu^0,$ where $\mathbf{t}_{\mu}^0$ denotes the coefficient of $1$ in $\mathbf{t}_{\mathbf{\mu}}$.\label{TZero}
     
     \item The difficulty in proving Identity \ref{lem:WhatIsTauIdentity} is not due to the presence of Bernoulli numbers in the definition of $I_{\mu}^{\untw}$; in fact the Bernoulli numbers appear to play no role whatsoever, as Identity \ref{lem:WhatIsTauIdentity} is a special case of the following more general formula. Let 
     \begin{align}\label{eq:MoreGeneralF}
     \mathbf{f}=-z\prod_{0\le j\le r}\left(\sum_{\beta\ge0}\frac{Q^\beta}{\beta!(-z)^\beta}\exp\left(\sum_{\substack{k\ge0\\0\le\ell\le k+1}}c_{k,\ell}^{(j)}z^k\beta^\ell\right)\right)^{\mu_j}.
     \end{align} Then experimentally we appear to have
 \begin{align}\label{eq:MoreGeneralTau}
 \tau^{(1,\ldots,1)}(\mathbf{t}_{\mathbf{f}})&=-\sum_{0\le j\le r}\mu_j\sum_{n_0,n_1,\ldots\ge0}\left(1+\sum_{k\ge0}kn_k\right)^{-2+\sum_{k\ge0}n_k}(-Q)^{1+\sum_{k\ge0}kn_k}\prod_{k\ge0}\frac{((k+1)c_{k,k+1}^{(j)})^{n_k}}{n_k!},
 \end{align}
 where $-z+\mathbf{t}_{\mathbf{f}}$ denotes the part of $\mathbf{f}$ with nonnegative powers of $z$. Using Note \eqref{TZero}, Identity \ref{lem:WhatIsTauIdentity} is the special case where $$c_{k,\ell}^{(j)}=\sum_{i\ne j}\frac{B_{k-\ell+1}(0)}{\ell!(k-\ell+1)!}s_{k}^{(j,i)}.$$ (Indeed, the absence of $c_{k,\ell}$ for $\ell\le k$ in \eqref{eq:MoreGeneralTau} shows that the Bernoulli numbers $B_m(0)$ for $m>0$ are entirely irrelevant.) Note that we require $\ell\le k+1$ in \eqref{eq:MoreGeneralF} because, in the expression $G^{(j,i)}(-\beta z/\eta,z/\eta)$ from \eqref{eq:GDef}, the power of $\beta$ is always at most one more than the power of $z$. \label{Bernoulli}

\item Writing $$\mathbf{t}_{\mu}^0=1\cdot(y_0+y_1z+y_2z^2+\cdots),$$ and evaluating the integrals $$\left\langle1,1,\mathbf{t}_{\mu}^0,\ldots,\mathbf{t}_{\mu}^0\right\rangle_{0,n+2}^{\mu},$$ gives the expression
   \begin{align*}
  \tau^{(1,\ldots,1)}(\mathbf{t}_{\mu}^0)&=\sum_{n=1}^\infty\sum_{\zeta\in\ZPart(n-1,n)}\frac{1}{\abs{S_\zeta}}\binom{n-1}{\zeta}\prod_{\eta\in\zeta}y_\eta.
  \end{align*}
We may compute each $y_i$ explicitly. For example, if we set every $c_{k',\ell'}$ to zero \emph{except for one}, say $c_{k,\ell},$ and we have $\mu_j=0$ for all but one $j$, then we have:
 \begin{align*}
  y_\eta=(-1)^\eta\mu_j!\sum_{\substack{N\ge\eta/k}}(c_{k,\ell}^{(j)})^N\frac{Q^{kN-\eta+1}}{N!(kN-\eta+1)!}\sum_{\pi\in\ZPart(kN-\eta+1,\mu_j)}\binom{kN-\eta+1}{\pi}\frac{\left(\sum_{\lambda\in\pi}\lambda^\ell\right)^N}{\abs{S_\pi}}.
  \end{align*}
  \label{Explicit}
  \item Some straightforward combinatorial identities can be used to expand $\tau^{(1,\ldots,1)}(\mathbf{t}_\mu)$ as a polynomial in $\mu$, whose coefficients are sums over partitions. Experimentally, the coefficient of $\mu^m$ ``miraculously'' cancels to give zero whenever $m>1$. (Indeed, proving this ``linearity'' would suffice for the purpose of Theorem \ref{thm:MirrorTheorem}; we do not need the explicit formula.)
  \item Figure \ref{Id7.1} verifies Identity \ref{lem:WhatIsTauIdentity} for $\mu=(1,0,\ldots)$ and $\mu=(3,0,\ldots)$, to second order in $s_k=\sum_{1\le i\le r}s_k^{(i,0)}$ for $k\le 2$ (and to zeroth order in $s_k$ for $k>2$).
  \begin{figure}
  \includegraphics[width=5in]{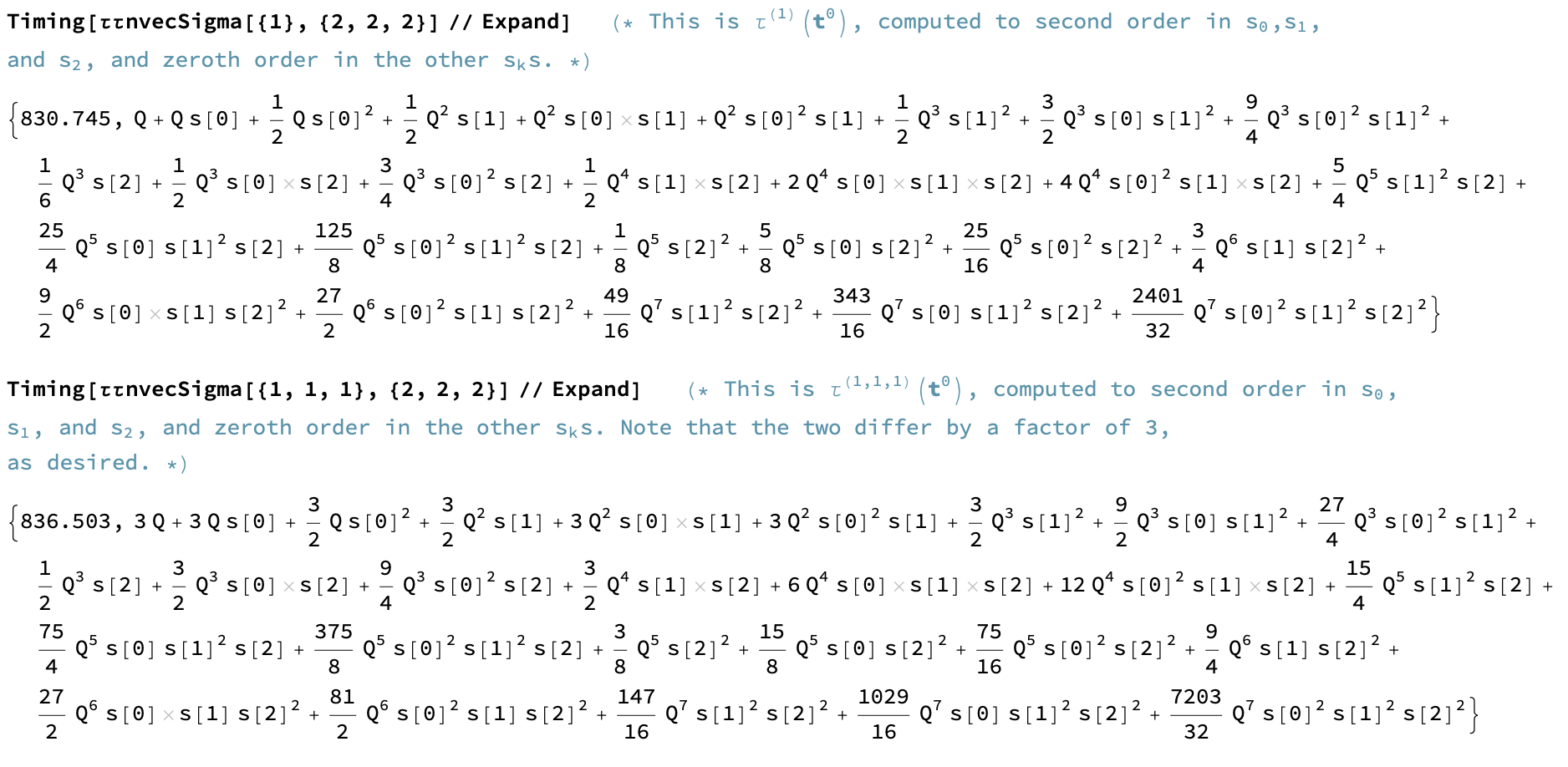}
  \caption{Experimental verification of Identity 7.1}\label{Id7.1}
  \end{figure}
    \end{enumerate}
 
  \medskip
  
  \noindent\textbf{Notes about Identity \ref{lem:WhatIsTauSigma}:}
  \begin{enumerate}\setcounter{enumi}{5}
  \item The same argument as in Note \eqref{TZero} shows that $\tau^\sigma(\mathbf{t}_\mu)\in x_\sigma\cdot\C[[Q,\{s_k^{(j,i)}\}]]+O(\mathbf{x}^2)$. As in Note \eqref{Explicit}, it is straightforward to expand $\tau^\sigma(\mathbf{t}_\mu)$ as an explicit sum over partitions.
  \item Unlike in Note \eqref{Bernoulli}, the Bernoulli numbers do appear to play a nontrivial role in Identity \ref{lem:WhatIsTauSigma}.
  \item Figure \ref{Id7.2} verifies Identity \ref{lem:WhatIsTauSigma} for $\sigma=\{4\},$ $\sigma=\{4,1\}$, and $\sigma=\{3,2\}$, to the same orders in $s_k$ as above. Observe in particular that Identity \ref{lem:WhatIsTauSigma} predicts that $\tau^\sigma(\mathbf{t}_\mu)$ is identical for these three choices of $\sigma$. (By $\sigma=\{4,1\}$, we really mean $\mu=(5,0\ldots,)$ and $\sigma=(\{4,1\},\{\},\ldots).$)
  \begin{figure}
  \includegraphics[width=5in]{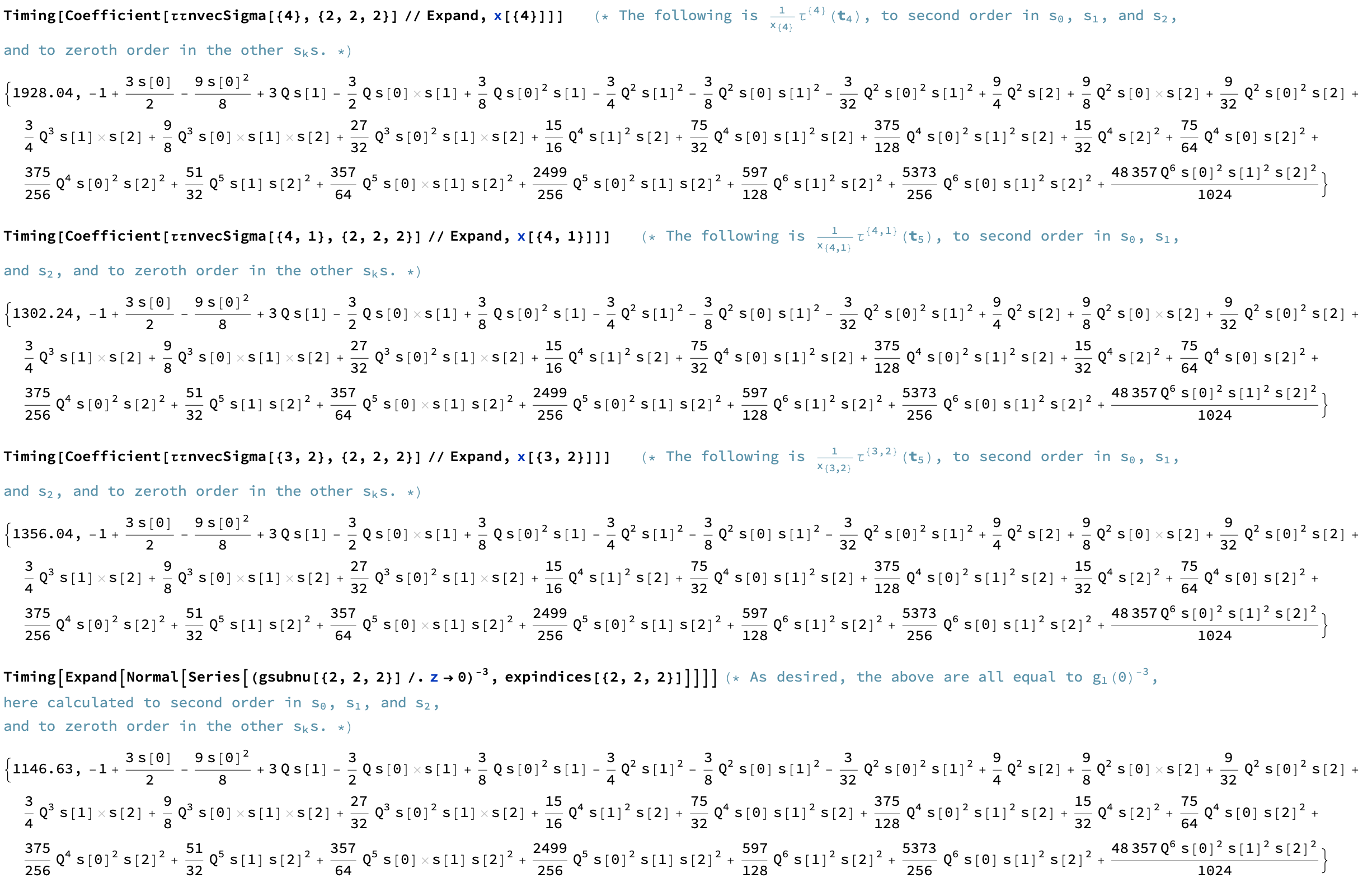}
  \caption{Experimental verification of Identity 7.2}\label{Id7.2}
  \end{figure}
  \end{enumerate}

\bibliographystyle{alpha}
\bibliography{Symrefs.bbl}

\end{document}